\documentclass[a4paper, 10pt]{article}
\usepackage[latin1]{inputenc}
\usepackage[english]{babel}

\usepackage{color}

\usepackage{vmargin}
\setmarginsrb{3cm}{2.5cm}{3cm}{1.5cm}{0cm}{0cm}{0cm}{1.5cm}

\newcommand{\matlab}{MATLAB\textsuperscript{\textregistered}}
\newcommand{\intel}{Intel\textsuperscript{\textregistered}}

\newif\ifmatlab\matlabtrue
\def\matlab{{{\sc matlab }}}

\newcommand\IR{ {\mathbb R}}

 \newtheorem{theorem}{Theorem}[section]
 \newtheorem{lemma}[theorem]{Lemma}
 \newtheorem{proposition}[theorem]{Proposition}
 
 \newtheorem{remark}[theorem]{Remark}

 \newenvironment{proof}[1][Proof]{\begin{trivlist}
 \item[\hskip \labelsep {\bfseries #1}]}{\end{trivlist}}

\usepackage{savesym}
\savesymbol{AND}
\usepackage{algorithmic}
\usepackage{algorithm}
\usepackage{mathtools}
\usepackage{euscript,multicol}
\usepackage{color}
\usepackage{wrapfig,rotating,calc}
\usepackage{tikz}
\usepackage{pgfplots}
\usepackage{bbm}
\usetikzlibrary{calc}

\usepackage{subcaption}

\usepackage{amssymb,amsmath}
\usepackage{graphicx}          

\numberwithin{equation}{section}

\begin{document}

\title{Control strategies for the Fokker-Planck equation}

\author{Tobias Breiten\footnote{Institute of Mathematics and Scientic
Computing,
Karl-Franzens-Universit\"at, Heinrichstr. 36, 8010 Graz, Austria ({\tt
tobias.breiten@uni-graz.at})} \quad
Karl Kunisch\footnote{Institute of Mathematics and Scientic Computing,
Karl-Franzens-Universit\"at, Heinrichstr. 36, 8010 Graz, Austria and Johann
Radon Institute for Computational and Applied
Mathematics~(RICAM), Austrian Academy of Sciences, Altenbergerstra\ss e
69, A-4040 Linz, Austria ({\tt
karl.kunisch@uni-graz.at})} \quad
Laurent Pfeiffer\footnote{Institute of Mathematics and Scientic Computing, Karl-Franzens-Universit\"at, Heinrichstr. 36, 8010 Graz, Austria ({\tt
laurent.pfeiffer@uni-graz.at})}}

\maketitle

\begin{abstract}
Using a projection-based decoupling of the Fokker-Planck equation, control
strategies that allow to speed up the convergence to the stationary
distribution are investigated. By
means of an operator theoretic framework for a bilinear control system,
two different feedback control laws are proposed. Projected Riccati and
Lyapunov equations are derived and properties of the associated solutions are
given. The well-posedness of the closed loop systems is shown and local and
global stabilization results, respectively, are obtained. An essential tool in
the construction of the controls is the choice of appropriate control shape
functions. Results for a two dimensional double well potential illustrate the
theoretical findings in a numerical setup.
\end{abstract}

\textbf{Mathematics Subject Classification.} 35Q35, 49J20, 93D05, 93D15.\newline

\textbf{Keywords}. Fokker-Planck equation, bilinear control systems, Lyapunov functions,
Riccati equation, Lyapunov equation.

\section{Introduction}
\label{sec:intro}

 To partially set the stage, let us consider a very large set of dragged Brownian particles, whose motion is described by a stochastic differential equation (SDE) in $\IR^{2n}$ called the  Langevin equation:
\begin{equation*}
\text{d}x(s) = y(s) \, \text{d} s; \quad
\text{d}y(s) = -\beta y(s) \, \text{d} s + F(x,s) \, \text{d} s + \sqrt{2 \beta k T/m} \, \text{d} B(s).
\end{equation*}
Here $s$ is the time variable, $\beta > 0$ is a friction parameter, $m$ the mass of the particle, $k$ the Boltzmann constant, $T$ is the temperature, and  $B$ is an $n$-dimensional Brownian motion. The force $F$ is assumed to be related  to a potential $V$, so that $F(x,s)= -\nabla V(x,s)$.
For large values of $\beta$, the Langevin equation can be approximated by the Smoluchowski equation:
\begin{equation*}
\text{d} x(t)= - \nabla V(x,t) \text{d} t + \sqrt{2\nu} \, \text{d} B_t,
\end{equation*}
where $t= s/\beta$ and $\nu= kT/m$. The probability density function $\rho$ of the solution to the above equation is the solution to the Fokker-Planck equation:
\begin{equation*}
\frac{\partial \rho}{\partial t} = \nabla \cdot J(x,t), \quad \text{where: } J(x,t)= \nu \nabla \rho + \rho
\nabla V.
\end{equation*}
The variable $J$ is the probability current. The simplification of the Langevin equation is discussed in \cite[Section 10.4]{Ris96}, see also \cite[Section 4.3.4]{Gar04} for details on the connection between SDEs and the Fokker-Planck equation. In this article, we consider the Fokker-Planck equation with reflective boundary conditions:
\begin{equation}\label{eq:FPE}
\begin{aligned}
  \frac{\partial \rho}{\partial t} &= \nu \Delta \rho + \nabla \cdot (\rho
\nabla V) && \text{in } \Omega \times (0,\infty), \\
0&=(\nu \nabla \rho + \rho \nabla V) \cdot \vec{n}  && \text{on } \Gamma \times
(0,\infty),\\
\rho(x,0)&=\rho_0(x)  && \text{in } \Omega,
\end{aligned}
\end{equation}
where $\Omega \subset \mathbb R^n$
denotes a bounded domain  with smooth boundary
$\Gamma = \partial\Omega$, and
$\rho_0$ denotes an initial probability distribution with
$\int_\Omega \rho_0(x) \mathrm{d}x=1.$
The boundary condition states that the probability current has to vanish in the
normal direction on the boundary. This models the fact that any particle
reaching the boundary $\Gamma$ is reflected \cite[Section 5.2.3]{Gar04}. We
refer to \cite{LioS84} for a description of reflected SDEs.


The force $F$ can be an electric force, created by focusing a laser beam. The obtained structure is called optical tweezer and enables to manipulate microscopic particles, see \cite{JonMV15}. We refer to \cite{GorBL12} for an overview of feedback control problems in optical trapping.
Following the discussion in \cite{HarST13}, let us assume that we can interact
with the particle by means of an optical tweezer such that the potential $V$ is
of the form
\begin{align}\label{eq:FPE_potential}
    V(x,t)=G(x) + \alpha(x) u(t),
\end{align}
where $\alpha$ is a control shape function satisfying
\begin{align}\label{eq:shape_boundary_cond}
\alpha \in W^{1,\infty}(\Omega)\cap W^{2,\max(2,n)} (\Omega)
\text{ with }
\nabla \alpha \cdot \vec{n}=0 \ \ \text{on }
\Gamma.
\end{align}
A more precise characterization of $\alpha$ will be given in Subsection
\ref{subsec:alpha}. Thus the control enters in bilinear and separable form into
the state equation. While the case that $G$ is piecewise smooth
is certainly of interest, see e.g. \cite{Ris96}, we  focus here on the regular
case and assume that $G\in W^{1,\infty}(\Omega)\cap W^{2,\max(2,n)} (\Omega).$

We will consider system \eqref{eq:FPE} as an abstract bilinear control system
of the form
\begin{align}\label{eq:aux_intro_1}
  \dot{y}= \mathcal{A} y + u\mathcal{N} y + \mathcal{B}u, \ y(0)=y_0,
\end{align}
on an appropriate Hilbert space $\mathcal{Y}.$ In this setting, the unbounded
operator $\mathcal{A}$ will be the infinitesimal generator of an analytic,
strongly continuous semigroup on $\mathcal{Y}.$  The control objective will be to improve the asymptotic stability of the system to a steady state $\rho_\infty$.
With regard to the design of
suboptimal feedback laws for the shifted variable $y= \rho-\rho_\infty$, we consider two different strategies either of which
are based on the linearized version of \eqref{eq:aux_intro_1}. The first
feedback law relies on the infinite horizon cost functional and  is of the form
\begin{align}\label{eq:aux_intro_2}
  \mathcal{J}(y,u)=\frac{1}{2} \int_0^\infty \langle y,\mathcal{M} y\rangle +
|u|^2 \mathrm{d}t,
\end{align}
and is obtained by means of an algebraic operator Riccati equation
\begin{align}\label{ric_intr}
 \mathcal{A}^* {\Pi} + {\Pi}
\mathcal{A} - {\Pi} {\mathcal{B}}
\mathcal{B}^* {\Pi}   + \mathcal{M}  = 0.
\end{align}
via  $u=-\mathcal{B}^*\Pi y.$ For the associated nonlinear
closed-loop system, we show that for $\|y_0\| <
\varepsilon,$ the system converges to zero with an exponential rate.

As an alternative, we investigate a nonlinear feedback law based on the
solution $\Upsilon$ to an operator Lyapunov equation
\begin{align} \label{eq:lyapunov_intr}
  \mathcal{A}^* \Upsilon  + \Upsilon \mathcal{A} + 2 \mu I = 0,
\end{align}
for an appropriately chosen parameter $\mu >0.$ Though the control
will not be obtained from an optimal control problem, it will be shown to
yield a globally, exponentially stabilizing feedback law.

The boundary conditions that we have chosen (for the state equation and for $\alpha$) ensure a mass conservation property. Therefore, the control, which acts inside a differential operator, does not affect the dynamics on a subspace of the state space.
As a consequence, we actually have to work with a formulation of \eqref{eq:aux_intro_1} on the subspace of elements having zero mean and equations \eqref{ric_intr} and \eqref{eq:lyapunov_intr} have to be adapted accordingly.
Another
important aspect is the choice of the control potential $\alpha$ within $V$, see
\eqref{eq:FPE_potential}. Our choice is  guided by a criterium formulated in
the infinite dimensional version of the Hautus criterion.

Besides the large number of publications  which consider the Fokker-Planck
equations primarily from the stochastic point we mention  \cite{LeB08} which
gives an analytical framework for Fokker-Planck equations with irregular
coefficients, a semigroup approach for Kolmogorov operators with applications
to the Fokker-Planck equations \cite{Bog09}, and a detailed functions space
analysis of steady state solutions in \cite{Hua15}. Concerning stabilization of infinite dimensional systems by means of linearization techniques and the use of Riccati equations to devise feedback mechanisms we refer to e.g. \cite{BT14, RayT10, TheBR10}. Bilinear control systems  arise in the context of parameter estimation problems, for example, and in the control of quantum mechanical equations. Concerning controlability of such systems we refer to the monograph \cite{Kha10}, and the references given there.

The construction of  suboptimal feedback laws on the basis of applying linear
quadratic regulator theory to  conveniently defined linearizations
has many predecessors. In the context of distributed parameter systems we refer
to e.g. \cite{BarLT06,RayT10,Ray06}.
In all these papers the control enters linearly into the control system, while
it appears in a bilinear fashion in our problem  \eqref{eq:FPE} with the
control entering in the potential $V$ specified in  \eqref{eq:FPE_potential}.
We also stress that   the
control acts on the differential operator,
more precisely on the convection term if \eqref{eq:FPE} is considered as a
diffusion-convection equation. Hence our problem does not
belong to  the class of bilinear control problems which was investigated in
\cite{BalS79} where the control operator multiplies a bounded term in the state
equation.

A brief description of the contents of  the paper is given next. Section 2 is
devoted to establishing well-posedness of the state equation. We provide the
functions space setting in a form which is required for our results on
stabilization and as basis for the numerical treatment. Section 3 summarizes
some properties of the Fokker-Planck operator with reflecting boundary
conditions  and provides a succinct splitting of the state equation with respect
to the ground state and its complement. A Riccati-based stabilizing feedback
mechanism together with an appropriate choice for the control potential is
investigated in Section 4. Section  5 provides an alternative which is based on
a Lyapunov technique. Loosely speaking,  the Riccati-based approach is local
and allows an arbitrary decay rate, while the Lyapunov technique is global but
it only effects the first eigenspace different from the ground state.  Section
6 describes a numerical approach and provides examples which illustrate the
theoretical results.

\section{Well-posedness} In this short section we establish basic well-posedness
properties of the state equation \eqref{eq:FPE}. For  arbitrary $T>0$ we shall
refer to $\rho$ as (variational) solution of \eqref{eq:FPE} on $(0,T)$ if
 $$\rho \in W(0,T)=L^2(0,T;H^1(\Omega))\cap H^1(0,T;(H^1(\Omega))^*)$$
 and for a.e. $ t \in (0,T)$
\begin{equation}
\label{eq:K1}
\begin{aligned}
&\langle \rho_t(t),v\rangle+\langle \nu \nabla \rho(t)+\rho(t) \nabla G,\nabla v
\rangle + u(t) \langle\rho(t) \nabla \alpha,\nabla v \rangle=0 \;\text{for all}
\; v \in H^1(\Omega)\\[1.5ex]
&\rho(0)=\rho_0.
\end{aligned}
\end{equation}
Above $(H^1(\Omega))^*$ denotes the topological dual  of
$H^1(\Omega)$, with respect to $L^2(\Omega)$ as pivot space.  Let us recall
that $W(0,T) \subset C([0,T],L^2(\Omega))$, see e.g. \cite[Theorem 11.4]{Chi00}
, so that $\rho (0)$ is well defined. We also repeat the standing assumption
that $G$ and $\alpha$ are elements of  $W^{1,\infty}(\Omega)\cap W^{2,\max(2,n)}
(\Omega)$,
which in particular implies that the Neumann trace of $\alpha$ is well-defined.
These assumptions will be used in the following basic
well-posedness result on the state equation \eqref{eq:FPE}.
\begin{proposition}
	For every $u\in L^2(0,T)$ and $\rho_0 \in L^2(\Omega)$ there exists a
unique solution to (1.1).  If moreover $\rho_0 \in H^1(\Omega)$ and
$\Delta\alpha \in L^\infty(\Omega)$, then $\rho_t \in L^2 (0,T;L^2(\Omega)),\,
\rho \in C([0,T];H^1(\Omega)) $, $\nabla \cdot(\nu\Delta\rho +\rho\nabla
G),\Delta \rho\in
L^2(0,T;L^2(\Omega))$ and $(\nu\nabla\rho +\rho\nabla G)\cdot \vec n =0$ in
$L^2(0,T;H^{-1/2}(\Gamma))$.
\end{proposition}

\begin{proof}
	The claim can be verified by a standard Galerkin approximation technique
and we therefore only give the necessary a-priori estimates. Taking the inner
products with $\rho(t)$ in \eqref{eq:K1} we obtain
	\begin{equation*}
	\frac{1}{2}\frac{\mathrm{d}}{\mathrm{d}t}|\rho(t)|^2+\nu
|\nabla\rho(t)|^2 \leq  (|\nabla
G|_{L^\infty}+|u(t)||\nabla \alpha|_{L^\infty})|\rho(t)||\nabla\rho(t)|
	\end{equation*}
	and hence
	\begin{equation}
	\label{eq:K2}
\begin{aligned}
	\frac{\mathrm{d}}{\mathrm{d}t}|\rho(t)|^2+\nu|\nabla\rho(t)|^2 \leq
\frac{1}{\nu}(|\nabla
G|^2_{L^\infty}+|u(t)|^2|\nabla\alpha|^2_{L^\infty})|\rho(t)|^2.
	\end{aligned}
\end{equation}
By Gronwall's lemma we have for every $t\geq 0$
	$$
	|\rho(t)|^2\leq |\rho_0|^2  \exp (\, \frac{1}{\nu}\, \int_{0}^{t}(|\nabla
G|^2_{L^\infty}+|u(s)|^2|\nabla\alpha|^2_{L^\infty}) \, \mathrm{d}s\, ).
 	$$
 	Together with \eqref{eq:K2} this implies that $\rho \in L^2(0,T;
H^1(\Omega)) \cap L^\infty(0,T,L^2(\Omega))$. To verify that $ \rho_t\in
L^2(0,T; (H^1(\Omega))^*)$ we recall that
\begin{equation*}
|\phi|_{(H^1(\Omega))^*} =
\sup_{|\psi|_{H^1(\Omega)\le 1}} \langle \phi, \psi
\rangle_{(H^1(\Omega))^*,H^1(\Omega)},
\end{equation*}
for any $\phi \in (H^1(\Omega))^*$.  We
obtain the existence of a constant $C$ independent of $\rho, G$, and $\alpha$
such that
 \begin{equation*}
 \int_0^T |\rho_t(t)|_{H^1(\Omega)^*} \le C\, \int_0^T (|\nabla \rho(t)|^2 +
|\nabla G|_{L^\infty}^2 |\rho(t)|^2 +  |\nabla \alpha|_{L^\infty}^2
|\rho|^2_{L^\infty(0,T;L^2(\Omega))} |u(t)|^2  ) \, dt.
 \end{equation*}
 Since the right hand side is bounded we have that  $\rho_t \in
L^2(0,T;H^1(\Omega))^*$, and thus $\rho \in W(0,T)$.
 \\
 	
To gain extra regularity we set $v=e^{G/\nu}\rho_t$ in
\eqref{eq:K1} and obtain, using $\nabla \alpha \cdot  \vec{n}=0$ on $\Gamma$,
that
 	$$
 	|e^{G/2\nu}\rho_t|^2+\nu \langle
e^{-G/\nu}\nabla(e^{G/\nu}\rho),\nabla(e^{G/\nu}\rho_t)\rangle+ u(t)\langle
\rho(t)\nabla\alpha,\nabla(e^{G/\nu}\rho_t(t))\rangle=0
 	$$
 	and thus
 	$$
 	\begin{aligned}
|e^{G/2\nu}\rho_t|^2+\frac{\nu}{2}\frac{\mathrm{d}}{\mathrm{d}t}|e^{-G/2\nu}
\nabla(e^{ G/\nu}\rho)|^2 &\leq |u(t)|
\langle\nabla(\rho(t)\nabla_\alpha),e^{G/\nu}\rho_t\rangle\\
 	& \leq |u(t)||e^{G/2\nu}\nabla(\rho(t)\nabla\alpha)||e^{G/2\nu}\rho_t|.
 	\end{aligned}
 	$$
  	This implies the estimate
 	$$
 	\begin{aligned}
|e^{G/2\nu}\rho(t)_t|^2+\nu\frac{\mathrm{d}}{\mathrm{d}t}|e^{-G/2\nu}\nabla(e^{
G/\nu} \rho(t))|^2 &\leq |u(t)|^2|e^{+G/2\nu}(\rho(t)\Delta\alpha+
\nabla\rho(t)\cdot\nabla\alpha)|^2\\
 	&\leq
2|u(t)|^2(|e^{+G/2\nu}\rho(t)\Delta\alpha|^2+|e^{+G/2\nu}
\nabla\rho(t)\cdot\nabla\alpha|^2)\\
 	\leq 2|u(t)|^2(|e^{G/2\nu}\Delta
\alpha|^2_{L^\infty}|\rho(t)|^2&+2|\nabla\alpha|^2_{L^\infty}|e^{G/2\nu}
(\nabla\rho(t)+\frac{1}{\nu}\rho(t)\nabla G)|^2\\ 	
&+\frac{2}{\nu}|e^{G/2\nu}|^2_{L^\infty}|\nabla\alpha|^2_{L^\infty}|\nabla
G|^2_{L^\infty}|\rho(t)|^2).
 	\end{aligned}
 	$$
 	With $K_1=2|e^{G/2\nu}|^2
_{L^\infty}(|\Delta\alpha|^2_{L^\infty}+\frac{2}{\nu}|\nabla
\alpha|^2_{L^\infty}+|\nabla G|^2_{L^\infty})$ and
$K_2=4|\nabla\alpha|^2_{L^\infty}$ we have
 	$$
 	|e^{G/2\nu}\rho_t(t)|^2+\nu
\frac{\mathrm{d}}{\mathrm{d}t}|e^{-G/2\nu}\nabla(e^{G/\nu}\rho(t))|^2
 	\leq
K_1|u(t)|^2|\rho(t)|^2+K_2|u(t)|^2|e^{-G/2\nu}\nabla(e^{G/\nu}\rho(t))|^2.
 	$$
 	Integration on $(0,t)$, with $t\in(0,T]$ implies that
 	\begin{equation}
 	\label{eq:K3}
 	\begin{aligned}
&\nu|e^{-G/2\nu}\nabla(e^{G/\nu}\rho(t))|^2+\int_{0}^{t}|e^{G/2\nu}\rho_t(s)|^2\
,\mathrm{d}s\\ &\leq \nu|e^{-G/2\nu}\nabla(e^{G/\nu}\rho_0)|^2 +
K_1|\rho|^2_{C([0,T],L^2(\Omega))}|u|^2_{L^2(0,T)} \\
 	&+K_2\int_{0}^{t}|u(t)|^2
|e^{-G/2\nu}\nabla(e^{G/\nu}\rho(s))|^2\,\mathrm{d}s.
 	\end{aligned}
\end{equation}

Neglecting for a moment the second term on the left hand side of the
inequality and applying Gronwall's inequality implies that
$e^{-G/\nu}\nabla(e^{G/\nu}\rho)\in C([0,T],L^2(\Omega)^n)$ and hence $\rho \in
C([0,T],H^1(\Omega))$. Using this fact in \eqref{eq:K3} implies that
$\rho_t\in L^2(0,T;L^2(\Omega))$ and hence
\begin{equation*}
\nabla\cdot(\nu \nabla\rho
+\rho\nabla G) \in L^2(0,T;L^2(\Omega))
\end{equation*}
from \eqref{eq:K1}.  Thus $\nu
\nabla\rho +\rho\nabla G \in L^2(0,T; L^2_{\mathrm{div}}(\Omega))$,  and again
by
\eqref{eq:K1} we have $(\nu\nabla\rho +\rho\nabla G)\cdot \vec n =0$ in
$L^2(0,T;H^{-1/2}(\Gamma))$, see e.g. \cite[p.101]{Tar07} as desired.
Here
 $L^2_{\mathrm{div}}(\Omega))$ denotes the space $\{\vec \varphi \in
L^2(\Omega)^n\colon \nabla \cdot \,\vec \varphi\in L^2(\Omega) \}$.
The
properties that
 $\nabla\cdot (\nu \nabla\rho
+\rho\nabla G) \in L^2(0,T;L^2(\Omega))$ and $\rho\in L^2(0,T;H^1(\Omega))$ can
be exploited to obtain that  $\Delta\rho\in L^2(0,T;L^2(\Omega))$. Since
$\nabla\cdot(\rho \nabla G) = \nabla \rho \cdot \nabla G + \rho \Delta G$ and $
\nabla \rho \cdot \nabla G \in L^2(0,T; L^2(\Omega)) $, by the fact that
$\nabla G \in L^\infty (\Omega)$, it suffices to argue that $\rho \Delta G \in
L^2(0,T;L^2(\Omega))$. This follows from the continuous embedding of
$H^1(\Omega)$ into $L^{(\frac{2n}{n-2})}(\Omega)$ and the H\"older inequality
with weights $p=\frac{n}{n-2}, p'=\frac{n}{2}$.
\end{proof}

The solution of the Fokker-Planck equation satisfies  structural properties
including preservation of probability and nonnegativity which we establish next.

\begin{proposition}\label{lem:positivity}Let  $u\in L^2(0,T)$ and $\rho_0 \in
L^2(\Omega)$.

  (i) For every  $t\in [0,T]$ we have $\int_{\Omega}^{}\rho (t) \,\mathrm{d}x =
\int_{\Omega_0}^{} \rho_0\, \mathrm{d}x$.

  (ii)  If $ \rho_0 \ge 0$ a.e. on $\Omega$, then  $\rho(x,t) \ge 0$ for
all $t > 0$ and almost all $x \in \Omega.$
\end{proposition}
\begin{proof}
Setting $v=1$ in  \eqref{eq:K1} we obtain the preservation of probability
$\int_{\Omega}^{}\rho(t)\, \mathrm{d}x =\int_{\Omega}^{}\rho_0 \, \mathrm{d}x $
for all
$t\in[0,T]$.

Turning to the verification of (ii) let us denote by $\rho = \rho^+ - \rho^-$
the decomposition of the state $\rho$
into its nonnegative and its negative part, respectively. It then also holds
that
\begin{align*}
  e^{\frac{G}{\nu}} \rho = \left(e^{\frac{G}{\nu}}\rho \right)^+ -
\left(e^{\frac{G}{\nu}} \rho \right)^- = e^{\frac{G}{\nu}}(\rho^+ - \rho^-),
\end{align*}
and $\rho^+,\rho^- \in L^2(0,T;H^1(\Omega))$, see e.g. \cite[Lemma 11.2]{Chi00}
Note that we can write $\nu \nabla \rho + \rho \nabla G = e^{-\frac{G}{\nu}}
\nabla (\nu e^{\frac{G}{\nu}} \rho)$.
Hence by setting $ v= -e^{\frac{G}{\nu}} \rho^-$  in  \eqref {eq:K1} we  obtain
that
\begin{align*}
 \left \langle \frac{\mathrm{d}}{\mathrm{d}t} (\rho^+ - \rho^-), -
e^{\frac{G}{\nu}} \rho^- \right\rangle
= - \nu \left\langle e^{-\frac{G}{\nu}} \nabla (e^{\frac{G}{\nu}} \rho),\nabla
(e^{\frac{G}{\nu}}\rho^- ) \right\rangle+ u  \left\langle \rho \nabla \alpha,
\nabla (e^{\frac{G}{\nu}} \rho^- ) \right\rangle.
\end{align*}
With \cite[Lemma 11.2]{Chi00} it now follows that
\begin{align*}
  \frac{1}{2} \frac{\mathrm{d}}{\mathrm{d}t } \| e^{\frac{G}{2\nu}} \rho^- \|^2
&= \left \langle \frac{\mathrm{d}}{\mathrm{d}t} (e^{\frac{G}{2\nu}}\rho^-),
e^{\frac{G}{2\nu}} \rho^- \right \rangle  \\
&\le - \nu \| e^{-\frac{G}{2\nu}}
\nabla (e^{\frac{G}{\nu}} \rho^- ) \|^2 + |u| \| e^{\frac{G}{2\nu}} \rho^-
\nabla \alpha \| \| e^{-\frac{G}{2\nu}} \nabla (e^{\frac{G}{\nu}}\rho^ -) \| \\
&\le - \nu \| e^{-\frac{G}{2\nu}}
\nabla (e^{\frac{G}{\nu}} \rho^- ) \|^2 + \frac{\nu}{2}\| e^{-\frac{G}{2\nu}}
\nabla (e^{\frac{G}{\nu}}\rho^ -) \|^2  \\
&\qquad +  \frac{1}{2\nu} |u|^2  \| \nabla
\alpha \|^2_{L^{\infty}(\Omega)} \| e^{\frac{G}{2\nu}} \rho^-
 \| ^2  \\
 &\le \frac{1}{2\nu} |u|^2  \| \nabla
\alpha \|^2_{L^{\infty}(\Omega)} \| e^{\frac{G}{2\nu}} \rho^-
 \| ^2
\end{align*}
An application of Gronwall's inequality now yields that from $\rho^-(0)=0,$ it
follows that $e^{\frac{G}{2\nu}} \rho^-(t) =0 $,  and hence that
$\rho^-(t)=0$ for all $t\ge 0.$

\end{proof}

\section{The operator form of the Fokker-Planck equation}

The goal of this section is to formulate \eqref{eq:FPE} as an
abstract Cauchy problem such that the linearized system can be studied by means
of semigroup methods. Hence, let us consider the abstract \textit{bilinear}
control system
\begin{equation}\label{eq:abs_pur_bil}
 \begin{aligned}
  \dot{\rho}(t) &= \mathcal{A} \rho(t) + \mathcal{N}\rho(t)
u(t), \\
\rho(0) &= \rho_0,
\end{aligned}
\end{equation}
where the operators $\mathcal{A}$ and $\mathcal{N}$ are defined as follows
\begin{equation}\label{eq:A_N_op}
\begin{aligned}
  \mathcal{A}\colon \mathcal{D}(\mathcal{A})&\subset L^2(\Omega) \to
L^2(\Omega),\\
\mathcal{D}(\mathcal{A})&= \left\{\rho \in H^2(\Omega) \left| (\nu \nabla \rho +
\rho \nabla G) \cdot \vec{n}  =0 \text{ on } \Gamma \right. \right\}, \\
\mathcal{A}\rho& = \nu \Delta \rho + \nabla \cdot (\rho \nabla G), \\[1ex]
\mathcal{N}\colon H^1(\Omega)& \to L^2(\Omega),\ \  \mathcal{N}\rho =
\nabla \cdot (\rho \nabla \alpha).
\end{aligned}
\end{equation}
Let us recall \cite{Ada75} that we have the following embeddings
\begin{align*}
  W^{2,2}(\Omega) \hookrightarrow  \begin{cases} C(\Omega) &\mbox{if } n
=1,2,3, \\
L^q(\Omega), \ q \in [1,\infty)  \ \ & \mbox{if } n =4, \\
L^{\frac{2n}{n-4}}(\Omega)  \ & \mbox{if } n\ge 5. \end{cases}
\end{align*}
Since by assumption $\alpha,G \in
W^{1,\infty}(\Omega)\cap W^{2,\max(2,n)}(\Omega),$
a short computation involving the H\"older inequality shows that $\mathcal{A}$ and
$\mathcal{N}$ are well-defined. Its $L^2(\Omega)$-adjoints are now given
by
\begin{equation}\label{eq:A_N_adj_op}
\begin{aligned}
  \mathcal{A}^*\colon \mathcal{D}(\mathcal{A}^*)&\subset L^2(\Omega) \to
L^2(\Omega),\\
\mathcal{D}(\mathcal{A}^*)&= \left\{\varphi \in H^2(\Omega) \left| (\nu \nabla
\varphi ) \cdot \vec{n}  =0 \text{ on } \Gamma \right. \right\}, \\
\mathcal{A}^*\varphi & = \nu \Delta \varphi - \nabla G \cdot \nabla \varphi,
\\[1ex]
\mathcal{N}^*\colon H^1(\Omega) &\to
L^2(\Omega),\ \ \mathcal{N}^*\varphi = -\nabla \varphi \cdot \nabla \alpha.
\end{aligned}
\end{equation}

We emphasize that, due to \eqref{eq:shape_boundary_cond}, a solution $\rho
\in \mathcal{D}(\mathcal{A})$  of \eqref{eq:abs_pur_bil} automatically satisfies
the zero flux boundary conditions of \eqref{eq:FPE}.

\subsection{Properties of the Fokker-Planck operator}

For what follows, it will be convenient to summarize some known
qualitative properties of the uncontrolled Fokker-Planck equation
\begin{equation}\label{eq:abs_lin_unc}
 \begin{aligned}
  \dot{\rho}(t) &= \mathcal{A} \rho(t), \quad
\rho(0) = \rho_0,
\end{aligned}
\end{equation}
compare \cite[Chapter 5/6]{Ris96}. For the sake of a
self-contained presentation, we also provide the proofs for the statements.

Following \cite{Ris96}, let us introduce
$\Phi(x) = \log{\nu} + \frac{G(x)}{\nu},$ such that
$e^{\frac{\Phi(x)}{2}}=\sqrt{\nu}
e^{\frac{G(x)}{2\nu}}.$ Further, define the  operator
\begin{equation}\label{eq:shift_op}
  \begin{aligned}
    \mathcal{A}_s&\colon \mathcal{D}(\mathcal{A}_s) \subset L^2(\Omega)
\to
L^2(\Omega), \\
  \mathcal{D}(\mathcal{A}_s)&= \left\{ \varrho \in H^2(\Omega) \left| (\nu
\nabla \varrho + \frac{1}{2} \varrho \nabla G )\cdot \vec{n}=0 \text{ on
} \Gamma \right. \right\}, \\
\mathcal{A}_s&= e^{\frac{\Phi}{2}} \mathcal{A} e^{-\frac{\Phi}{2}}.
  \end{aligned}
\end{equation}
A straightforward calculation using $\nu \nabla
\Phi = \nabla G$ shows that
$$\mathcal{A}(e^{-\frac{\Phi}{2}}\rho) = \nu e^{-\frac{\Phi}{2}}
\left(\Delta \rho + \frac{1}{2}\rho \Delta \Phi - \frac{1}{4} \rho \nabla \Phi
\cdot \nabla \Phi \right).$$
Using the previously mentioned embeddings and H\"older inequality, it can be
shown that $\mathcal{A}_s\rho = e^{\frac{\Phi}{2}} \mathcal{A}
e^{-\frac{\Phi}{2}} \rho$ is indeed in $L^2(\Omega)$ for $\rho \in H^2(\Omega).$
Moreover, it turns out that the spectrum of $\mathcal{A}$ coincides with that of
$\mathcal{A}_s$ and, in particular, is discrete.
\begin{lemma}\label{lem:spec_A}
  The operator $\mathcal{A}_s$ is self-adjoint. The spectrum
$\sigma(\mathcal{A}_s)$ of $\mathcal{A}_s$ consists of pure point
spectrum contained in $\overline{\mathbb R}_-$ with $0\in
\sigma(\mathcal{A}_s)$ and only accumulation point $-\infty.$ The
eigenfunctions $\{\psi _i\}_{i=0}^\infty$ form a complete orthogonal set.
Further
$\sigma( \mathcal{A}_s )= \sigma(\mathcal{A})$ and $\psi_i$ is an
eigenfunction of $\mathcal{A}$ if and only if $e^{\frac{\Phi}{2}} \psi_i$ is
an
eigenfunction of $\mathcal{A}_s.$ Similarly, $\psi_i$ is an eigenfunction of
$\mathcal{A}$ if and only if $e^{\Phi} \psi_i$ is an eigenfunction of
$\mathcal{A}^*.$ Finally, $\rho_\infty= e^{-\Phi}$ is an
eigenfunction of $\mathcal{A}$ associated to the eigenvalue $0.$
\end{lemma}
\begin{proof}
  Let $\varrho_1,\varrho_2 \in \mathcal{D}(\mathcal{A}_s).$ Since $\nu \nabla
\Phi = \nabla G,$ we conclude that
$\mathcal{A}_s\varrho_1 $ is
given as
\begin{equation}\label{eq:aux_spec}
\begin{aligned}
\mathcal{A}_s\varrho_1=  e^{\frac{\Phi}{2}}\mathcal{A}
e^{-\frac{\Phi}{2}}\varrho_1
&= \nu e^{\frac{\Phi}{2}}  \left(\Delta (e^{-\frac{\Phi}{2}}\varrho_1 ) +
\nabla
\cdot ( e^{-\frac{\Phi}{2}} \varrho_1 \nabla \Phi )\right) \\
&= \nu
e^{\frac{\Phi}{2}}  \nabla \cdot \left(\nabla (e^{-\frac{\Phi}{2}}\varrho_1 ) +
 e^{-\frac{\Phi}{2}} \varrho_1 \nabla \Phi \right) \\
 &= \nu
e^{\frac{\Phi}{2}}  \nabla \cdot \left( e^{-\Phi} \nabla ( e^{\frac{
\Phi}{2}} \varrho_1 ) \right).
\end{aligned}
\end{equation}
Similarly we obtain that
\begin{align*}
  0=(\nu \nabla \varrho_1 + \frac{1}{2} \varrho_1 \nabla G )\cdot \vec{n}= (\nu
e^{-\frac{\Phi}{2}} \nabla (e^{\frac{\Phi}{2}} \varrho_1 ))\cdot \vec{n}\ \
\text{on } \Gamma.
\end{align*}
Thus, it holds that
\begin{align*}
  \int_{\Omega} \varrho_2 \mathcal{A}_s\varrho_1\; \mathrm{d}x &=
\int_{\Omega} \varrho_2 \left(
\nu
e^{\frac{\Phi}{2}}  \nabla \cdot ( e^{-\Phi} \nabla ( e^{\frac{
\Phi}{2}} \varrho_1 )  )\right) \; \mathrm{d}x \\
&=  \int_{\Gamma} (\varrho_2 \nu e^{-\frac{\Phi}{2}}
\nabla(e^{\frac{\Phi}{2}} \varrho_1) )\cdot \vec{n}  \;
\mathrm{d}s
- \nu \int_{\Omega} e^{-\Phi} \nabla ( e^{\frac{\Phi}{2}} \varrho_2
) \nabla ( e^{\frac{
\Phi}{2}} \varrho_1 ) \; \mathrm{d}x \\
&= \int_{\Gamma} (\varrho_2 \nu e^{-\frac{\Phi}{2}}
\nabla(e^{\frac{\Phi}{2}} \varrho_1) )\cdot \vec{n}  \;
\mathrm{d}s
- \nu \int_\Gamma \left( \varrho_1 e^{-\frac{\Phi}{2}} \nabla
(e^{\frac{\Phi}{2}}
\varrho_2 ) \right) \cdot \vec{n} \; \mathrm{d}s \\
&\quad + \nu \int_{\Omega} \varrho_1 e^{\frac{\Phi}{2}}  \nabla \cdot
\left(e^{-\Phi} \nabla (e^{\frac{\Phi}{2}} \varrho_2  ) \right)  \; \mathrm{d}x
\\
&= \int_{\Omega} \varrho_1  \mathcal{A}_s \varrho_2 \; \mathrm{d}x.
\end{align*}
As a consequence we have that $\mathcal{A}_s^*= \mathcal{A}_s,$
thus it is a self-adjoint and closed operator in $L^2(\Omega).$ By
\eqref{eq:aux_spec}, we also have that for each $\varrho \in
\mathcal{D}(\mathcal{A}_s)\colon$
\begin{align*}
  \int_{\Omega} \varrho \mathcal{A}_s \varrho \;\mathrm{d}x =   -\nu
\int_{\Omega} |\nabla(e^{\frac{\Phi}{2}}\varrho  ) |^2
e^{-\Phi}\;\mathrm{d}x \le 0,
\end{align*}
and hence $ \mathcal{A}_s$ is a negative operator. It follows that
there exists $\beta \in \mathbb R_+ $ which is in the resolvent set
of $-\mathcal{A}_s.$ Consequently, for each $f\in L^2(\Omega)$
\begin{align*}
  (-\mathcal{A}_s + \beta I) \varrho =f
\end{align*}
has a unique solution $\varrho\in \mathcal{D}(\mathcal{A}_s)$ depending
continuously on $f.$ We observe that $\varrho$ is the solution to
\begin{align*}
 - \nu \Delta \varrho - \frac{1}{2}e^{\frac{\Phi}{2}} \nabla \cdot (
e^{-\frac{\Phi}{2}} \varrho \nabla G ) + \beta \varrho +\frac{1}{2}\nabla
\varrho \cdot \nabla G &=  f
&& \hspace{-1cm}\text{in } \Omega , \\
(\nu \nabla \varrho + \frac{1}{2} \varrho  \nabla G) \cdot \vec{n} &=0  &&
\hspace{-1cm}\text{on }
\Gamma.
\end{align*}
Testing this equation with $\varrho$ we obtain
\begin{align*}
  \nu \int _{\Omega} | \nabla \varrho |^2 \;\mathrm{d}x +
\frac{1}{2}\int_{\Omega} e^{-\frac{\Phi}{2}} \varrho \nabla
(e^{\frac{\Phi}{2}} \varrho)
\cdot \nabla G \; \mathrm{d}x + \frac{1}{2}\int_{\Omega}\varrho \nabla \varrho
\cdot \nabla G \; \mathrm{d}x = \int _\Omega  ( f-\beta \varrho) \varrho \;
\mathrm{d}x
\end{align*}
and hence
\begin{align*}
  \nu |\nabla \varrho|_{L^2(\Omega)}^2& \le  |\nabla G
|_{L^\infty(\Omega)}  |e^{-\frac{\Phi}{2}}|_{L^{\infty}(\Omega)} |e^{\frac{\Phi}{2}}|_{L^\infty(\Omega)}  |
\varrho
|_{L^2(\Omega)} (|\nabla \varrho| _{L^2(\Omega)}
 +  |\varrho |_{L^2(\Omega)})  \\ & \quad + ( |
f|_{L^2(\Omega)} + \beta |\varrho|_{L^2(\Omega)} )|\varrho|_{L^2(\Omega)}.
\end{align*}
Together with the continuous dependence of $\varrho \in L^2(\Omega)$ on $f,$ we
deduce the existence of a constant $K$ such that
\begin{align*}
  |\varrho| _{H^1(\Omega)} \le K | f|_{L^2(\Omega)}.
\end{align*}
Thus $-\mathcal{A}_s + \beta I$ has a compact resolvent as operator in
$L^2(\Omega).$ Consequently, the spectrum of $\mathcal{A}_s$ consists
entirely of isolated eigenvalues with finite multiplicity in $\mathbb R_-,$
with only accumulation point $-\infty,$ see, e.g., \cite[Chapter 3]{Kat80}.

The relation between the eigenfunctions of $\mathcal{A}$ and $\mathcal{A}_s$
follow immediately  from the definition of the operator $\mathcal{A}_s.$
Moreover, note that by \eqref{eq:aux_spec} it holds that $e^{-\frac{\Phi}{2}}
$ is an eigenfunction of $\mathcal{A}_s$ associated
to the eigenvalue $0.$ The associated eigenfunctions of $\mathcal{A}$ and
$\mathcal{A}^*$ are $\rho_\infty= e^{-\Phi}$ and the
constant function $\mathbbm{1} $ with value $1,$ respectively.

\end{proof}

Since $\mathcal{A}_s$ is self-adjoint, it follows from
\begin{align*}
 \int_{\Omega} \varrho \mathcal{A}_s \varrho \; \mathrm{d}x \le 0 \ \ \text{for
all } \varrho \in \mathcal{D}(\mathcal{A}_s)
\end{align*}
that $\mathcal{A}_s$ is dissipative, see \cite[Chapter 1, Definition
4.1]{Paz83}. Together with the fact that the range of $\beta I - \mathcal{A}_s$
is
surjective, the Lumer-Phillips theorem \cite[Chapter 1, Theorem 4.3]{Paz83}
implies that $\mathcal{A}_s$ generates a semigroup of contractions on
$L^2(\Omega).$ Consequently $\mathcal{A}$ generates a semigroup $S(t)$ of class
$G(M,0)$ in $L^2(\Omega),$ i.e. $\| S(t)\| \le M$ for all $t.$ Moreover, $S(t)$ 
is an analytic semigroup, see,
e.g., \cite[Section 5.4] {Tan79} and the mild solution to
\eqref{eq:abs_lin_unc} is given by
\begin{equation}\label{eq:mild_sol_lin_unc}
  \begin{aligned}
    \rho(t)=S(t) \rho_0.
  \end{aligned}
\end{equation}

\subsection{Decoupling the Fokker-Planck equation}

According to Lemma \ref{lem:spec_A}, it is clear that $\rho_\infty=e^{-\Phi}$
is a stationary solution of \eqref{eq:FPE}. From now on, let us assume
that $\rho_\infty$ is normalized such that $\int_{\Omega} \rho_\infty
\; \mathrm{d}x = 1.$ While $\rho_\infty$ is asymptotically stable, the convergence rate (given by the second eigenvalue) can be undesirably slow.
An approximation of the convergence rate for small values of $\nu$ is given by:
$Ce^{-\Delta_G/\nu}$, where $C>0$ is a constant and where the constant
$\Delta_G$ -- called
energy activation -- is the highest potential barrier that the particle has to
overcome to reach the most stable equilibrium. This estimate is proved in
\cite[p.251]{MatS81} for 2-dimensional infinite potential fields. The case of
a bistable double-well potential with reflecting conditions (in dimension 1) is
also treated in \cite[Section 5.10.2]{Ris96}.

Following similar works
\cite{Ray06,TheBR10}, we subsequently study the applicability of a Riccati-based
feedback law obtained from a suitable stabilization problem.
Starting from \eqref{eq:abs_pur_bil}, let us introduce the shifted state $y:=
\rho - \rho_\infty.$ Using that $\mathcal{A} \rho_\infty=0,$ we obtain the
transformed system
\begin{equation}\label{eq:abs_bil}
  \begin{aligned}
    \dot{y}(t) &= \mathcal{A} y(t) + \mathcal{N}y(t)
u(t) + \mathcal{B} u(t), \\
y(0) &= \rho_0-\rho_\infty,
  \end{aligned}
\end{equation}
with $\mathcal{B}=\mathcal{N}\rho_\infty.$
Here, the control operator $\mathcal{B}$ and its adjoint
are defined as
\begin{align*}
  \mathcal{B}&\colon \mathbb R \to L^2(\Omega), \ \ \mathcal{B} c=
c\mathcal{N}\rho_\infty,\\
  \mathcal{B}^* & \colon L^2(\Omega) \to \mathbb R, \ \ \mathcal{B}^*v =
\langle \mathcal{N} \rho_\infty, v \rangle  .
\end{align*}
For our feedback design, it will be convenient to work with a decoupled version
of \eqref{eq:abs_bil}. We therefore introduce the projection $\mathcal{P}$
\textit{onto $\mathbbm{1}^\perp$ along $\rho_\infty$}
\begin{equation*}
  \begin{aligned}
    &\mathcal{P}\colon L^2(\Omega)\to L^2(\Omega), \quad \mathcal{P}y = y -
 \int_{\Omega} y \;\mathrm{d}x \;  \rho_\infty, \\
&\mathrm{im}(\mathcal{P})  =\left\{ v \in L^2(\Omega)\colon
\int_\Omega v\; \mathrm{d}x = 0 \right\}, \quad
\mathrm{ker}(\mathcal{P}) = \mathrm{span}\left\{\rho_\infty \right\}.
  \end{aligned}
\end{equation*}
Hence, the complementary projection $\mathcal{Q}$ is given as
\begin{equation*}
  \begin{aligned}
&\mathcal{Q}\colon L^2(\Omega)\to L^2(\Omega), \quad \mathcal{Q}y =
(I-\mathcal{P} )y =
\int_{\Omega} y \;\mathrm{d}x  \; \rho_\infty, \\
&\mathrm{im}(\mathcal{Q})  =\mathrm{ker}(\mathcal{P}), \quad
\mathrm{ker}(\mathcal{Q}) = \mathrm{im}(\mathcal{P}).
  \end{aligned}
\end{equation*}
With these definitions, the $L^2(\Omega)$ adjoint  of $\mathcal{P}$ is the
projection $\mathcal{P}^*$ \textit{onto $\rho_\infty^\perp$ along
$\mathbbm{1}$}
\begin{align*}
 &\mathcal{P}^*\colon L^2(\Omega)\to L^2(\Omega), \quad  \mathcal{P}^*y = y-
\int_{\Omega}  \rho_\infty y\; \mathrm{d}x\;\mathbbm{1},\\
&\mathrm{im}(\mathcal{P}^*)  =\left\{ v \in L^2(\Omega)\colon
\int_\Omega \rho_\infty v\; \mathrm{d}x = 0 \right\}, \quad
\mathrm{ker}(\mathcal{P}^*) = \left\{ \mathbbm{1} \right\}.
\end{align*}
Finally, the complementary projection $\mathcal{Q}^*$ reads
\begin{align*}
  &\mathcal{Q}^*\colon L^2(\Omega)\to L^2(\Omega), \quad  \mathcal{Q}^*y =
\int_{\Omega}  \rho_\infty y\; \mathrm{d}x\;\mathbbm{1}, \\
&\mathrm{im}(\mathcal{Q}^*)  =\mathrm{ker}(\mathcal{P}^*), \quad
\mathrm{ker}(\mathcal{Q}^*) = \mathrm{im}(\mathcal{P}^*).
\end{align*}
We now can decompose our state space as follows
\begin{equation}\label{eq:decomp_state_space}
\begin{aligned}
\mathcal{Y}&=L^2(\Omega) = \mathrm{im}(\mathcal{P}) \oplus
\mathrm{im}(\mathcal{Q})=:
\mathcal{Y}_{\mathcal{P}} \oplus\mathcal{Y}_{\mathcal{Q}}, \\
y &=
y_{\mathcal{P}} + y_{\mathcal{Q}}= \mathcal{P}y + \mathcal{Q}y, \ y \in
L^2(\Omega).
\end{aligned}
\end{equation}
This results in the following decomposition of \eqref{eq:abs_bil}
\begin{equation*}
\begin{aligned}
  \dot{y}_\mathcal{P} + \dot{y}_{\mathcal{Q}}&= \mathcal{A}(y_\mathcal{P}
  +y_\mathcal{Q} ) + \mathcal{N} (y_\mathcal{P}   +y_\mathcal{Q} )u +
\mathcal{B} u \\
y_\mathcal{P}(0) &= \mathcal{P}\rho_0 ,\quad  y_\mathcal{Q}(0)= \mathcal{Q}
\rho_0 - \rho_\infty.
\end{aligned}
\end{equation*}
Applying respectively $\mathcal{P}$ and $\mathcal{Q}$ to this equation
yields
\begin{equation}\label{eq:abs_bil_proj_a}
\begin{aligned}
  \begin{pmatrix}
   \dot{y}_\mathcal{P} \\
   \dot{y}_\mathcal{Q}
  \end{pmatrix}
 = \begin{pmatrix}
      \mathcal{P} \mathcal{A} & \mathcal{P} \mathcal{A} \\
      \mathcal{Q} \mathcal{A} & \mathcal{Q} \mathcal{A}
   \end{pmatrix}
   \begin{pmatrix} y_\mathcal{P} \\ y_\mathcal{Q} \end{pmatrix}+
   \begin{pmatrix}
      \mathcal{P} \mathcal{N} & \mathcal{P} \mathcal{N} \\
      \mathcal{Q} \mathcal{N} & \mathcal{Q} \mathcal{N}
   \end{pmatrix}
   \begin{pmatrix} y_\mathcal{P} \\ y_\mathcal{Q} \end{pmatrix} u +
   \begin{pmatrix}
        \mathcal{P} \mathcal{B} \\
       \mathcal{Q} \mathcal{B}
   \end{pmatrix} u.
\end{aligned}
\end{equation}
Let us note
that $\mathcal{A}\rho_\infty=0$, $\mathcal{A}^*\mathbbm{1}=0$ and $\mathcal{N}^*
\mathbbm{1}=0.$ For $ y_\mathcal{P}\in \mathrm{im}(\mathcal{P})
\cap \mathcal{D}(\mathcal{A}), y_\mathcal{Q}\in
\mathrm{im}(\mathcal{Q})\cap
\mathcal{D}(\mathcal{A})$ and $v\in \mathcal{D}(\mathcal{A}^*),$ observe
that
\begin{equation*}
\langle \mathcal{A} y_\mathcal{Q},v \rangle = 0, \ \
 \langle \mathcal{Q}\mathcal{A}y_\mathcal{P},v
\rangle = \langle y_\mathcal{P},\mathcal{A}^*\mathcal{Q}^* v
\rangle  = 0.
\end{equation*}
For $y \in H^1(\Omega)$ and $v \in L^2(\Omega)$,
\begin{align*}
& \langle \mathcal{Q} \mathcal{N} y, v \rangle= \langle y,  \mathcal{N}^*
\mathcal{Q}^* v \rangle= 0, \\
& \langle \mathcal{Q} \mathcal{B}, v \rangle   =
\langle \mathcal{Q} \mathcal{N} \rho_\infty, v \rangle   = \langle
\rho_\infty, \mathcal{N}^* \mathcal{Q}^* v \rangle  =0.
\end{align*}
%
%
Hence, we have the identities:
\begin{equation}
\begin{aligned} \label{eq:identities}
& \mathcal{P}\mathcal{A}= \mathcal{A}\ \text{(on $\mathcal{D}(\mathcal{A})$)}, \quad
\mathcal{Q} \mathcal{A}= 0 \ \text{(on $\mathcal{D}(\mathcal{A})$)}, \quad
\mathcal{P} \mathcal{N}= \mathcal{N} \ \text{(on $H^1(\Omega)$)} \\
& \mathcal{Q} \mathcal{N}= 0 \ \text{(on $H^1(\Omega)$)}, \quad
\mathcal{P} \mathcal{B}= \mathcal{B}\ \text{(on $\mathbb R$)}, \quad
\mathcal{Q} \mathcal{B}= 0\ \text{(on $\mathbb R$)}.
\end{aligned}
\end{equation}
As a consequence, \eqref{eq:abs_bil_proj_a} simplifies as follows:
\begin{equation*}
\begin{pmatrix} \dot{y}_{\mathcal{P}} \\ \dot{y}_{\mathcal{Q}} \end{pmatrix}=
\begin{pmatrix}
\mathcal{P}\mathcal{A} & 0 \\ 0 & 0
\end{pmatrix}
\begin{pmatrix}
y_{\mathcal{P}} \\ y_{\mathcal{Q}}
\end{pmatrix}
+ u
\begin{pmatrix}
\mathcal{P} \mathcal{N} & \mathcal{P} \mathcal{N} \\ 0 & 0
\end{pmatrix}
\begin{pmatrix}
y_{\mathcal{P}} \\ y_{\mathcal{Q}}
\end{pmatrix}
+ u
\begin{pmatrix}
\mathcal{P}\mathcal{B} \\ 0,
\end{pmatrix}
\end{equation*}
hence,
\begin{equation*}
  \begin{aligned}
       \dot{y}_\mathcal{P}
 &=         \mathcal{A}  y_\mathcal{P} +
       \mathcal{N}y_\mathcal{P}u +
\mathcal{N} (\mathcal{Q}\rho_0 - \rho_\infty ) u +
         \mathcal{B} u, \quad y_\mathcal{P}(0) = \mathcal{P} \rho
_0, \\
 y_\mathcal{Q}(t) &= \mathcal{Q}\rho_0 - \rho_\infty, \ t\ge 0.
   \end{aligned}
\end{equation*}
By definition of $\mathcal{B}$ and the fact that  $\int_{\Omega} \rho_0 \;
\mathrm{d}x=1,$ we finally obtain:
\begin{equation}\label{eq:abs_bil_proj_b}
\begin{aligned}
\dot{y}_\mathcal{P} &= \widehat{\mathcal{A}} y_\mathcal{P} +
\widehat{\mathcal{N}}y_\mathcal{P}u + \widehat{\mathcal{B}}u,
\quad y_\mathcal{P}(0)= \mathcal{P} \rho_0, \\
y_\mathcal{Q}(t) &= \mathcal{Q}\rho_0 - \rho_\infty=0, \ t\ge 0,
\end{aligned}
\end{equation}
where $I_\mathcal{P}\colon \mathcal{Y}_\mathcal{P} \to \mathcal{Y}$ denotes the
injection of $\mathcal{Y}_\mathcal{P}$ into $\mathcal{Y}$ and
\begin{align*}
 \widehat{\mathcal{A}}&=\mathcal{A}
I_\mathcal{P} \text{ with }\mathcal{D}(\widehat{\mathcal{A}})=
\mathcal{D}(\mathcal{A})\cap
\mathcal{Y}_\mathcal{P} , \\
\widehat{\mathcal{N}} &= \mathcal{N} I_\mathcal{P} \text{ with }
\mathcal{D}(\widehat{\mathcal{N}}) =H^1(\Omega) \cap \mathcal{Y}_\mathcal{P}
 , \\
\widehat{\mathcal{B}} &= \mathcal{B} 
\end{align*}
are operators considered in $\mathcal{Y}_\mathcal{P}.$

\section{A Riccati-based feedback law}\label{sec:Ric}

\subsection{Stabilizing the linearized system}\label{subsec:lin}
For the linearized  decoupled and shifted system
\begin{equation}\label{eq:abs_lin}
  \begin{aligned}
   \dot{y}_ \mathcal{P} &=  (\widehat{\mathcal{A}}+\delta I) y_ \mathcal{P}(t)
+
 \widehat{\mathcal{B}}
u, \quad
y_ \mathcal{P}(0) = \mathcal{P}\rho_0,
  \end{aligned}
\end{equation}
let us focus on the cost functional
\begin{equation}\label{eq:cost_func_abs}
 J(y_ \mathcal{P},u)= \frac{1}{2}\int_0^\infty \langle y_\mathcal{P}
(t),\mathcal{M}
y _ \mathcal{P}(t) \rangle
_{L^2(\Omega)} \; \mathrm{d}t + \frac{1}{2}\int_0^\infty
|u(t)|^2 \; \mathrm{d}t,
\end{equation}
where $\mathcal{M}\in\mathcal{L}(\mathcal{Y}_\mathcal{P})$ is a self-adjoint
nonnegative operator on
$\mathcal{Y}_{\mathcal{P}}$ which is such that the pair $(\mathcal{A},\mathcal{M})$ is detectable.
We denote by $\Theta$ the orthogonal projection on
$\mathcal{Y}_{\mathcal{P}}$:
\begin{equation}\label{eq:orth_proj}
  \begin{aligned}
    &\Theta\colon L^2(\Omega)\to L^2(\Omega), \quad \Theta y = y
-
 \frac{1}{|\Omega|}\int_{\Omega} y \;\mathrm{d}x \; \mathbbm{1}, \\
&\mathrm{im}(\Theta)  =\mathrm{im}(\mathcal{P})=\mathcal{Y}_\mathcal{P}, \quad
\mathrm{ker}(\Theta) = \left\{ \mathbbm{1} \right\}.
  \end{aligned}
\end{equation}
Note that $\Theta^* = \Theta$ and, in particular,
$\Theta=I_\mathcal{P}^*.$ Let us then define the operator
\begin{align*}
  \mathcal{A}^\sharp\colon
\mathcal{D}(\mathcal{A}^\sharp) &\subset \mathcal{Y}_\mathcal{P} \to
\mathcal{Y}_\mathcal{P}, \quad
\mathcal{D}(\mathcal{A}^\sharp)= \mathcal{D}(\mathcal{A}^*) \cap
\mathcal{Y}_\mathcal{P}, \quad
\mathcal{A}^\sharp\phi  = \Theta \mathcal{A}^*\phi.
\end{align*}
\begin{lemma}\label{lemma:adjointAhat}
The operator $\mathcal{A}^{\sharp}$ is the adjoint operator of
$\widehat{\mathcal{A}}$. Moreover, let $(\lambda,\phi) \in \mathbb{R} \times
\mathcal{D}(\mathcal{A}^{\sharp})$ be such that $\mathcal{A}^{\sharp} \phi=
\lambda \phi$. Then, $(\lambda, \mathcal{P}^* \phi)$ is an eigenpair of
$\mathcal{A}^*$. Conversely, if $(\lambda,\varphi) \in \mathbb{R} \times
\mathcal{D}(\mathcal{A}^*)$ is an eigenpair of $\mathcal{A}^*$, then
$(\lambda,\Theta \varphi)$ is an eigenpair of $\mathcal{A}^{\sharp}$.
\end{lemma}
\begin{proof}
For $y_\mathcal{P} \in \mathcal{D}(\widehat{\mathcal{A}})$ and $z_\mathcal{P}
\in \mathcal{D}(\mathcal{A}^\sharp)$ it now holds that
\begin{align*}
  \langle \widehat{\mathcal{A}}y_\mathcal{P},z_\mathcal{P} \rangle
  =\langle \mathcal{A}I_\mathcal{P}y_\mathcal{P},z_\mathcal{P} \rangle
=\langle y_\mathcal{P},I_\mathcal{P}^*\mathcal{A}^* z_\mathcal{P} \rangle
= \langle  y_\mathcal{P}, \Theta\mathcal{A}^*z_\mathcal{P}
\rangle  = \langle y_\mathcal{P},\mathcal{A}^\sharp z_\mathcal{P}
\rangle.
\end{align*}
Note also that
\begin{align*}
 \langle \mathcal{P} \widehat{\mathcal{A}} y_\mathcal{P},z_\mathcal{P} \rangle
= \langle y_\mathcal{P}, I_\mathcal{P}^* \mathcal{A}^*
(z_\mathcal{P}-\mathcal{Q}^* z_\mathcal{P} ) \rangle = \langle y_\mathcal{P},
I_\mathcal{P}^* \mathcal{A}^* z_\mathcal{P} \rangle = \langle y_\mathcal{P},
\mathcal{A}^\sharp z_\mathcal{P} \rangle,
\end{align*}
such that we conclude that $\mathcal{A}^\sharp =
(\mathcal{P}\mathcal{A} I_\mathcal{P})^*=\widehat{\mathcal{A}}^*.$
For
what follows, let $y \in \mathcal{D}(\mathcal{A})$ and
$z\in\mathcal{D}(\mathcal{A}^*)$ be given. Since
$I=\mathcal{P}^*+\mathcal{Q}^*$ and $I = \Theta+ (I-\Theta),$ we then have
\begin{align*}
  \langle y,\mathcal{A}^*z \rangle =  \langle y,\mathcal{P}^* \Theta
\mathcal{A}^*z \rangle + \langle y,\mathcal{P}^* (I-\Theta) \mathcal{A}^* z
\rangle + \langle y, \mathcal{Q}^* \mathcal{A}^*z \rangle.
\end{align*}
Using that $\mathrm{im}(I-\Theta)=\mathrm{ker}(\mathcal{P}^*)$ and
$\mathrm{im}(\mathcal{Q}) = \{\rho_\infty\},$ we obtain
\begin{align}\label{eq:aux_eigf}
 \langle y,\mathcal{A}^*z \rangle =  \langle y,\mathcal{P}^* \Theta
\mathcal{A}^*z \rangle  + \langle \mathcal{A} \mathcal{Q}y, z \rangle = \langle
y,\mathcal{P}^*\Theta \mathcal{A}^*z \rangle.
\end{align}
This yields the following relation between the eigenfunctions of $\mathcal{A}^*$
and those of $\mathcal{A}^\sharp.$ Let $(\lambda,\phi)$ be such that
$\mathcal{A}^\sharp \phi = \lambda \phi.$ It then follows  by
\eqref{eq:aux_eigf} that
\begin{align*}
  \mathcal{P}^* (\lambda \phi) = \mathcal{P}^*( \mathcal{A}^\sharp \phi ) =
\mathcal{P}^* (\Theta \mathcal{A}^*) \phi = \mathcal{A}^* \phi = \mathcal{A}^*
(\mathcal{P}^* + \mathcal{Q}^*) \phi = \mathcal{A}^*\mathcal{P}^* \phi.
\end{align*}
Hence, $(\lambda,\mathcal{P}^*\phi)$ is an eigenpair of $\mathcal{A}^*.$
Analogously, assume that $(\lambda,\varphi)$ satisfies $\mathcal{A}^* \varphi =
\lambda \varphi.$ We now obtain
\begin{align*}
  \Theta (\lambda \varphi) = \Theta (\mathcal{A}^*\varphi)=\Theta \mathcal{A}^*
(\Theta + (I-\Theta) ) \varphi = \Theta \mathcal{A}^* \Theta \varphi,
\end{align*}
implying that $(\lambda,\Theta \varphi)$ is an eigenpair of
$\mathcal{A}^\sharp.$
\end{proof}

\subsection{Stabilizability and the choice of $\alpha$}
\label{subsec:alpha}
Let us also note that the adjoint of $\widehat{\mathcal{B}}=\mathcal{P}
\mathcal{B}$ as operator from $\mathbb R$ to $\mathcal{Y}_\mathcal{P}$ is given
by $\widehat{\mathcal{B}}^*=\mathcal{B}^* \mathcal{P}^* = \mathcal{B}^*
I_\mathcal{P}$ and we drop the notation $I_\mathcal{P}$ below.

Up to this point, we have assumed that $\alpha \in
W^{1,\infty}(\Omega)\cap W^{2,\max(2,n)}(\Omega)$ is such that
\eqref{eq:shape_boundary_cond} is fulfilled. Let us now provide further details
on how to choose $\alpha.$ It is well-known \cite{CurZ95} that the cost
functional \eqref{eq:cost_func_abs} is naturally associated to the following
operator Riccati equation
\begin{align*}
 (\widehat{\mathcal{A}}+\delta I)^* \widehat{\Pi} + \widehat{\Pi}
(\widehat{\mathcal{A}}+\delta I) - \widehat{\Pi} \widehat{\mathcal{B}}
\widehat{\mathcal{B}}^* \widehat{\Pi}   + \mathcal{M}  = 0,
\end{align*}
which is interpreted in the weak sense, i.e.,
\begin{align*}
 &\left\langle (\widehat{\mathcal{A}}+\delta I)^* \widehat{\Pi}y_1,y_2
\right\rangle _{L^2(\Omega)} + \left\langle \widehat{\Pi}
(\widehat{\mathcal{A}}+\delta I)y_1,y_2 \right \rangle _{L^2(\Omega)} \\ &
\qquad  - \left\langle \widehat{\mathcal{B}}^*
\widehat{\Pi}y_1,\widehat{\mathcal{B}}^* \widehat{\Pi} y_2 \right\rangle
_{\mathbb R} + \left\langle \mathcal{M} y_1,y_2\right\rangle _{L^2(\Omega)} =
0,
\end{align*}
for all $y_1,y_2 \in \mathcal{D}(\widehat{\mathcal{A}}).$
In particular, in case the pair $(\widehat{\mathcal{A}},\widehat{ \mathcal{B}}
)$ is $\delta$-stabilizable, see \cite[Definition 5.2.1]{CurZ95}, there exists
a unique nonnegative self-adjoint solution $\widehat{\Pi} \in \mathcal{L}(
\mathcal{Y}_\mathcal{P})$ such that
$$ \widehat{\mathcal{A}}_\Pi:=
\widehat{\mathcal{A}}+\delta
I-\widehat{\mathcal{B}} \widehat{\mathcal{B}}^* \widehat{\Pi}
$$
generates an exponentially stable semigroup on $\mathcal{Y}_\mathcal{P},$ see,
e.g., \cite[p.295]{CurZ95},\cite[p.125-127]{LasT00} or
\cite[p.519]{Benetal07}, where it is also proved that $\widehat{\Pi}$ enjoys
extra regularity since $\widehat{\mathcal{A}}$ is analytic. With
regard to $\delta$-stabilizability of
$(\widehat{\mathcal{A}},\widehat{ \mathcal{B}} ),$ assume that $\varphi_i$ are
eigenfunctions of $\mathcal{A}^*$ associated to the eigenvalues
\begin{align*}
  -\delta \le \lambda_d \le \dots \le \lambda_2 < 0=\lambda_1.
\end{align*}
With the notation introduced before, consider then the elliptic equation
\begin{equation}\label{eq:elliptic_alpha}
\begin{aligned}
  \nabla \cdot (\rho_\infty \nabla \alpha) &=  \mathcal{P} \sum_{i=2}^d
e^{-\Phi} \varphi_i && \text{in } \Omega, \\
(\rho_\infty \nabla \alpha )\cdot \vec{n} &= 0  && \text{on } \Gamma.
\end{aligned}
\end{equation}
From classical elliptic regularity results, see, e.g., \cite[Theorem
3.28/3.29]{Tro87}, we
conclude that there exists a unique solution $\alpha \in
W^{2,p}(\Omega)/\mathbb R$ for any $p>0$ to \eqref{eq:elliptic_alpha}. In
particular, $\alpha \in W^{1,\infty}(\Omega)/\mathbb R\cap
W^{2,\max(2,n)}(\Omega)/\mathbb R.$ As a
consequence of
this choice of $\alpha,$ we obtain the desired stabilizability result.
\begin{lemma}\label{lem:stabilizability}
  Let $\alpha \in W^{1,\infty}(\Omega)/\mathbb R\cap
W^{2,\max(2,n)}(\Omega)/\mathbb R$ denote the unique
solution to
\eqref{eq:elliptic_alpha}. Then the pair $(\widehat{\mathcal{A}},\widehat{
\mathcal{B}} )$ is $\delta$-stabilizable.
\end{lemma}
\begin{proof}
  We are going to verify the assertion by means of the infinite dimensional
Hautus test for stabilizability see \cite[Part V, Proposition 3.3]{Benetal07}
or \cite[Theorem 5.2.11]{CurZ95}. Hence, we need to show that
\begin{align*}
    \mathrm{ker}(\lambda I-\mathcal{A}^\sharp) \cap
\mathrm{ker}(\widehat{\mathcal{B}}^*) = \{0\} \quad
\text{for } \lambda \in   \overline{ {\mathbb C}}_{-\delta} \cap
\sigma(\mathcal{A}^\sharp),
\end{align*}
where $\overline{ {\mathbb C}}_{-\delta}= \{ \lambda \in \mathbb{C} \,|\, \text{Re}(\lambda) \geq - \delta \}$.
Let us therefore assume that $(\lambda_j,\phi_j),$ $j\in \{2,\dots,d\}$ is an
eigenpair of $\mathcal{A}^\sharp.$ By Lemma \ref{lem:spec_A}, Lemma
\ref{lemma:adjointAhat} and \eqref{eq:elliptic_alpha}, it follows that
\begin{align*}
  \widehat{\mathcal{B}}^* \phi_j= \langle \widehat{\mathcal{B}},\phi_j \rangle
=  \left\langle \mathcal{P} \sum_{i=2}^d
e^{-\Phi} \varphi_i,\phi_j \right\rangle = \sum _{i=2}^d \langle e^{-\Phi}
\varphi_i, \varphi_j \rangle = \| e^{-\frac{\Phi}{2}} \varphi_j \|^2
\end{align*}
which shows the statement.
\end{proof}

From now on, we assume that $\alpha$ is such that the Hautus criterion is satisfied and therefore that $(\widehat{\mathcal{A}},\widehat{\mathcal{B}})$ is $\delta$-stabilizable.

\subsection{The Riccati equation}

With the notation introduced in \eqref{eq:decomp_state_space},
consider the following two Riccati equations:
\begin{align*}
(\mathcal{A}^*+ \delta \mathcal{P}^*)\Pi + \Pi (\mathcal{A}+ \delta \mathcal{P})
- \Pi \mathcal{B} \mathcal{B}^* \Pi + \mathcal{P}^* \mathcal{M} \mathcal{P} = 0,
\quad & \Pi \in \mathcal{L}(\mathcal{Y}),\ \Pi^*= \Pi \tag{R1}
\label{eq:Riccati1} \\
(\widehat{\mathcal{A}}^*+ \delta I)\widehat{\Pi} + \widehat{\Pi}
(\widehat{\mathcal{A}}+ \delta I) - \widehat{\Pi} \widehat{\mathcal{B}}
\widehat{ \mathcal{B}}^*
\widehat{\Pi} + \mathcal{M} = 0, \quad & \widehat{\Pi} \in
\mathcal{L}(\mathcal{Y}_\mathcal{P}), \ \widehat{\Pi}^*= \widehat{\Pi}.
\tag{R2} \label{eq:Riccati2}
\end{align*}

\begin{lemma} \label{lem:reductionRiccati}
If the operator $\Pi \in \mathcal{L}(\mathcal{Y})$ is a solution to
\eqref{eq:Riccati1}, then
$\widehat{\Pi}:= \Theta \Pi I_\mathcal{P} \in
\mathcal{L}(\mathcal{Y}_\mathcal{P})$ is a
solution to \eqref{eq:Riccati2} and there exists $\gamma \in
\mathbb{R}$ such that  $\Pi= \mathcal{P}^* \widehat{\Pi} \mathcal{P} + \gamma
\mathbbm{1} \mathbbm{1}^*$.
Conversely, if $\widehat{\Pi}$ is a solution to \eqref{eq:Riccati2}, then for
all $\gamma \in \mathbb{R}$, $\Pi= \mathcal{P}^* \widehat{\Pi} \mathcal{P} +
\gamma \mathbbm{1} \mathbbm{1}^*$ is a solution to \eqref{eq:Riccati1}.
\end{lemma}

\begin{proof}
Let us define:
\begin{equation*}
\mathcal{R}\colon y \in \mathcal{Y} \mapsto (\mathcal{P}y,\langle \mathbbm{1},y
\rangle) \in \mathcal{Y}_\mathcal{P} \times \mathbb{R}.
\end{equation*}
The operator $\mathcal{R}$ is a homeomorphism. Note that for all $(z,\alpha) \in
\mathcal{Y}_\mathcal{P} \times \mathbb{R}$ and for all $y \in \mathcal{Y}$,
\begin{equation*}
\mathcal{R}^{-1} (z,\alpha) = z + \alpha \rho_\infty, \quad
\mathcal{R}^*(z,\alpha) = \mathcal{P}^*z + \alpha \mathbbm{1}, \quad
\mathcal{R}^{-*}y =(\Theta y, \langle \rho_\infty, y \rangle).
\end{equation*}
Let $\Pi \in \mathcal{L}(\mathcal{Y})$ be a solution to \eqref{eq:Riccati1} and
define  $\widetilde{\Pi}= \mathcal{R}^{-*} \Pi \mathcal{R}^{-1} \in
\mathcal{L}(\mathcal{Y}_\mathcal{P} \times \mathbb{R})$.
The operator $\widetilde{\Pi}$ is a solution to the following equation:
\begin{equation} \label{eq:RiccatiProj}
(\widetilde{\mathcal{A}} + \delta \widetilde{\mathcal{P}})^*
\widetilde{\Pi}+\widetilde{\Pi} (\widetilde{\mathcal{A}}+ \delta
\widetilde{\mathcal{P}})
 - \widetilde{\Pi}
\widetilde{\mathcal{B}} \widetilde{\mathcal{B}}^* \widetilde{\Pi} + \widetilde{\mathcal{M}}= 0,
\end{equation}
where:
\begin{equation*}
\widetilde{\mathcal{M}}= \mathcal{R}^{-*}\mathcal{P}^* \mathcal{M}\mathcal{P} \mathcal{R}^{-1}, \quad
\widetilde{\mathcal{A}}= \mathcal{R} \mathcal{A} \mathcal{R}^{-1}, \quad
\widetilde{\mathcal{P}}= \mathcal{R} \mathcal{P} \mathcal{R}^{-1}, \quad
\text{and} \quad
\widetilde{\mathcal{B}}= \mathcal{R}\mathcal{B}.
\end{equation*}
We represent any operator $\mathcal{X} \in \mathcal{L}(\mathcal{Y}_\mathcal{P}
\times \mathbb{R})$ as follows: $\mathcal{X}= \begin{pmatrix} \mathcal{X}_{11} &
\mathcal{X}_{12} \\ \mathcal{X}_{21} & \mathcal{X}_{22} \end{pmatrix}$,
where $\mathcal{X}_{11} \in \mathcal{L}(\mathcal{Y}_\mathcal{P})$,
$\mathcal{X}_{12} \in \mathcal{Y}_\mathcal{P}$, $\mathcal{X}_{21} \in
\mathcal{Y}_\mathcal{P}^*$, and $\mathcal{X}_{22} \in \mathbb{R}$ are uniquely
defined by the relation:
\begin{equation*}
\mathcal{X}(y,\beta)= (\mathcal{X}_{11}y + \mathcal{X}_{12} \beta,
\mathcal{X}_{21}y+ \mathcal{X}_{22} \beta), \quad \forall (y,\beta) \in
\mathcal{Y}_\mathcal{P} \times \mathbb{R}.
\end{equation*}
One can easily check with \eqref{eq:identities} that:
\begin{equation} \label{eq:reducedOp}
\widetilde{\mathcal{A}}= \begin{pmatrix} \widehat{\mathcal{A}} & 0 \\
0 & 0 \end{pmatrix}, \quad
\widetilde{\mathcal{P}}= \begin{pmatrix} I & 0 \\
0 & 0 \end{pmatrix}, \quad
\widetilde{\mathcal{M}}= \begin{pmatrix} \mathcal{M} & 0 \\ 0 & 0
\end{pmatrix}, \quad
\widetilde{\mathcal{B}}\widetilde{\mathcal{B}}^*= \begin{pmatrix}
\mathcal{B}\mathcal{B}^* & 0 \\ 0 & 0 \end{pmatrix}.
\end{equation}
Combining \eqref{eq:RiccatiProj} and \eqref{eq:reducedOp}, we obtain that
$\widetilde{\Pi}_{11}$ is a solution to \eqref{eq:Riccati2}. Moreover,
\begin{equation*}
(\widehat{\mathcal{A}}^* + \delta I) \widetilde{\Pi}_{12} - \widetilde{\Pi}_{11} \mathcal{B}
\mathcal{B}^* \widetilde{\Pi}_{12}= 0 \quad \text{and} \quad
\widetilde{\Pi}_{12}^* \mathcal{B} \mathcal{B}^* \widetilde{\Pi}_{12}= 0.
\end{equation*}
Thus, $\mathcal{B}^* \widetilde{\Pi}_{12}= 0$ and $(\widehat{\mathcal{A}}^* + \delta I)
\widetilde{\Pi}_{12}= 0$. As a consequence of the Hautus criterion,
$\widetilde{\Pi}_{12}=0$. Setting $\widehat{\Pi}= \widetilde{\Pi}_{11}$ and
$\gamma= \widetilde{\Pi}_{22}$, one can easily check that:
$\Pi= \mathcal{R}^* \widetilde{\Pi} \mathcal{R}= \mathcal{P}^* \widehat{\Pi}
\mathcal{P} + \gamma \mathbbm{1} \mathbbm{1}^*$.
The converse implication can be proved in a similar manner.
\end{proof}

\begin{lemma} \label{lmm:existenceUniquenessRiccati}
There exists a unique non-negative self-adjoint operator $\Pi$ solution to \eqref{eq:Riccati1} such that $\Pi \rho_\infty= 0$.
\end{lemma}

\begin{proof}
Let $\Pi$ be defined by $\Pi= \mathcal{P}^* \widehat{\Pi} \mathcal{P}$, where $\widehat{\Pi}$ is the unique non-negative solution to \eqref{eq:Riccati2}. By Lemma \ref{lem:reductionRiccati}, $\Pi$ is a solution to \eqref{eq:Riccati1} and clearly, $\Pi$ is non-negative and $\Pi \rho_\infty= 0$.
Now, let $\Pi'$ be a non-negative self-adjoint operator, solution to \eqref{eq:Riccati1}, and such that $\Pi' \rho_\infty= 0$. By Lemma \ref{lem:reductionRiccati}, there exist an operator $\widehat{\Pi}'$, solution to \eqref{eq:Riccati2} and $\gamma \in \mathbb{R}$ such that $\Pi'= \mathcal{P}^* \widehat{\Pi}' \mathcal{P} + \gamma \mathbbm{1} \mathbbm{1}^*$. Since $\Pi' \rho_\infty= 0$, we have:
$
0= \mathcal{P}^* \widehat{\Pi}' \mathcal{P} \rho_\infty + \gamma \mathbbm{1} \mathbbm{1}^* \rho_\infty,
$
and therefore, $\gamma= 0$, since $\mathcal{P} \rho_\infty= 0$ and $\mathbbm{1}^* \rho_\infty \neq 0$.
Since $\Pi'$ is non-negative, we obtain that for all $y \in \mathcal{Y}_{\mathcal{P}}$,
$0 \leq \langle y, \Pi' y \rangle= \langle \mathcal{P}y, \widehat{\Pi}' \mathcal{P} y \rangle = \langle y, \widehat{\Pi}' y \rangle$, which proves that $\widehat{\Pi}'$ is non-negative. Therefore, $\widehat{\Pi}'= \widehat{\Pi}$ and $\Pi'= \mathcal{P}^* \widehat{\Pi} \mathcal{P}= \Pi$. Finally, $\Pi$ is the unique non-negative solution to \eqref{eq:Riccati1} such that $\Pi \rho_\infty = 0$.
\end{proof}

\begin{remark}
The Riccati equations \eqref{eq:Riccati1} and \eqref{eq:Riccati2} both provide the same feedback. Let $\Pi$ be a solution to \eqref{eq:Riccati1}, let $\widehat{\Pi}$ be a solution to \eqref{eq:Riccati2}, let $\gamma \in \mathbb{R}$ be such that $\Pi= \mathcal{P}^* \widehat{\Pi} \mathcal{P} + \gamma \mathbbm{1} \mathbbm{1}^*$. Then, for all $y \in \mathcal{Y}$,
\begin{equation*}
- \mathcal{B}^* \Pi y
= - \mathcal{B}^* \big( \mathcal{P}^* \widehat{\Pi} \mathcal{P} + \gamma \mathbbm{1} \mathbbm{1}^* \big)y
= - \mathcal{B}^* \mathcal{P}^* \widehat{\Pi} (\mathcal{P}y)
= - \widehat{\mathcal{B}}^* \widehat{\Pi}(\mathcal{P}y),
\end{equation*}
since $\mathcal{B}^* \mathbbm{1}= 0$ and $\widehat{\mathcal{B}}^*=\mathcal{B}^* \mathcal{P}^*$. The first and the last term of the above equation respectively correspond to the feedback controls associated with $\Pi$ and $\widehat{\Pi}$.
\end{remark}

\subsection{Local exponential stabilization of the nonlinear system}

In this section, we study the effect of the static state feedback law
$u=-\widehat{\mathcal{B}}^* \widehat{\Pi} y_\mathcal{P}$ when applied to the
nonlinear system
\begin{align*}
  \dot{y}_ \mathcal{P} &=  \widehat{\mathcal{A}} y_ \mathcal{P}  +
u\widehat{\mathcal{N}}y_ \mathcal{P}   +
 \widehat{\mathcal{B}}
u , \quad
y_ \mathcal{P}(0)= \mathcal{P}\rho_0.
\end{align*}
Since we are interested in local exponential stabilization results, let us
introduce the transformed state $z_{\mathcal{P}}=e^{\delta t} y_\mathcal{P}$
where $\delta$ is as in Subsection \ref{subsec:lin}. We then obtain the
transformed
system
\begin{align*}
    \dot{z}_{ \mathcal{P}} &=  ( \widehat{\mathcal{A}}  + \delta I )z_{
\mathcal{P}}  +
u\widehat{\mathcal{N}}z_ {\mathcal{P}}   +
 \widehat{\mathcal{B}}
 (e^{\delta t} u) , \quad
z_ {\mathcal{P}}(0)= \mathcal{P}\rho_0.
\end{align*}
As a consequence, our goal is a local stability result for the system
\begin{align*}
    \dot{z}_{ \mathcal{P}} &=  ( \widehat{\mathcal{A}}  + \delta I )z_{
\mathcal{P}}  - (\widehat{\mathcal{B}}^* \widehat{\Pi} z_{\mathcal{P}})
{\widehat{\mathcal{N}}_\delta}z_
{\mathcal{P}}   -
  \widehat{\mathcal{B}}( \widehat{\mathcal{B}}^* \widehat{\Pi}z_
{\mathcal{P}}  ), \quad
z_ {\mathcal{P},\delta}(0)= \mathcal{P}\rho_0.
\end{align*}
where $\widehat{\mathcal{N}}_\delta = e^{-\delta t} \widehat{\mathcal{N}}.$
Using once more the notation
$\widehat{\mathcal{A}}_\Pi= \widehat{\mathcal{A}} +\delta I
- \widehat{\mathcal{B}} \widehat{\mathcal{B}}^* \widehat{\Pi},$ let us first
consider the following nonhomogeneous system
\begin{align}\label{eq:nonhomg}
  \dot{z}_\mathcal{P} = \widehat{\mathcal{A}}_\Pi z_\mathcal{P}
+ f, \quad z_\mathcal{P}(0) =\mathcal{P} \rho_0.
\end{align}
Since $\widehat{\mathcal{B}}\widehat{\mathcal{B}}^* \Pi \in
\mathcal{L}(\mathcal{Y}_\mathcal{P})$ we conclude that
\begin{equation}\label{eq:dom_ops}
 \begin{aligned}
\mathcal{D}(\widehat{
\mathcal{A}}_\Pi ) &= \mathcal{D}(\widehat{\mathcal{A}}) =
\mathcal{D}(\mathcal{A}) \cap
\mathcal{Y}_\mathcal{P}    , \quad
\mathcal{D}(\widehat{
\mathcal{A}}^*_\Pi )  = \mathcal{D}(\widehat{A}^*) = \mathcal{D}(\mathcal{A}^*)
\cap \mathcal{Y}_\mathcal{P} .
 \end{aligned}
\end{equation}
For the following calculus of interpolation spaces, assume that $\lambda \in
\mathbb R$ in the resolvent set of $\mathcal{A}$ is chosen such that the
fractional powers of $\mathcal{A}_\lambda:= (\lambda I - \mathcal{A})$ are
well-defined. From \cite[Section 1.17.1]{Tri78}, it follows that
\begin{align*}
  [\mathcal{D}(\mathcal{A}_\lambda)\cap
\mathcal{Y}_\mathcal{P},\mathcal{Y}\cap
\mathcal{Y}_\mathcal{P}]_\alpha &=
[\mathcal{D}(\mathcal{A}_\lambda),\mathcal{Y}]_\alpha \cap
\mathcal{Y}_\mathcal{P},  \\
[\mathcal{D}(\mathcal{A}^*_\lambda)\cap
\mathcal{Y}_\mathcal{P},\mathcal{Y}\cap
\mathcal{Y}_\mathcal{P}]_\alpha &=
[\mathcal{D}(\mathcal{A}^*_\lambda),\mathcal{Y}]_\alpha \cap
\mathcal{Y}_\mathcal{P} .
\end{align*}
According to \cite[Appendix 3A]{LasT00}, for $\alpha=\frac{1}{2},$ we can
identify the above interpolation spaces as follows
\begin{align*}
  [\mathcal{D}(\mathcal{A} _\lambda),\mathcal{Y}]_{\frac{1}{2}} \cap
\mathcal{Y}_\mathcal{P} = H^1(\Omega)\cap \mathcal{Y}_\mathcal{P} =
[\mathcal{D}(\mathcal{A}^*_\lambda),\mathcal{Y}]_{\frac{1}{2}} \cap
\mathcal{Y}_\mathcal{P}.
\end{align*}
Moreover, with \cite[Volume I, Section 12]{LioM72} it holds that
\begin{align*}
  [ [\mathcal{D}(\mathcal{A}_\lambda),\mathcal{Y}_\mathcal{P} ]_{\frac{1}{2}},
[\mathcal{D}(\mathcal{A}_\lambda^*),\mathcal{Y}_\mathcal{P}  ]'_{\frac{1}{2}}
]_{\frac{1}{2}} = \mathcal{Y}_\mathcal{P}.
\end{align*}
For the following result, let us introduce the space
\begin{align*}
  W_\mathcal{P}(Q_\infty):= L^2(0,\infty;H^1(\Omega) \cap
\mathcal{Y}_\mathcal{P}) \cap
H^1(0,\infty;[H^1(\Omega)\cap \mathcal{Y}_\mathcal{P}]')
\end{align*}
endowed with the norm
\begin{align*}
  \| y_\mathcal{P} \|_{W_\mathcal{P}(Q_\infty)} = \left(\int_0^\infty \|
y_\mathcal{P} \|_{H^1(\Omega)}^2 + \| y'_\mathcal{P} \|_{[H^1(\Omega)]'}^2 \;
\mathrm{d}t \right)^{\frac{1}{2}}.
\end{align*}
Based on known regularity results for analytic semigroups, we now have.
\begin{theorem}\label{thm:reg_nonh}
 Let $f\in L^2(0,\infty;[H^1(\Omega)\cap \mathcal{Y}_\mathcal{P}]')$ and
$\rho_0 \in L^2(\Omega)$ be given. Then there exists a unique mild solution
$z_\mathcal{P}\in W_\mathcal{P}(Q_\infty)$ to \eqref{eq:nonhomg} satisfying
\begin{align*}
 \| z_\mathcal{P} \| _{W_\mathcal{P}(Q_\infty)} \le C ( \| f\|_{L^2(0,\infty;
[H^1(\Omega) ]')} + \| \rho_0 \|_{L^2(\Omega)} ) .
\end{align*}
In particular, $z_\mathcal{P} \in C_b([0,\infty);\mathcal{Y}_\mathcal{P}).$
\end{theorem}
 \begin{proof}
  The result immediately follows from \cite[Chapter 3, Theorem
2.2]{Benetal07} and \cite[Volume I, Theorem 4.2]{LioM72} together with the
given characterizations of the involved interpolation spaces.
 \end{proof}

 The next lemma will be used in the following theorem.

\begin{lemma}\label{lem:nonl_lip}
  Let $y_\mathcal{P},z_\mathcal{P} \in W_\mathcal{P}(Q_\infty).$ Then
\begin{align*}
&  \left\|(\widehat{\mathcal{B}}^* \widehat{\Pi} y_\mathcal{P} )
\widehat{\mathcal{N}}_\delta y_\mathcal{P}
- (\widehat{\mathcal{B}}^* \widehat{\Pi} z_\mathcal{P} )
\widehat{\mathcal{N}}_\delta z_\mathcal{P}
\right\|_{L^2(0,\infty;[H^1(\Omega) ]')} \\
& \hspace{3cm} \le \widetilde{C} (\|
y_\mathcal{P} \|_{W_\mathcal{P}(Q_\infty)} + \| z_\mathcal{P}
\|_{W_\mathcal{P}(Q_\infty)} ) \| y_\mathcal{P} -
z_\mathcal{P}\|_{W_\mathcal{P}(Q_\infty)}.
\end{align*}
\end{lemma}

\begin{proof}
  First note that we can extend the operator $\widehat{\mathcal{N}}\colon
H^1(\Omega) \cap \mathcal{Y}_\mathcal{P} \to \mathcal{Y}_\mathcal{P}$
to a continuous linear operator $\widehat{\mathcal{N}} \colon
\mathcal{Y}_\mathcal{P} \to [H^1(\Omega)\cap \mathcal{Y}_\mathcal{P}]'.$ For
$y_\mathcal{P} ,z_\mathcal{P} \in
W_\mathcal{P}(Q_\infty),$ we have
\begin{align*}
&  \left\|(\widehat{\mathcal{B}}^* \widehat{\Pi} y_\mathcal{P} )
\widehat{\mathcal{N}}_\delta y_\mathcal{P}
- (\widehat{\mathcal{B}}^* \widehat{\Pi} z_\mathcal{P} )
\widehat{\mathcal{N}}_\delta z_\mathcal{P}
\right\|_{L^2(0,\infty;[H^1(\Omega)  ]')} \\[1ex]
&\ \ \le \widetilde{C}\left\|(\widehat{\mathcal{B}}^*
\widehat{\Pi} (y_\mathcal{P}-z_\mathcal{P}) )
\widehat{\mathcal{N}} y_\mathcal{P}
\right\|_{L^2(0,\infty;[H^1(\Omega) ]')}
+\widetilde{C}\left\|(\widehat{\mathcal{B}}^* \widehat{\Pi}
z_\mathcal{P} )
\widehat{\mathcal{N}} (y_\mathcal{P}-z_\mathcal{P})
\right\|_{L^2(0,\infty;[H^1(\Omega)  ]')}.
\end{align*}
For the first term, it holds that
\begin{align*}
 &\left\|(\widehat{\mathcal{B}}^*
\widehat{\Pi} (y_\mathcal{P}-z_\mathcal{P}) )
\widehat{\mathcal{N}} y_\mathcal{P}
\right\|_{L^2(0,\infty;[H^1(\Omega) ]')} ^2   = \int_0^\infty | \langle
\widehat{\Pi} \widehat{\mathcal{N}}\rho_\infty,y_\mathcal{P}-z_\mathcal{P}
\rangle |^2 \|
\widehat{\mathcal{N}} y_\mathcal{P} \| _{[H^1(\Omega)  ]'}^2 \; \mathrm{d}t \\
&\quad \le \widetilde{C} \int_0^\infty \| y_\mathcal{P} - z_\mathcal{P}
\|_{L^2(\Omega)}^2 \| y_\mathcal{P} \|_{L^2(\Omega)}^2 \; \mathrm{d}t \le
\widetilde{C} (\sup_{t \in [0,\infty)} \| y_\mathcal{P} \|_{L^2(\Omega)} )^2
  \| y_\mathcal{P} - z_\mathcal{P} \|^2_{W_\mathcal{P}(Q_\infty)}.
 \end{align*}
With \cite[Volume I, Theorem 4.2]{LioM72} this yields
\begin{align*}
  \left\|(\widehat{\mathcal{B}}^*
\widehat{\Pi} (y_\mathcal{P}-z_\mathcal{P}) )
\widehat{\mathcal{N}} y_\mathcal{P}
\right\|_{L^2(0,\infty;[H^1(\Omega) ]')} \le \widetilde{C}
  \| y_\mathcal{P} \|_{W_\mathcal{P}(Q_\infty)}
  \| y_\mathcal{P} - z_\mathcal{P} \|_{W_\mathcal{P}(Q_\infty)}.
\end{align*}
Similarly, we continue with
\begin{align*}
 & \left\|(\widehat{\mathcal{B}}^* \widehat{\Pi}
z_\mathcal{P} )
\widehat{\mathcal{N}} (y_\mathcal{P}-z_\mathcal{P})
\right\|_{L^2(0,\infty;[H^1(\Omega)  ]')} = \int_0^\infty | \langle
\widehat{\Pi} \widehat{\mathcal{B}}, z_\mathcal{P} \rangle |^2 \|
\widehat{\mathcal{N}}( y_\mathcal{P}- z_\mathcal{P}) \| _{[H^1(\Omega)  ]'}^2
\; \mathrm{d}t \\
&\quad \le \widetilde{C} \int_0^\infty \| z_\mathcal{P} \|_{L^2(\Omega)}^2   \|
y_\mathcal{P} - z_\mathcal{P}
\|_{L^2(\Omega)}^2  \; \mathrm{d}t \le
\widetilde{C} (\sup_{t \in [0,\infty)} \| z_\mathcal{P} \|_{L^2(\Omega)} )^2
  \| y_\mathcal{P} - z_\mathcal{P} \|^2_{W_\mathcal{P}(Q_\infty)}.
\end{align*}
As before, this leads to
\begin{align*}
  \left\|(\widehat{\mathcal{B}}^* \widehat{\Pi}
z_\mathcal{P} )
\widehat{\mathcal{N}} (y_\mathcal{P}-z_\mathcal{P})
\right\|_{L^2(0,\infty;[H^1(\Omega)  ]')} \le \widetilde{C}
  \| z_\mathcal{P} \|_{W_\mathcal{P}(Q_\infty)}
  \| y_\mathcal{P} - z_\mathcal{P} \|_{W_\mathcal{P}(Q_\infty)}.
\end{align*}
Combining both estimates shows the assertion.
\end{proof}

\begin{theorem}\label{thm:loc_stab}
   Let $C$ and $\widetilde{C}$ denote the constants from Theorem
\ref{thm:reg_nonh} and Lemma \ref{lem:nonl_lip}, respectively. If  $\| \rho _0
\| _{L^2(\Omega)} \le \frac{3}{16 C^2 \widetilde{C}}, $ then
\begin{align*}
  \dot{z}_\mathcal{P}&= \widehat{\mathcal{A}}_{\Pi} z_\mathcal{P}  -
(\widehat{\mathcal{B}}^* \widehat{\Pi} z_\mathcal{P} )
\widehat{\mathcal{N}}_\delta z_\mathcal{P} , \quad z_\mathcal{P}(0) =
\mathcal{P}\rho_0,
\end{align*}
admits a unique solution $z_\mathcal{P} \in W_\mathcal{P}(Q_\infty)$ satisfying
 \begin{align*}
   \| z_\mathcal{P} \| _{W_\mathcal{P}(Q_\infty)} \le \frac{1}{4 C
\widetilde{C} }.
 \end{align*}
\end{theorem}

\begin{proof}
  We are going to show the assertion by a fixed point argument. For this
purpose,
consider the mapping $\mathcal{F}\colon W_\mathcal{P}(Q_\infty) \to
W_\mathcal{P}(Q_\infty),  w_\mathcal{P} \mapsto
z_{\mathcal{P},w}$ defined by
\begin{align*}
  \dot{z}_{\mathcal{P},w}&= \widehat{\mathcal{A}}_{\Pi} z_{\mathcal{P},w}  -
(\widehat{\mathcal{B}}^* \widehat{\Pi} w_\mathcal{P} )
\widehat{\mathcal{N}}_\delta w_\mathcal{P} , \quad z_{\mathcal{P},w}(0) =
\mathcal{P}\rho_0.
\end{align*}
Let $w_\mathcal{P} \in W_\mathcal{P}(Q_\infty)$ such that $\| w_\mathcal{P} \|
_{W_\mathcal{P}(Q_\infty)} \le \frac{1}{4C\widetilde{C}}.$
Lemma \ref{lem:nonl_lip} then implies that
\begin{align*}
  \left\| (\widehat{\mathcal{B}}^* \widehat{\Pi} w_\mathcal{P})
\widehat{\mathcal{N}}_{\delta} w_\mathcal{P} \right\|
_{L^2(0,\infty;[H^1(\Omega)]') }\le \frac{1}{16 C^2 \widetilde{C}}.
\end{align*}
With Theorem \ref{thm:reg_nonh} we conclude that the corresponding solution
satisfies
\begin{align*}
  \|z_{\mathcal{P},w}\|_{W_\mathcal{P}(Q_\infty)} \le C \left( \frac{1}{16C^2
\widetilde{C}} + \frac{3}{16 C^2 \widetilde{C} } \right) =
\frac{1}{4C\widetilde{C}}.
\end{align*}
Similarly, for $w_{\mathcal{P},1},w_{\mathcal{P},2} \in
W_\mathcal{P}(Q_\infty)$ with $\| w_{\mathcal{P},i} \|
_{W_\mathcal{P}(Q_\infty)} \le \frac{1}{4C\widetilde{C}}, i=1,2,$ the
associated solutions solutions $z_{\mathcal{P},w_1}$ and $z_{\mathcal{P},w_2}$
fulfill
\begin{align*}
\dot{ z}_{\mathcal{P},w_1} - \dot{ z}_{\mathcal{P},w_2} &=
\widehat{\mathcal{A}}_\Pi (z_{\mathcal{P},w_1} - z_{\mathcal{P},w_2} ) +
(\widehat{\mathcal{B}}^* \widehat{\Pi} w_{\mathcal{P},2} )
\widehat{\mathcal{N}}_\delta w_{\mathcal{P},2} - (\widehat{\mathcal{B}}^*
\widehat{\Pi} w_{\mathcal{P},1} )
\widehat{\mathcal{N}}_\delta w_{\mathcal{P},1} \\
 z_{\mathcal{P},w_1}(0) -
z_{\mathcal{P},w_2}(0) &= 0.
\end{align*}
Hence, Theorem \ref{thm:reg_nonh} yields
\begin{align*}
   \|z_{\mathcal{P},w_1}-z_{\mathcal{P},w_2}\|_{W_\mathcal{P}(Q_\infty)} \le C
\|(\widehat{\mathcal{B}}^*
\widehat{\Pi} w_{\mathcal{P},2} )  \widehat{\mathcal{N}}_\delta
w_{\mathcal{P},2} - (\widehat{\mathcal{B}}^*
\widehat{\Pi} w_{\mathcal{P},1} )
\widehat{\mathcal{N}}_\delta w_{\mathcal{P},1}\|_{L^2(0,\infty;[H^1(\Omega)]')}.
\end{align*}
Moreover, with Lemma \ref{thm:reg_nonh}, we obtain that
\begin{align*}
 \|z_{\mathcal{P},w_1}-z_{\mathcal{P},w_2}\|_{W_\mathcal{P}(Q_\infty)} \le 2 C
\widetilde{C} \left( \frac{1}{4C\widetilde{C}}\right)^2 = \frac{1}{2}
\frac{1}{4 C \widetilde{C} }.
\end{align*}
In other words, the mapping $\mathcal{F}$ is a contraction in the set
\begin{align*}
  \left\{ z_\mathcal{P} \in W_{\mathcal{P}}(Q_\infty) \left| \| z_\mathcal{P}
\|_{W_\mathcal{P}(Q_\infty)} \le \frac{1}{4C \widetilde{C}} \right. \right\}
\end{align*}
and the statement is shown.
\end{proof}

As a consequence of Theorem \ref{thm:loc_stab}, we have that $e^{\delta t}
y_\mathcal{P}\in W_\mathcal{P}(Q_\infty)$ implying that there exists a constant
$C$ such that $\| y_\mathcal{P} \| _{L^2(\Omega)} \le C e^{-\delta t} \|
\rho_0 \|_{L^2(\Omega)}.$

\section{A Lyapunov based feedback law}
\label{sec:lyap}

As an alternative to the Riccati based approach, in this section, we
propose a feedback law that allows to construct a \textit{global} Lyapunov
function for the nonlinear closed loop system. The idea is inspired by the
observations found in \cite{BalS79} for hyperbolic systems.

With the previously introduced notation, assume that $(\lambda_2,\psi_2)$
denotes the eigenpair of $\widehat{\mathcal{A}}$ associated to the first
nonzero eigenvalue. Hence, $\lambda_2$ determines the exponential decay rate
of the uncontrolled systems. Instead of using \eqref{eq:elliptic_alpha}, let
us determine
the control shape function $\alpha$ as a
solution to the elliptic equation
 \begin{equation}\label{eq:elliptic_alpha_Lyap}
\begin{aligned}
  \nabla \cdot (\rho_\infty \nabla \alpha) &=  \psi_2 && \text{in } \Omega, \\
(\rho_\infty \nabla \alpha )\cdot \vec{n} &= 0  && \text{on } \Gamma.
\end{aligned}
\end{equation}
As a consequence, this choice of $\alpha$ yields
$\widehat{\mathcal {B}}=\widehat{\mathcal{N}} \rho_\infty = \psi_2.$ Let
further $\mu >0 $ be chosen such that
\begin{equation}\label{eq:mu_resolvent}
\begin{aligned}
  \langle (\mu I - \widehat{\mathcal{A}}) y_\mathcal{P} , y_\mathcal{P} \rangle
_{L^2(\Omega)} &\ge \langle y_\mathcal{P}, y_\mathcal{P} \rangle_{H^1(\Omega)},
 \quad \text{for all } y_\mathcal{P} \in \mathcal{D}(\widehat{\mathcal{A}}).
\end{aligned}
\end{equation}
Since $\widehat{\mathcal{A}}$ generates an exponentially stable semigroup, it
is well-known \cite[Theorem 4.1.23]{CurZ95} that there exists a unique
self-adjoint nonnegative solution  $\Upsilon$ to the Lyapunov equation for
$y_\mathcal{P},z_\mathcal{P} \in \mathcal{D}(\widehat{\mathcal{A}})\colon$
 \begin{align}\label{eq:Lyap_eq}
   \langle \Upsilon y_\mathcal{P},\widehat{\mathcal {A}} z_\mathcal{P} \rangle
+ \langle  \widehat{\mathcal{A}}y_\mathcal{P},\Upsilon z_\mathcal{P} \rangle =
- 2\mu \langle y_\mathcal{P},z_\mathcal{P} \rangle .
 \end{align}
We then obtain the following result.
\begin{theorem}\label{thm:glob_stab}
  Let $\mu $ and $\Upsilon$ be as in \eqref{eq:mu_resolvent} and
\eqref{eq:Lyap_eq}, respectively. Consider the system
\begin{align}\label{eq:thm_glob_stab}
  \dot{y}_ \mathcal{P} &=  \widehat{\mathcal{A}} y_ \mathcal{P}  +
u\widehat{\mathcal{N}}y_ \mathcal{P}   +
 \widehat{\mathcal{B}}
u , \quad
y_ \mathcal{P}(0)= \mathcal{P}\rho_0.
\end{align}
where the control $u$ is defined by the feedback law
\begin{align*}
  u = -\langle \widehat{\mathcal{B}} + \widehat{\mathcal{N}} y_\mathcal{P},
\Upsilon y_\mathcal{P} + y_\mathcal{P}\rangle .
\end{align*}
Then the function $V(y_\mathcal{P}):= \langle y_\mathcal{P}, \Upsilon
y_\mathcal{P} + y_\mathcal{P} \rangle$ is a global Lyapunov function for
\eqref{eq:thm_glob_stab}.
\end{theorem}
\begin{proof}
Since $\Upsilon$ is self-adjoint and nonnegative, it obviously holds that
$V(y_\mathcal{P}) \ge \| y_\mathcal{P} \| ^2.$ Moreover, we obtain that
\begin{align*}
  \frac{\mathrm{d}}{\mathrm{d}t} V(y_\mathcal{P}) &= \langle
\widehat{\mathcal{A}}y_\mathcal{P},\Upsilon y_\mathcal{P} + y_\mathcal{P}
\rangle + \langle y_\mathcal{P}, \Upsilon  \widehat{\mathcal{A}}
y_\mathcal{P}+  \widehat{\mathcal{A}}
y_\mathcal{P} \rangle \\
&\quad - \langle \widehat{\mathcal{B}}+\widehat{\mathcal{N}}y_\mathcal{P},
\Upsilon y_\mathcal{P} + y_\mathcal{P} \rangle \langle \widehat{\mathcal{B}}+
\widehat{\mathcal{N}}  y_\mathcal{P}, \Upsilon y_\mathcal{P} + y_\mathcal{P}
\rangle \\
& \quad -\langle \widehat{\mathcal{B}}+\widehat{\mathcal{N}}y_\mathcal{P},
\Upsilon y_\mathcal{P} + y_\mathcal{P} \rangle \langle y_\mathcal{P}
,\Upsilon(\widehat{\mathcal{N}} y_\mathcal{P}+\widehat{\mathcal{B}}) +
(\widehat{\mathcal{N}} y_\mathcal{P}+\widehat{\mathcal{B}}) \rangle \\
&= -2\mu \langle y_\mathcal{P},y_\mathcal{P} \rangle + \langle
\widehat{\mathcal{A}} y_\mathcal{P},y_\mathcal{P} \rangle+ \langle
y_\mathcal{P}, \widehat{\mathcal{A}} y_\mathcal{P}\rangle \\
&\quad - 2 \langle \widehat{\mathcal{B}}+\widehat{\mathcal{N}}y_\mathcal{P},
\Upsilon y_\mathcal{P} + y_\mathcal{P} \rangle^2 \\
&\le 2\langle (\widehat{\mathcal{A}}-\mu I) y_\mathcal{P},y_\mathcal{P} \rangle
\le -2 \langle y_\mathcal{P}, y_\mathcal{P} \rangle _{H^1(\Omega)}
\end{align*}
which shows the assertion.
\end{proof}

In addition to the previous result, the feedback law locally increases the
exponential decay rate.
\begin{theorem}
  Let $\lambda_i,i=2,3,\dots$ denote the eigenvalues of the operator
$\widehat{\mathcal{A}}.$ Assume that
\begin{align*}
\widetilde{\lambda}_2:=\lambda_2 - \| \psi_2\|^2 +
\frac{\mu}{\lambda_2} \| \psi_2 \|^2 \neq \lambda_j, \quad j=3,\dots
\end{align*}
Then for the spectrum of the linearized closed loop operator it holds that
\begin{align*}
 \sigma (\widehat{\mathcal{A}} -
\widehat{\mathcal{B}}\widehat{\mathcal{B}}^*\Upsilon
-\widehat{\mathcal{B}}\widehat{\mathcal{B}}^* ) = \{\widetilde{\lambda}_2\}
\cup \{\lambda_j\}, \quad j\ge 3.
\end{align*}
\end{theorem}

\begin{proof}
  Due to \eqref{eq:Lyap_eq}, we find that
  \begin{align*}
    \langle \Upsilon \psi_2 , \widehat{\mathcal{A}} \psi_2 \rangle + \langle
\widehat{\mathcal{A}} \psi_2, \Upsilon \psi_2 \rangle  = - 2\mu \| \psi_2 \|^2 .
  \end{align*}
Since $\psi_2$ is an eigenfunction of $\widehat{\mathcal{A}},$ this implies that
\begin{align*}
  \langle \Upsilon \psi_2 ,\psi_2 \rangle = -\frac{\mu}{\lambda_2} \|
\psi_2\|^2.
\end{align*}
Further, from our choice of $\alpha,$ we already know that
$\widehat{\mathcal{B}} = \psi_2.$ Hence, it follows that
\begin{align*}
  (\widehat{\mathcal{A}} - \widehat{\mathcal{B}} \widehat{\mathcal{B}}^*
\Upsilon - \widehat{\mathcal{B}} \widehat{\mathcal{B}}^*) \psi_2 = \lambda_2
\psi_2 - \langle \Upsilon\psi_2, \psi_2 \rangle \psi_2 - \|
\psi_2\|^2 \psi_2
\end{align*}
which shows the first part. For $\beta_j:= \frac{\langle \Upsilon
\psi_j+\psi_j,\psi _2 \rangle }{\widetilde{\lambda}_2-\lambda_j},j=3,\dots$ we
further arrive at
\begin{align*}
 &(\widehat{\mathcal{A}}-\widehat{\mathcal{B}}\widehat{\mathcal{B}}^*
\Upsilon -\widehat{\mathcal{B}}\widehat{\mathcal{B}}^* ) (\psi _j
+ \beta_j \psi_2 )  = \beta_j \widetilde{\lambda}_2 \psi_2 +
\widehat{\mathcal{A}}\psi
_j  - \widehat{\mathcal{B}}\widehat{\mathcal{B}}^* \Upsilon \psi_j
- \widehat{\mathcal{B}}\widehat{\mathcal{B}}^*  \psi_j
\\
&\quad =\beta_j \widetilde{\lambda}_2 \psi_2 + \lambda_j\psi_j - \psi_2 \langle
\Upsilon \psi_j,\psi_2\rangle- \psi_2 \langle
 \psi_j,\psi_2\rangle
= \lambda_j(\psi_j+\beta_j \psi_2) .
\end{align*}
This shows the claim.
\end{proof}

\begin{remark}
  Let us emphasize that the feedback law is particularly useful
in cases where $\lambda_2$ is close to the imaginary axis and there is a gap
between $\lambda_2$ and $\lambda_3.$ Indeed, for $\lambda_2 \to 0,$ the term
$\frac{\mu}{\lambda_2} \to -\infty,$ such that the modified eigenvalue
$\widetilde{\lambda}_2$ is moved far away from the imaginary axis.
\end{remark}

\section{Numerical study - A two dimensional double well potential}

\newlength\fheight
\newlength\fwidth
\setlength\fheight{0.6\linewidth}
\setlength\fwidth{0.8\linewidth}

As a numerical example, we consider
\begin{equation}\label{eq:FPE_numex}
\begin{aligned}
  \frac{\partial \rho}{\partial t} &= \nu \Delta \rho + \nabla \cdot (\rho
\nabla G) + u \nabla \cdot (\rho \nabla\alpha) && \text{in } \Omega \times
(0,\infty), \\
0&=(\nu \nabla \rho + \rho \nabla V) \cdot \vec{n}  && \text{on } \Gamma \times
(0,\infty),\\
\rho(x,0)&=\rho_0(x)  && \text{in } \Omega,
\end{aligned}
\end{equation}
on $\Omega = (-1.5,1.5)\times (-1,1) \subset \IR^2,$ with $\nu=1$ and a two
dimensional double well potential of the form
$$G(x) = 3  (x_1^2-1)^2+6x_2^2.$$
For the spatial semidiscretization, a finite difference scheme with $k=
n_{x_1}\cdot n_{x_2} = 96 \cdot 64= 6144$ degrees of freedom was implemented.
The discretization $A \in \IR^{k\times k}$ of the operator $\mathcal{A}$ defined
as in \eqref{eq:A_N_op} was obtained by first discretizing the operator
$\mathcal{A}^*$ as given by \eqref{eq:A_N_adj_op} and then taking the transpose
of the resulting matrix. The reason for this indirect approach was that the
discretization of $\mathcal{A}^*$ only required the incorporation of
``standard'' Neumann boundary conditions rather than the mixed boundary
conditions arising for $\mathcal{A}.$ Due to the convective terms included in
$\mathcal{A}$ and $\mathcal{A}^*,$ a first order upwind scheme was utilized.
Let us emphasize that even for the value $\nu=1,$ this turned out to be
essential for the accuracy of the discretization. We also mention the
possibility of using more advanced discretization schemes that have been
proposed in the context of the Fokker-Planck equation, see, e.g.,
\cite{AnnB13,CooC70}. However, the finite
difference scheme lead to accurate approximations of the stationary
distribution and the preservation of probability was ensured up to machine
precision in all our numerical results. Figure \ref{fig:pot_stat} now shows the
discretization of the double well potential as well as the corresponding
(spatially discrete) stationary distribution $\rho_\infty^k.$
\begin{figure}[htb]
  \begin{subfigure}[c]{0.49\textwidth}
    \includegraphics[scale=0.45]{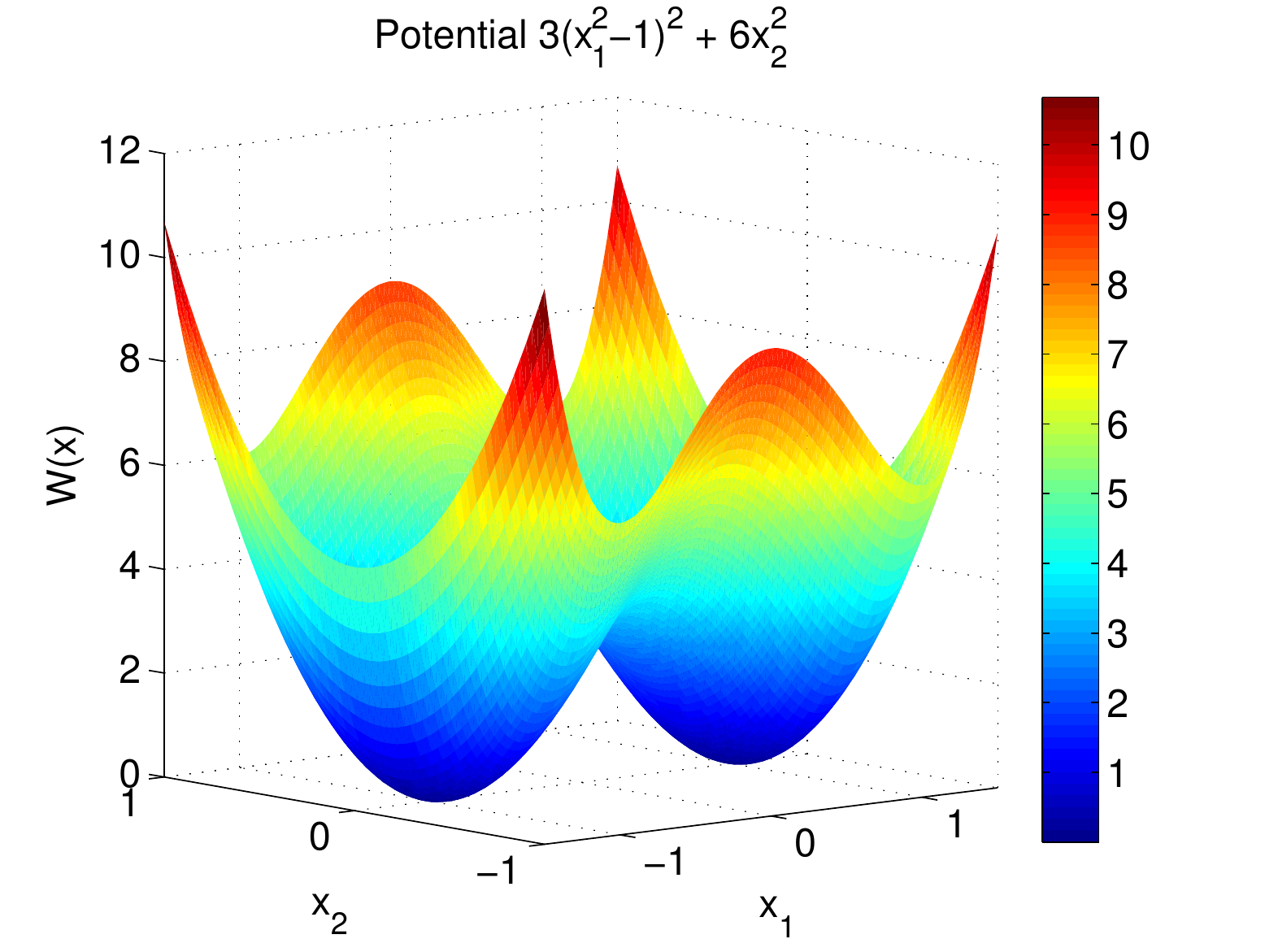}
    \end{subfigure}\begin{subfigure}[c]{0.49\textwidth}
    \includegraphics[scale=0.45]{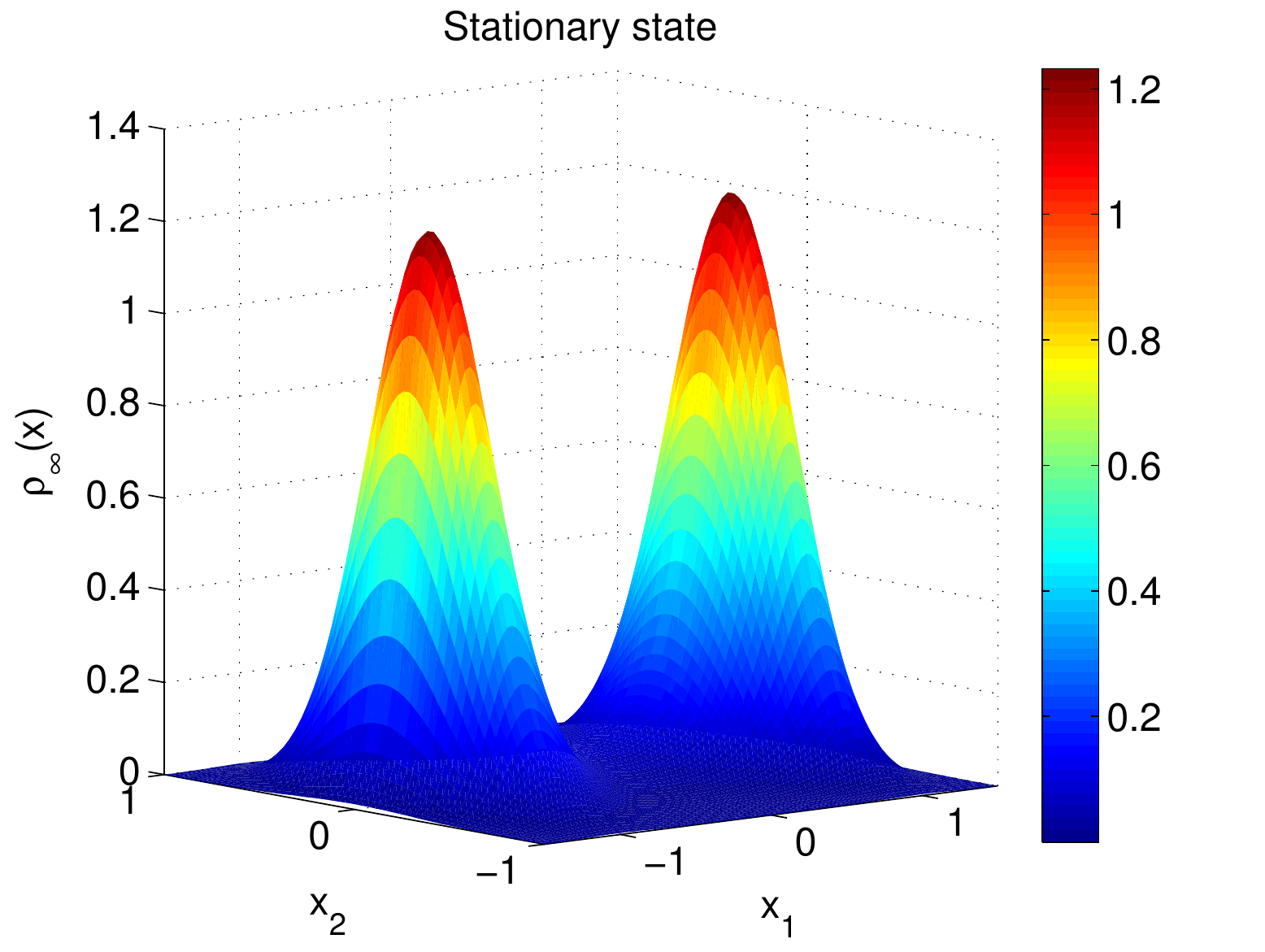}
    \end{subfigure}
  \caption{Confining double well potential (left) and associated stationary
state (right).}
  \label{fig:pot_stat}
    \end{figure}

For both the Riccati-based and the Lyapunov-based control strategy, the
discrete control operators $N$ and $B=N\rho_\infty^k$ were derived based on the
solutions $\alpha(x)$ to \eqref{eq:elliptic_alpha} and
\eqref{eq:elliptic_alpha_Lyap}. To be more precise, first, the involved
elliptic equations were also discretized by a finite difference scheme which,
due to the Neumann boundary conditions, lead to matrices $C$ with a zero
eigenvalue. The individual spatially discrete shape functions $\alpha^k$ were
obtained by utilizing the Moore-Penrose pseudoinverse of the $C$ matrices.
Finally, with the resulting $\alpha^k,$ the matrices $N$ were generated by the
discretization of the operator $\mathcal{N}$ defined in \eqref{eq:A_N_op}. For
the Riccati-based approach, we incorporated the eigenfunctions to the first
three nonzero eigenvalues into \eqref{eq:elliptic_alpha}, i.e., we set $d=4.$
Since varying the value $d$ lead to qualitatively similar behavior, we only
report on the results for the special case $d=4.$ Due to the Hautus criterion,
it was thus possible to solve the associated Riccati equation  with $\delta
\approx 12.26.$ The
corresponding control shape functions for the Riccati-based (left) and the
Lyapunov-based (right) approach are given
in Figure \ref{fig:control_shape}.
\begin{figure}[htb]
  \begin{subfigure}[c]{0.33\textwidth}
    \includegraphics[scale=0.29]{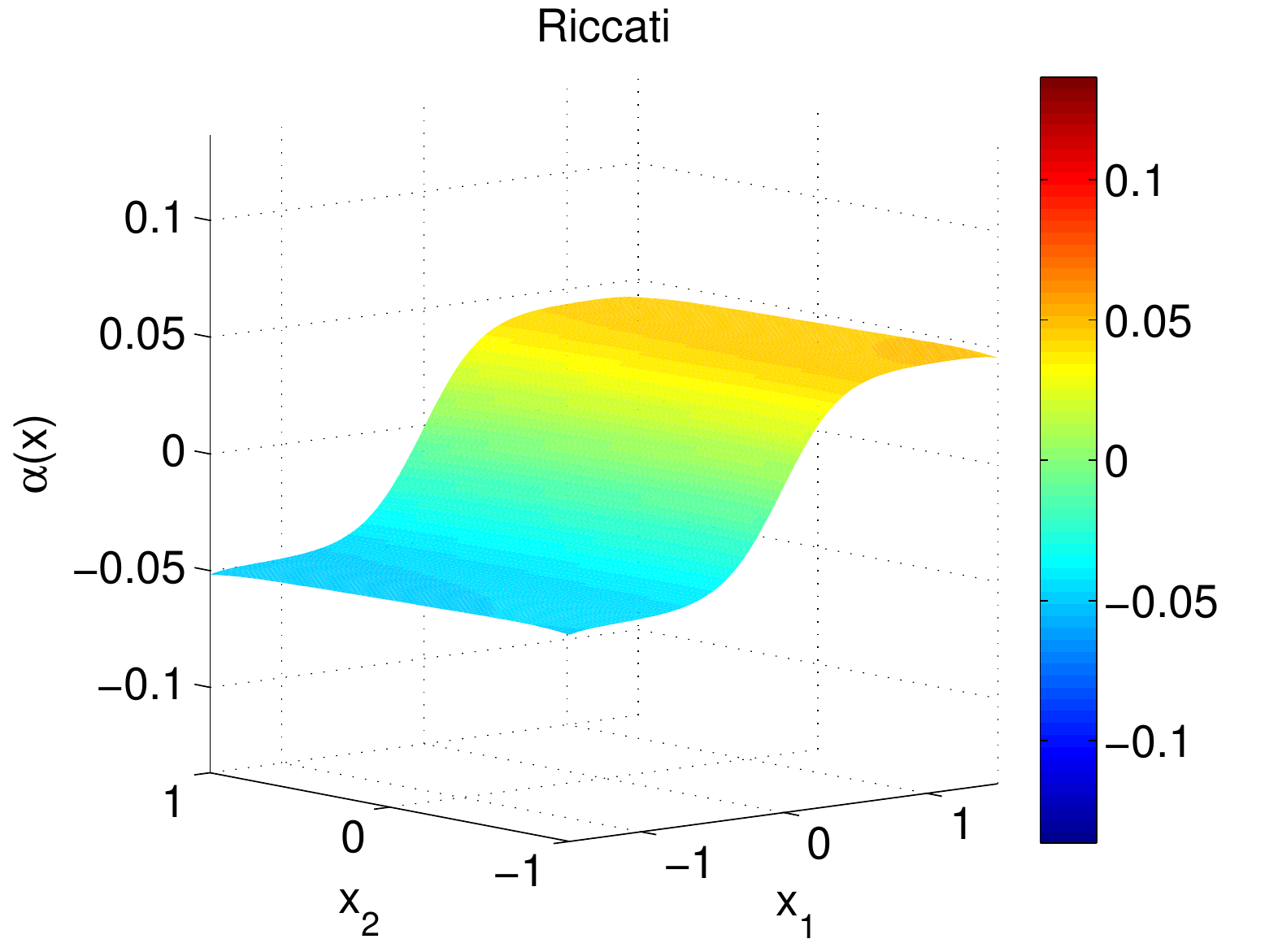}
    \end{subfigure}\begin{subfigure}[c]{0.33\textwidth}
    \includegraphics[scale=0.29]{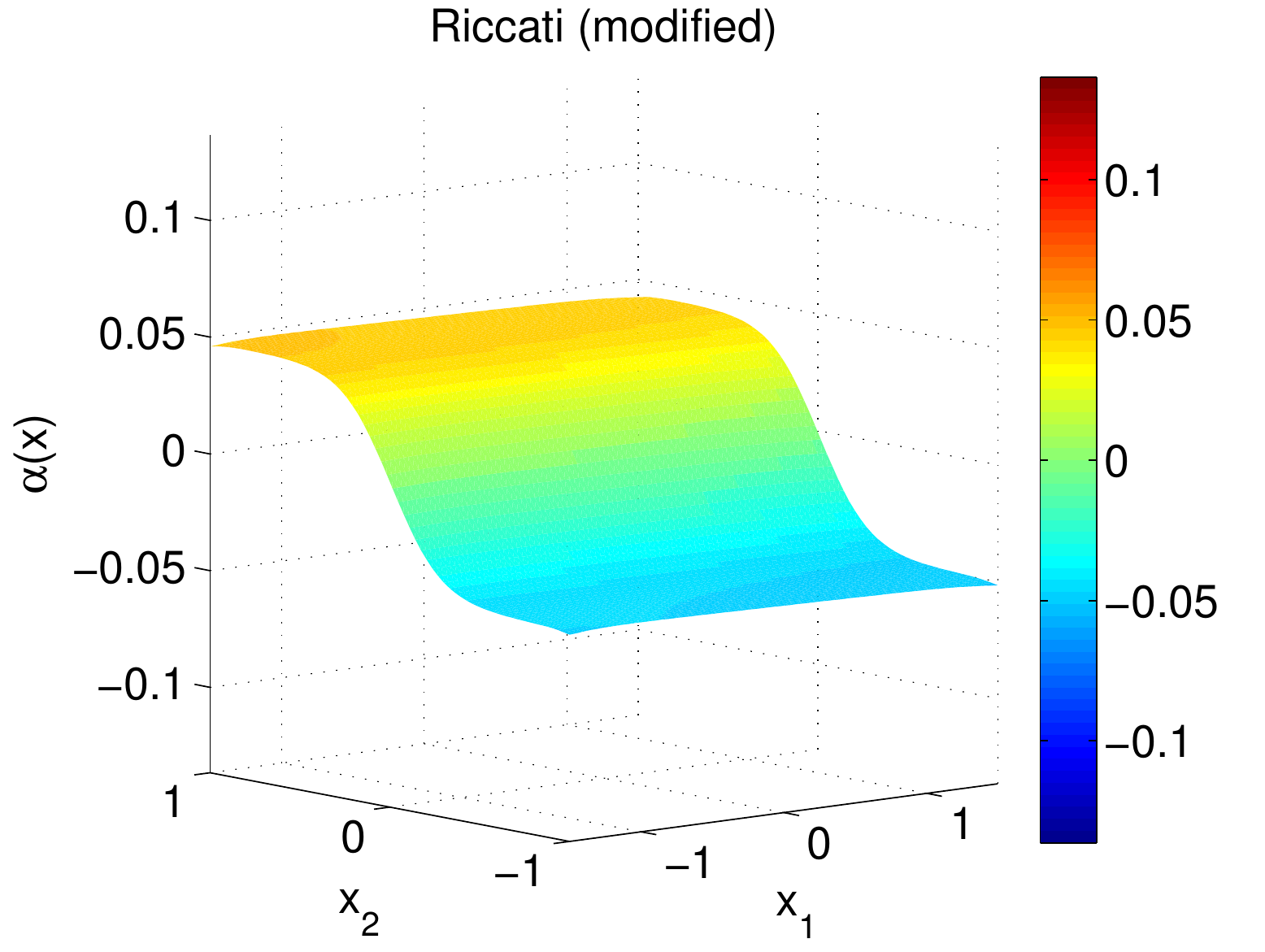}
    \end{subfigure}\begin{subfigure}[c]{0.33\textwidth}
    \includegraphics[scale=0.29]{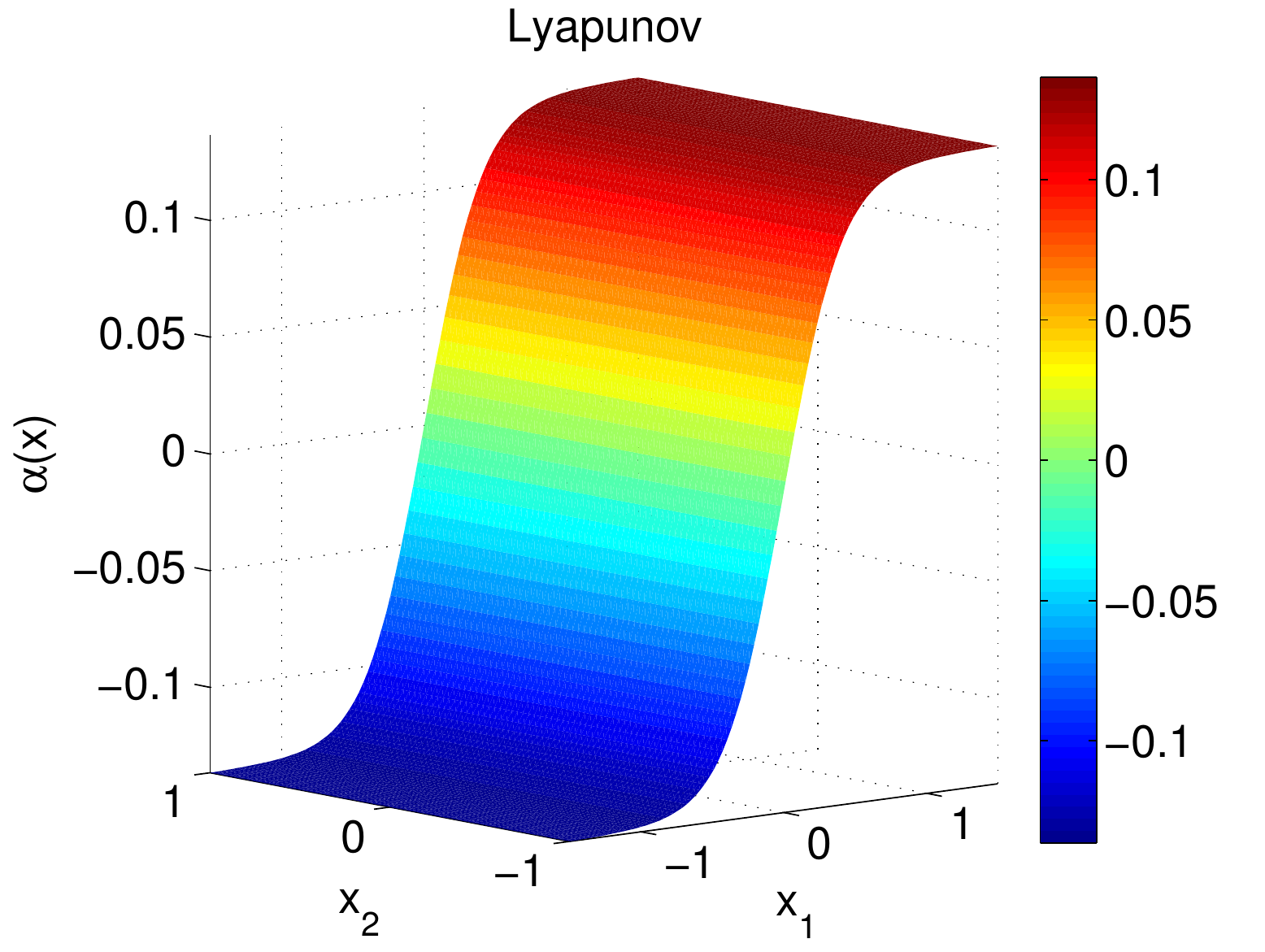}
    \end{subfigure}
  \caption{Control shape functions for different approaches.}
  \label{fig:control_shape}
    \end{figure}
In order to underline the benefit of using a ``specific'' $\alpha$ rather
than an ``arbitrary'' one, we also report on some results we obtained by
rotating the control shape function $\alpha$ (see Figure
\ref{fig:control_shape} center) while still using a Riccati-based feedback
law obtained from the linearized system.

All simulations were generated on an \intel Xeon(R) CPU E31270 @ 3.40 GHz x 8,
16 GB RAM, Ubuntu Linux 14.04, \matlab \;Version 8.0.0.783
(R2012b) 64-bit (glnxa64). The solutions of the ODE systems were always obtained
by the \matlab\;routine \texttt{ode23}. For solving the Riccati and Lyapunov
equations, we used the \matlab\;routine \texttt{care} and
\texttt{lyap}, respectively, and the technique presented below.

\subsection{Solving the Riccati equation}
\label{subsec:num_ric}

Based on the discretization scheme described above, let us at this point assume
that $A \in \IR^{k \times k}$,
$B \in \IR^{k \times 1}$, $M \in \IR^{k \times k}$, $\rho_\infty \in
\IR^{k}$ are given and satisfy:
\begin{equation*}
A \rho_\infty= 0, \quad
A^\top \mathbbm{1}= 0, \quad
B^\top \mathbbm{1}= 0, \quad
M^\top= M, \quad
\langle \mathbbm{1}, \rho_\infty \rangle= 1,
\end{equation*}
where $\mathbbm{1} = h_ {x_1}\cdot h_{x_2}
\begin{pmatrix}1,\dots,1\end{pmatrix}^\top$ and $h_x,h_y$ denote the mesh size.
We denote by $P$ the projection on $\mathbbm{1}^\perp$ along $\IR \rho_\infty$:
$P= I_k - \rho_\infty \mathbbm{1}^\top$. We denote by $(e_i)_{i=1,\dots,k}$ the
vectors of the canonical basis.
We aim at solving the following discretized Riccati equation:
\begin{equation} \label{eq:discretizedRiccati}
(A^\top + \delta P^\top) \Pi + \Pi (A+ \delta P) - \Pi B B^\top \Pi + P^\top M P= 0,
\quad
\Pi \rho_\infty= 0, \quad
\Pi^\top= \Pi.
\end{equation}
Let $R \in \IR^{k \times k}$ be a regular matrix satisfying:
\begin{equation*}
Re_k= \rho_\infty \quad \text{and} \quad
R^\top \mathbbm{1}= e_k.
\end{equation*}
Note that the condition $R^\top \mathbbm{1}= e_k$ is equivalent to: $\forall
i=1,\dots,k-1,\ R e_i \in \mathbbm{1}^\perp$.
An example of matrix $R$ is given by:
\begin{equation*}
R= \left( \begin{array}{cccc}
1 & &  & \rho_{\infty,1} \\
 & \ddots & & \vdots \\
 & & 1 & \rho_{\infty,k-1} \\
-1 & \hdots & -1 & \rho_{\infty,k}
\end{array}
\right).
\end{equation*}
Note that:
\begin{equation*}
R^{-1}= \left( \begin{array}{cccc}
1 & & & 0 \\
 & \ddots & & \vdots \\
 & & 1 & 0 \\
1 & \hdots & 1 & 1
\end{array} \right)
-
\left( \begin{array}{ccc}
\rho_{\infty,1} & \hdots & \rho_{\infty,1} \\
\vdots & \vdots & \vdots \\
\rho_{\infty,k-1} & \hdots & \rho_{\infty,k-1} \\
0 & \hdots & 0
\end{array} \right).
\end{equation*}
We also introduce: $Q = \left( \begin{array}{c} I_{k-1} \\ 0 \end{array}
\right)$.
Consider the reduced and discretized Riccati equation (in $\IR^{(k-1) \times
(k-1)}$):
\begin{equation} \label{eq:discretizedreducedRiccati}
(\widehat{A}^\top + \delta I_{k-1}) \widehat{\Pi} + \widehat{\Pi} (\widehat{A}+
\delta I_{k-1}) - \widehat{\Pi} \widehat{B} \widehat{B}^\top \widehat{\Pi} +
\widehat{M} = 0,\quad
\widehat{\Pi}^\top = \widehat{\Pi}, \quad
\end{equation}
where $\widehat{A}= Q^\top R^{-1} A R Q$, $\widehat{B}= Q^\top R^{-1} B$,
$\widehat{M}= Q^\top R^\top P^\top M P R Q$.

\begin{lemma}
Let $\Pi \in \IR^{k \times k}$. The matrix $\Pi$ is a solution to
\eqref{eq:discretizedRiccati} if and only if there exists a solution
$\widehat{\Pi}$ to \eqref{eq:discretizedreducedRiccati} such that $\Pi=
R^{-\top} \left( \begin{array}{cc} \widehat{\Pi} & 0 \\ 0 & 0 \end{array}
\right) R^{-1}$.
\end{lemma}

\begin{proof}
The proof is similar to the proof of Lemma \ref{lem:reductionRiccati}. Observe
that $\Pi$ is a solution to \eqref{eq:discretizedRiccati} if and only if
$\widetilde{\Pi}= R^\top \Pi R$ is a solution to
\begin{equation} \label{eq:rotatedRiccati}
(\widetilde{A}^\top + \delta \widetilde{P}^\top) \widetilde{\Pi} +
\widetilde{\Pi} (\widetilde{A} + \delta \widetilde{P}) - \widetilde{\Pi}
(\widetilde{B} \widetilde{B}^\top) \widetilde{\Pi} + \widetilde{M}= 0, \quad
\widetilde{\Pi} e_N= 0, \quad
\widetilde{\Pi}^\top = \widetilde{\Pi}
\end{equation}
where: $\widetilde{A}= R^{-1} A R$, $\widetilde{P}= R^{-1} P R$, $\widetilde{B}=
R^{-1} B$, $\widetilde{M}= R^\top P^\top M P R$. One can easily check that the last row
and the last column of the following matrices are null: $\widetilde{A}$,
$\widetilde{P}$, $\widetilde{B}\widetilde{B}^\top$, $\widetilde{M}$. Moreover,
the upper left block of $\widetilde{P}$ is $I_{k-1}$. The equivalence follows
directly from a block decomposition of equation \eqref{eq:rotatedRiccati}.
\end{proof}

\begin{remark}
Let us emphasize that computing the solution $\widehat{\Pi}$ to
\eqref{eq:discretizedreducedRiccati} is a challenging task already in the case
when $\Omega \subset \mathbb R^{n}$ with $n=2,3,$ respectively, in particular
because the matrices defining the reduced Riccati equation
\eqref{eq:discretizedreducedRiccati} are dense. On the other hand, according to
Lemma \ref{lem:spec_A} the only accumulation point of the spectrum of
$\mathcal{A}$ is $-\infty.$ Thus, as a perspective for future developments
geared at considering  control of the
Fokker-Planck equation in higher dimensions, it
is of interest to only $\delta$-stabilize the part of the spectrum
that is closest to the imaginary axis. This way, the resolution of a Riccati
equation of large dimension can be avoided at almost no loss of performance.
The idea goes back (at least) to \cite{Tri75} and is also studied in
\cite{RayT10} and the references therein.
A detailed discussion
together with an implementation tailored to the special structure of the
Fokker-Planck equation is currently being investigated. As an alternative way
for reducing the complexity we also mention specific model reduction approaches
as considered in \cite{Har11,HarST13}.
\end{remark}

\subsection{A random initial state}

The first test case is concerned with the evolution of the uncontrolled and
controlled systems for a random initial state $\rho_0^k$
(\texttt{rand($k$)}). The temporal evolution of the deviation of the
state $\rho(t)$ from the stationary distribution $\rho_\infty^k$ with respect
to the $L^2(\Omega)$-norm is shown in Figure \ref{fig:L2_rand}.
 \begin{figure}[htb]
\begin{center}
   \includegraphics{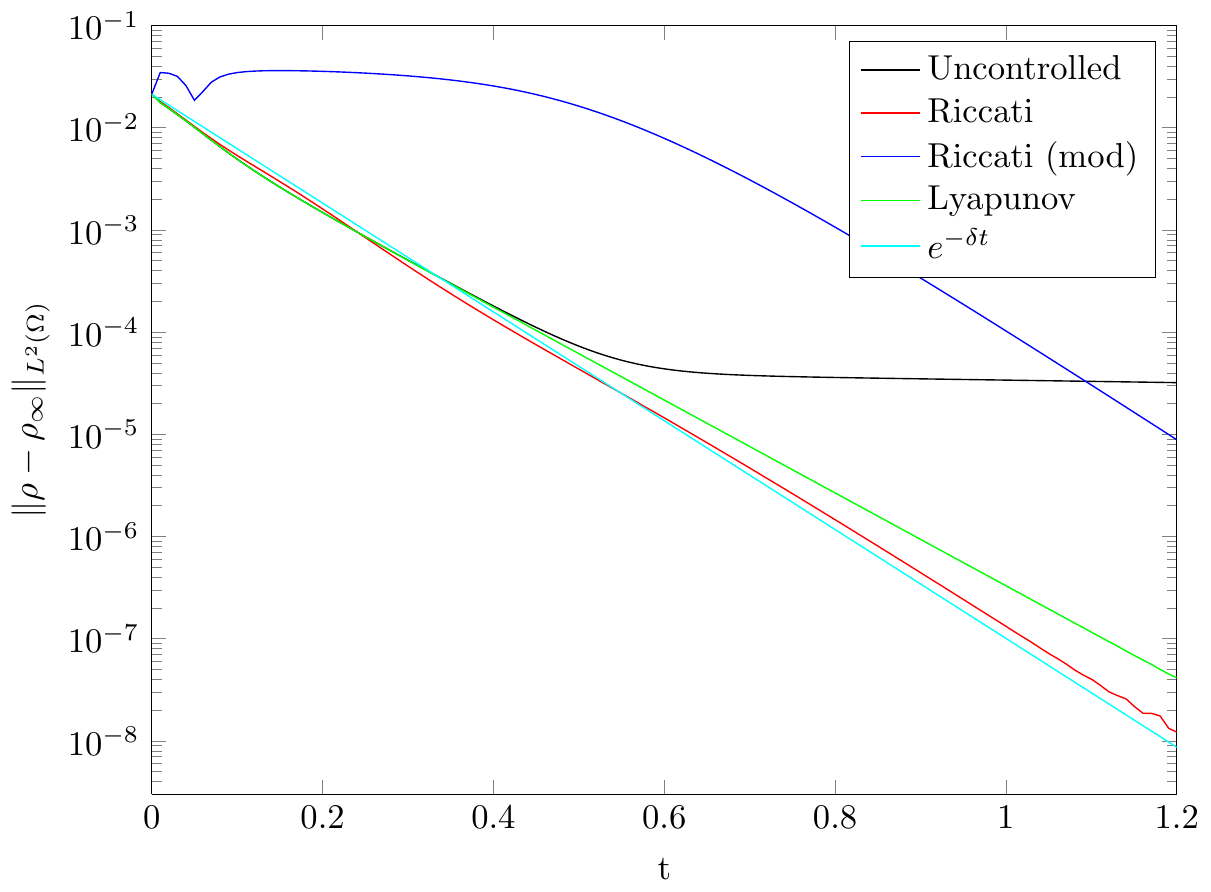}
  \caption{Comparison of $L^2(\Omega)$-norm evolution}
  \label{fig:L2_rand}
  \end{center}
\end{figure}
In addition to the dynamics of the systems, we also visualized the exponential
decay rate $\delta$ that one would expect from solving the
 Riccati equation discussed in Subsection \ref{subsec:num_ric}. Some comments
are in order. It can be seen that in the beginning, the uncontrolled system
approaches the stationary distribution as fast as the controlled systems. After
some time, however, the convergence rate becomes significantly slower. For the
controlled solutions, let us point out that there is almost no visible
difference between the Lyapunov-based approach and the Riccat-based approach.
On the other hand, with the rotated control shape function $\alpha,$ the
performance is clearly worse. In fact, in this case, the controlled dynamics
converge slower than for the uncontrolled case.
\begin{figure}[htb]
  \begin{subfigure}[c]{0.24\textwidth}
    \includegraphics[scale=0.22]{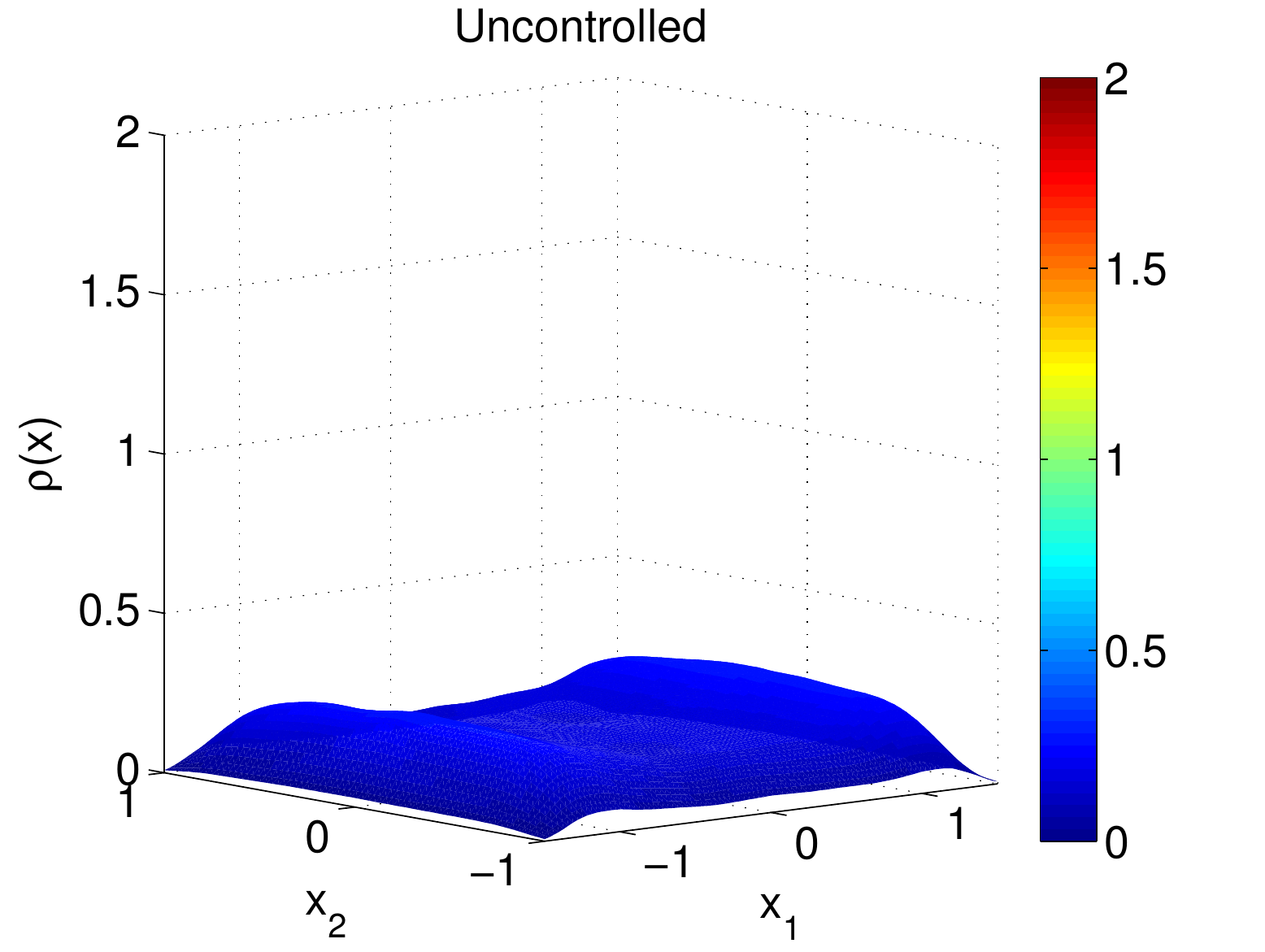}
    \caption{$t=0.01$.}
    \end{subfigure}\begin{subfigure}[c]{0.24\textwidth}
    \includegraphics[scale=0.22]{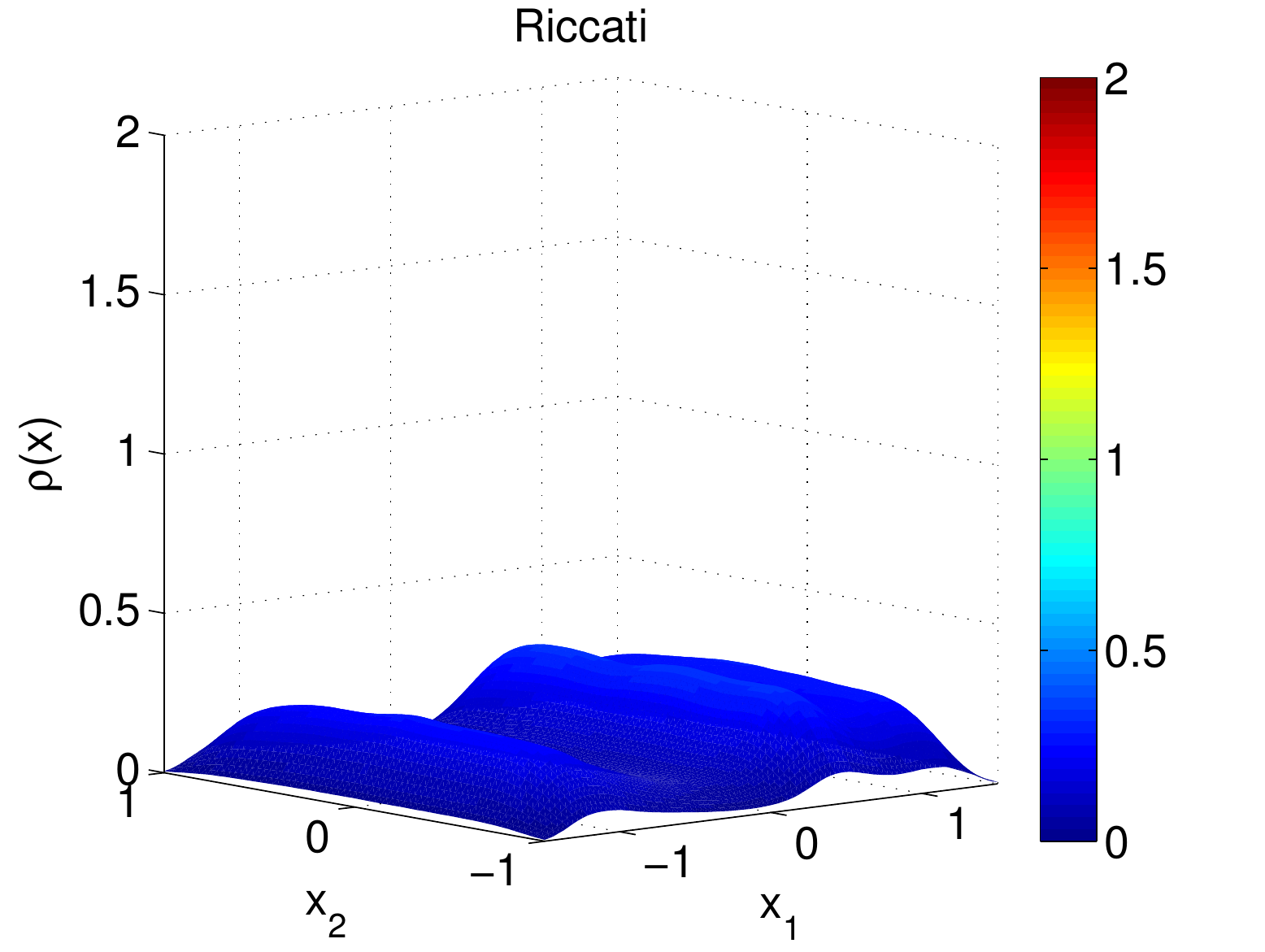}
    \caption{$t=0.01$.}
    \end{subfigure}\begin{subfigure}[c]{0.24\textwidth}
    \includegraphics[scale=0.22]{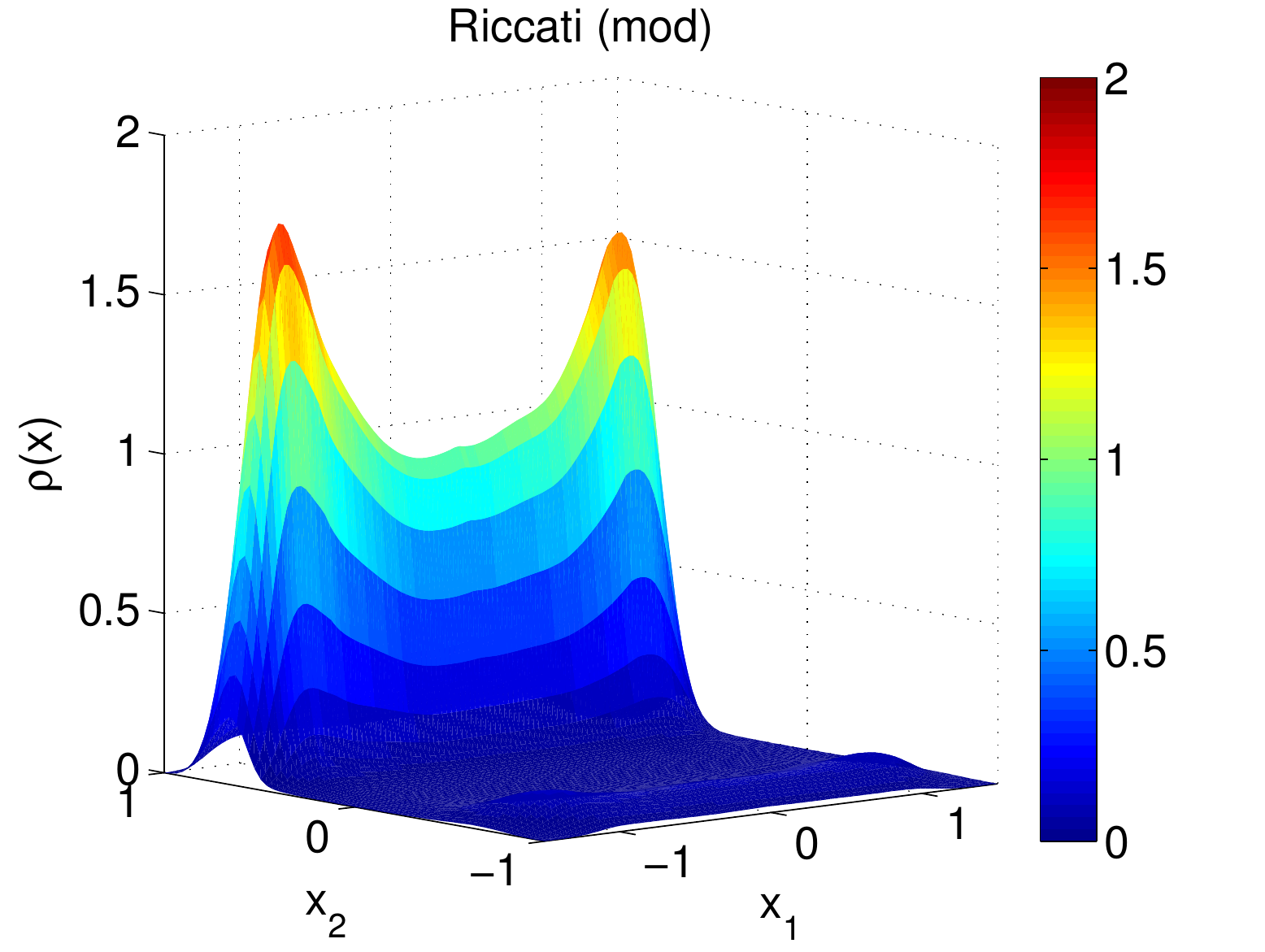}
    \caption{$t=0.01$.}
    \end{subfigure}\begin{subfigure}[c]{0.24\textwidth}
    \includegraphics[scale=0.22]{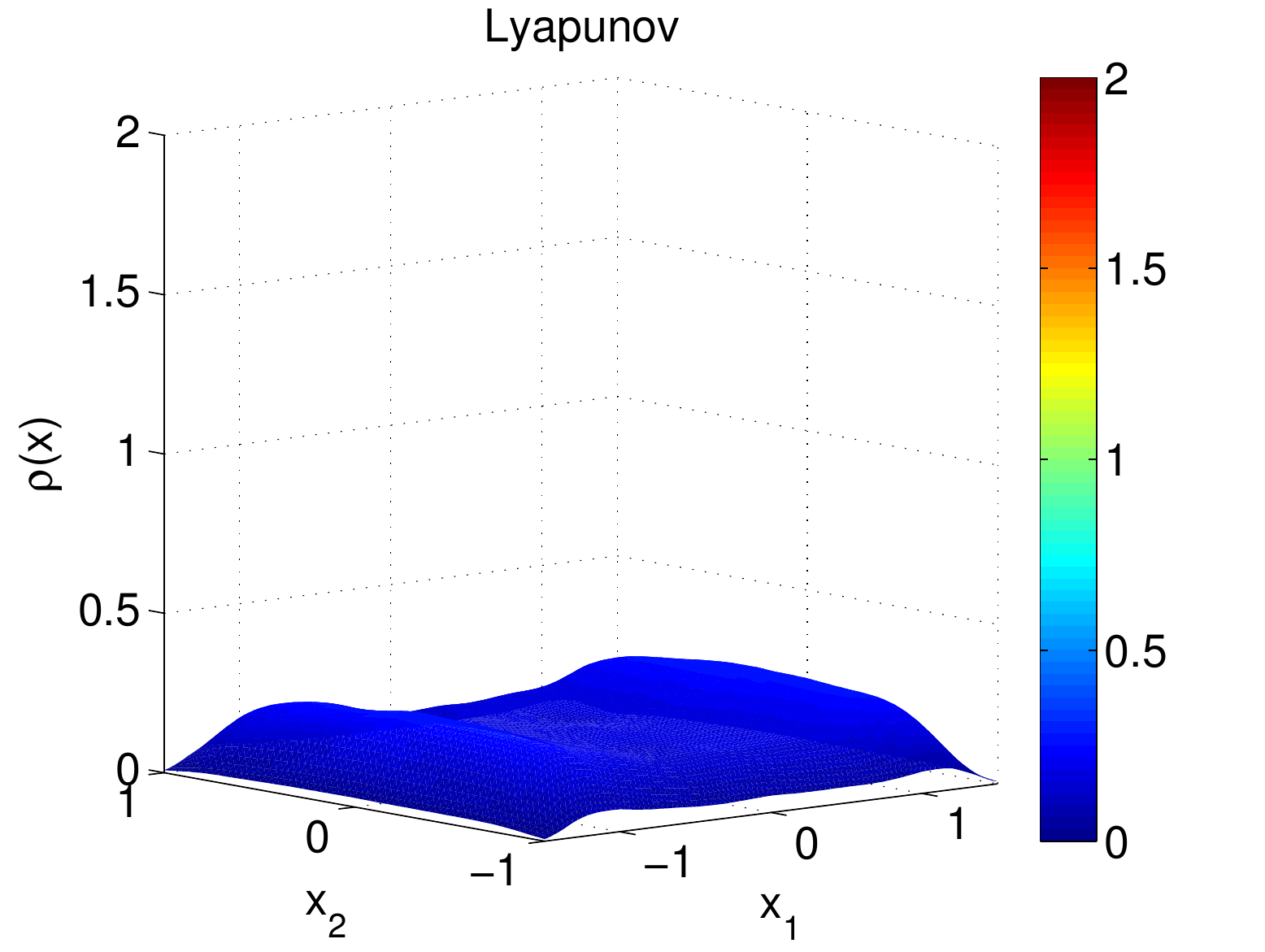}
    \caption{$t=0.01$.}
    \end{subfigure}\\ \begin{subfigure}[c]{0.24\textwidth}
    \includegraphics[scale=0.22]{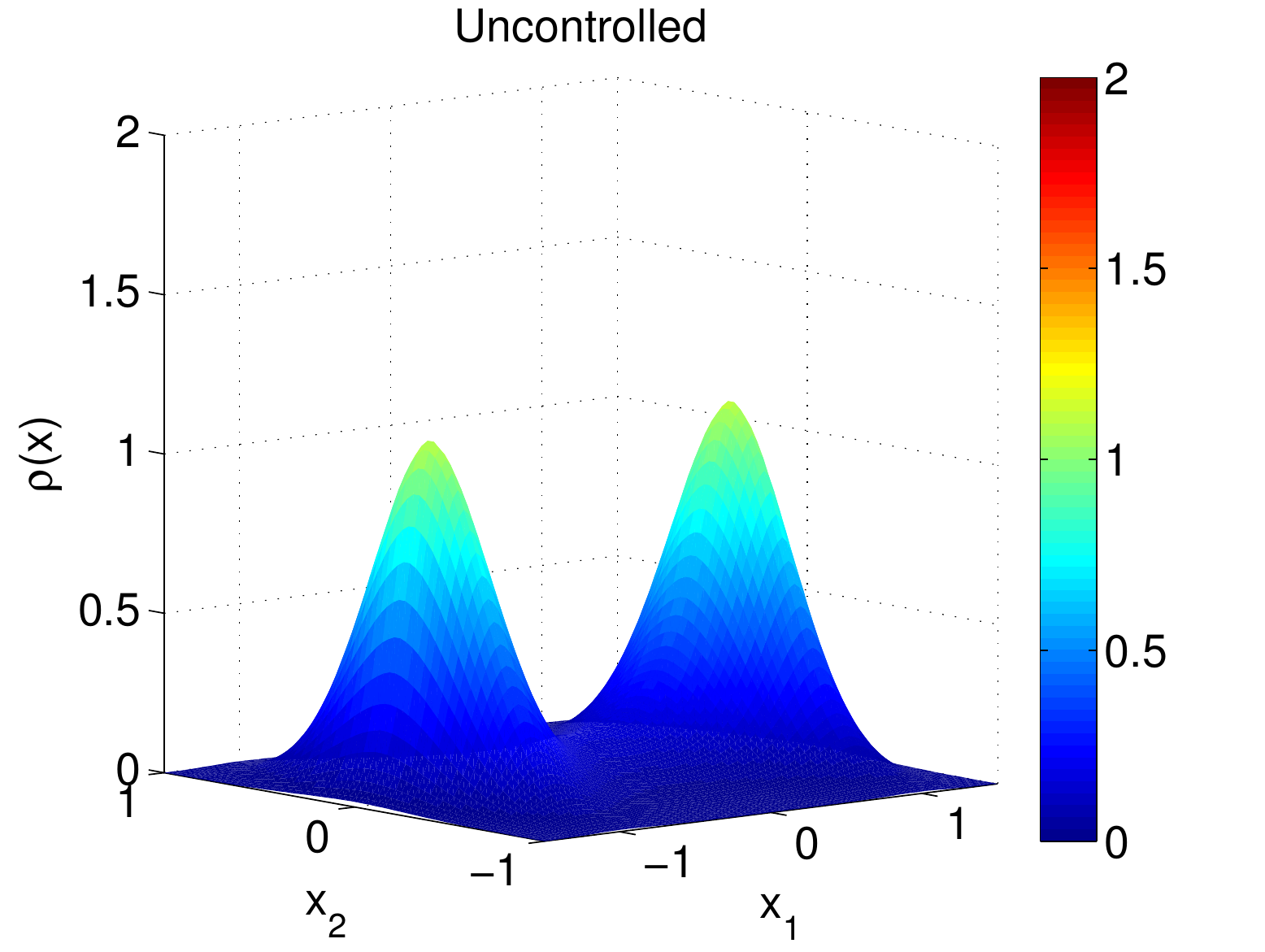}
    \caption{$t=0.15$.}
    \end{subfigure}\begin{subfigure}[c]{0.24\textwidth}
    \includegraphics[scale=0.22]{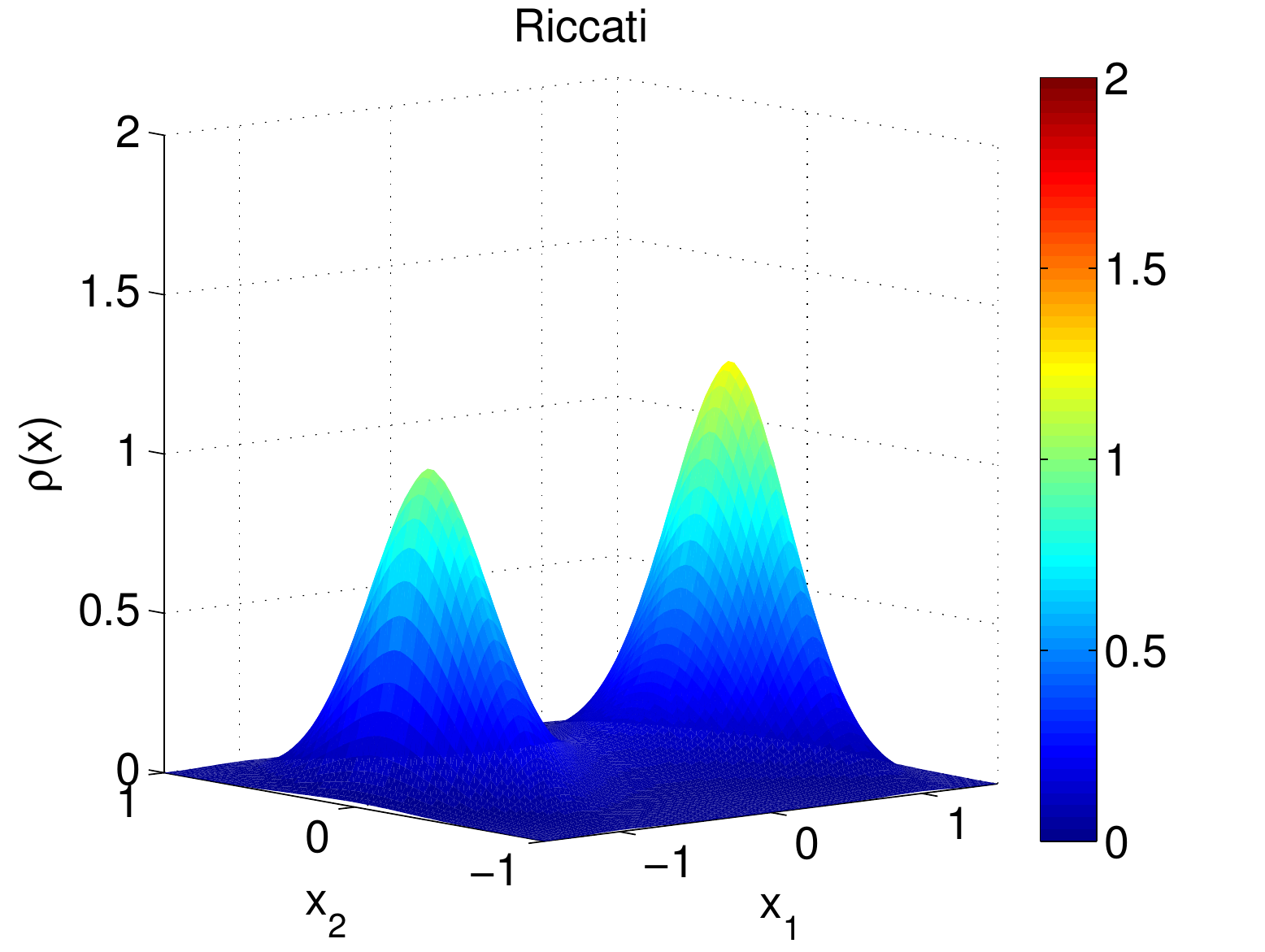}
    \caption{$t=0.15$.}
    \end{subfigure}\begin{subfigure}[c]{0.24\textwidth}
    \includegraphics[scale=0.22]{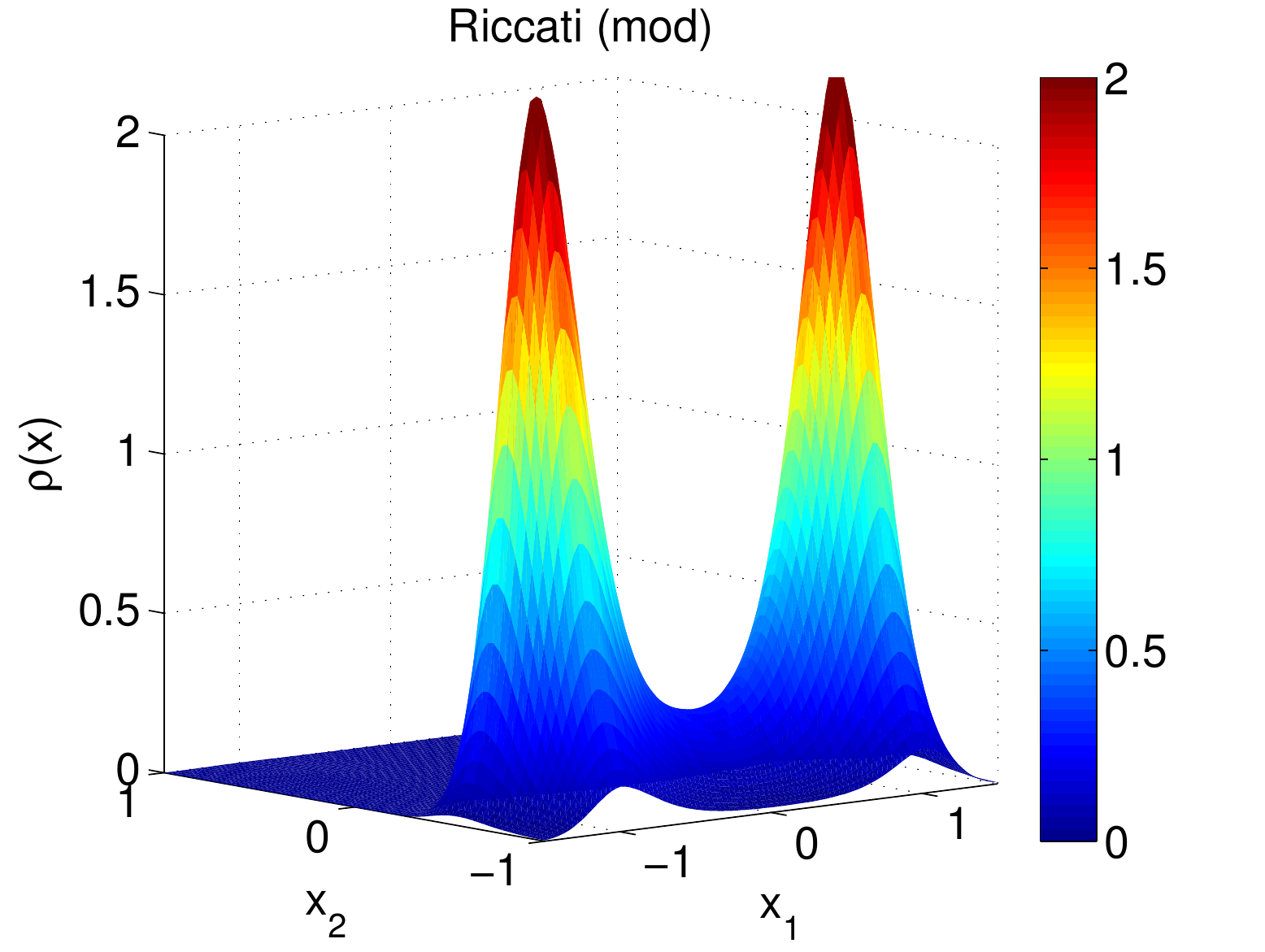}
    \caption{$t=0.15$.}
    \end{subfigure}\begin{subfigure}[c]{0.24\textwidth}
    \includegraphics[scale=0.22]{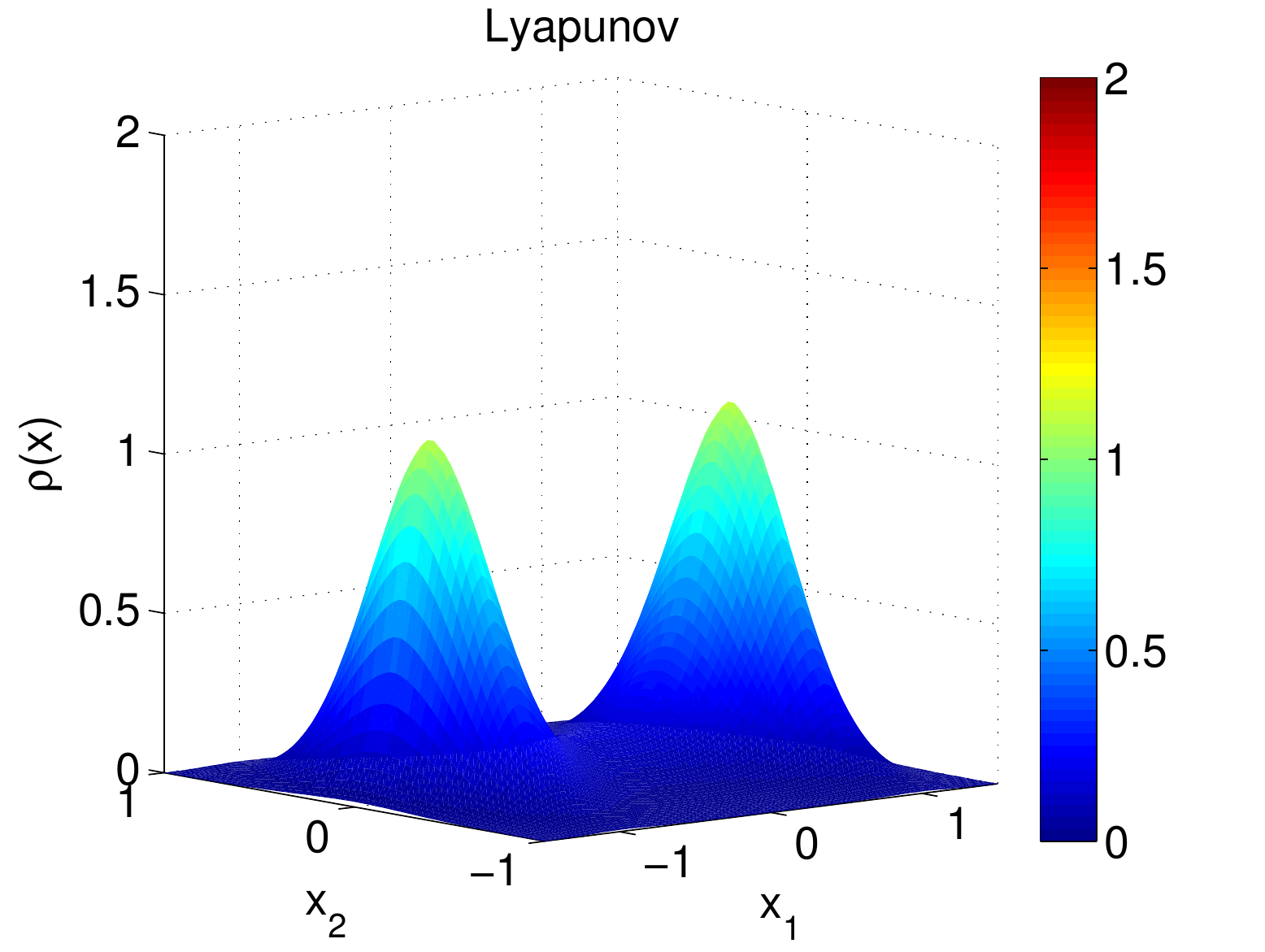}
    \caption{$t=0.15$.}
    \end{subfigure}
  \caption{Temporal evolution of the state $\rho$.}
  \label{fig:state_rand}
    \end{figure}
This phenomenon is understood better when considering  snapshots of
the solution for different time steps. In Figure \ref{fig:state_rand} the
results are shown for $t=0.01$ and $t=0.15.$ Except for the case of the rotated
 $\alpha,$ all solutions have approximately approached the stationary state at
time $t=0.15$ already. Taking into account the shape of the stationary
distribution, the shape of $\alpha$ for the Riccati-based and the
Lyapunov-based approach are intuitive. In both cases, the control allows to
lower the potential around the left well and to raise it around the right well.
Obviously, since $u$ is allowed to be positive as well as negative, this effect
can be reversed such that the right well is given preference. On the other
hand, when the shape function is subject to a rotation as done in the
experiments, both wells are equally important and no direct transition between
them is possible. This is exactly what happens in the simulation. The control
law pushes the particle first to the upper boundary before it is moved
back to the lower boundary, see Figure \ref{fig:state_rand}.
\begin{figure}[htb]
  \begin{subfigure}[c]{0.33\textwidth}
    \includegraphics[scale=0.29]{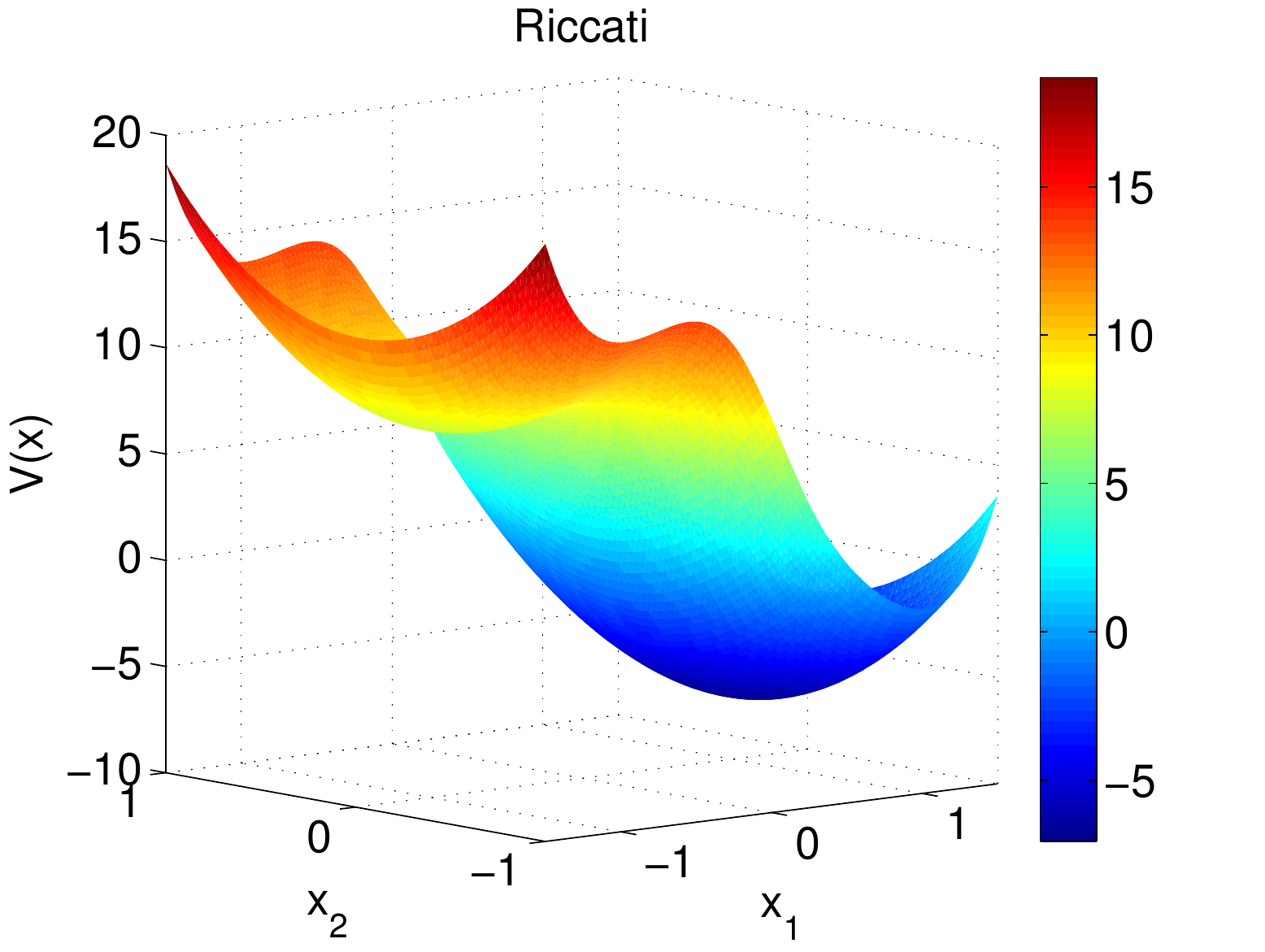}
    \caption{$t=0.01$.}
    \end{subfigure}\begin{subfigure}[c]{0.33\textwidth}
    \includegraphics[scale=0.29]{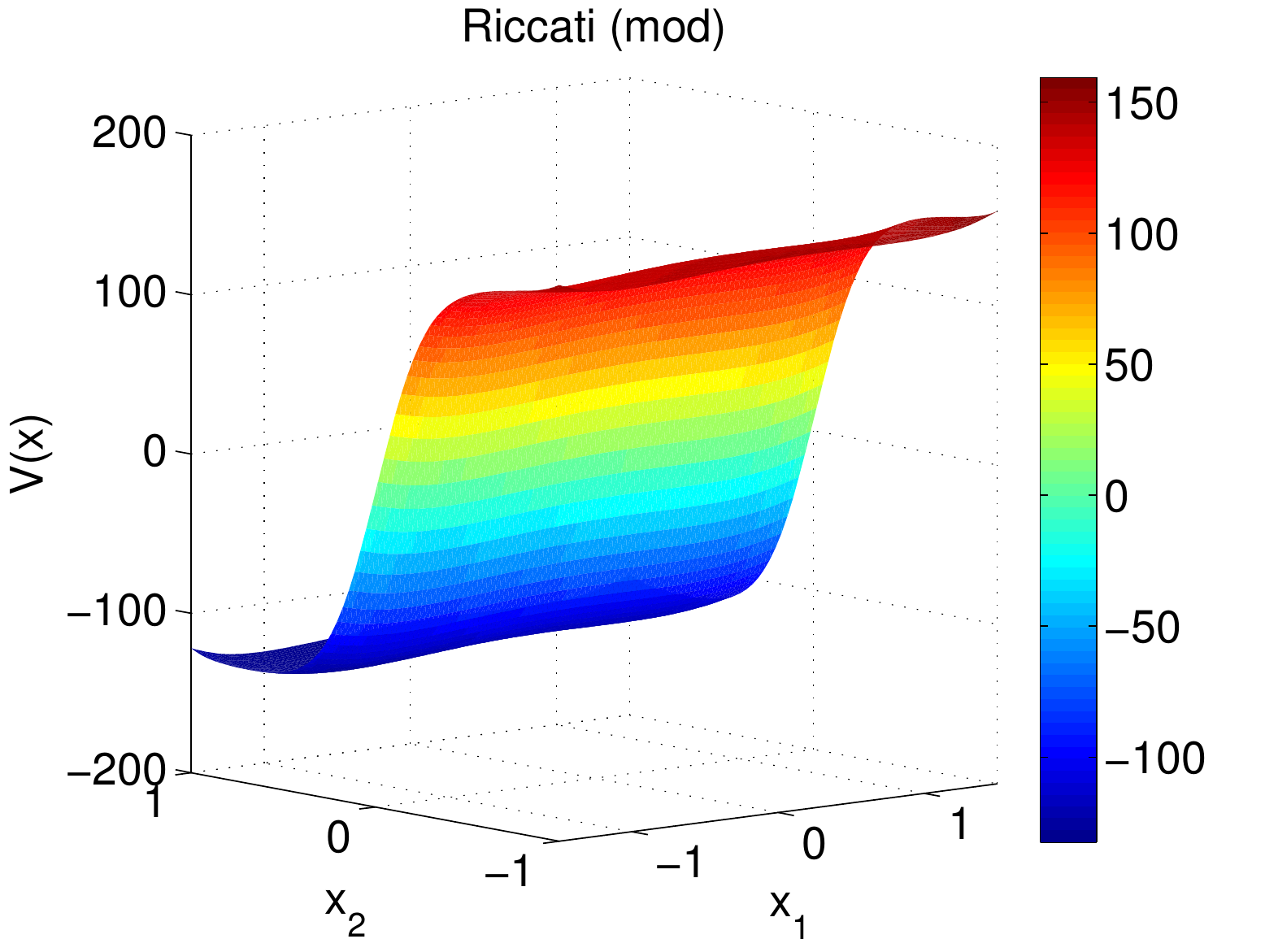}
    \caption{$t=0.01$.}
    \end{subfigure}\begin{subfigure}[c]{0.33\textwidth}
    \includegraphics[scale=0.29]{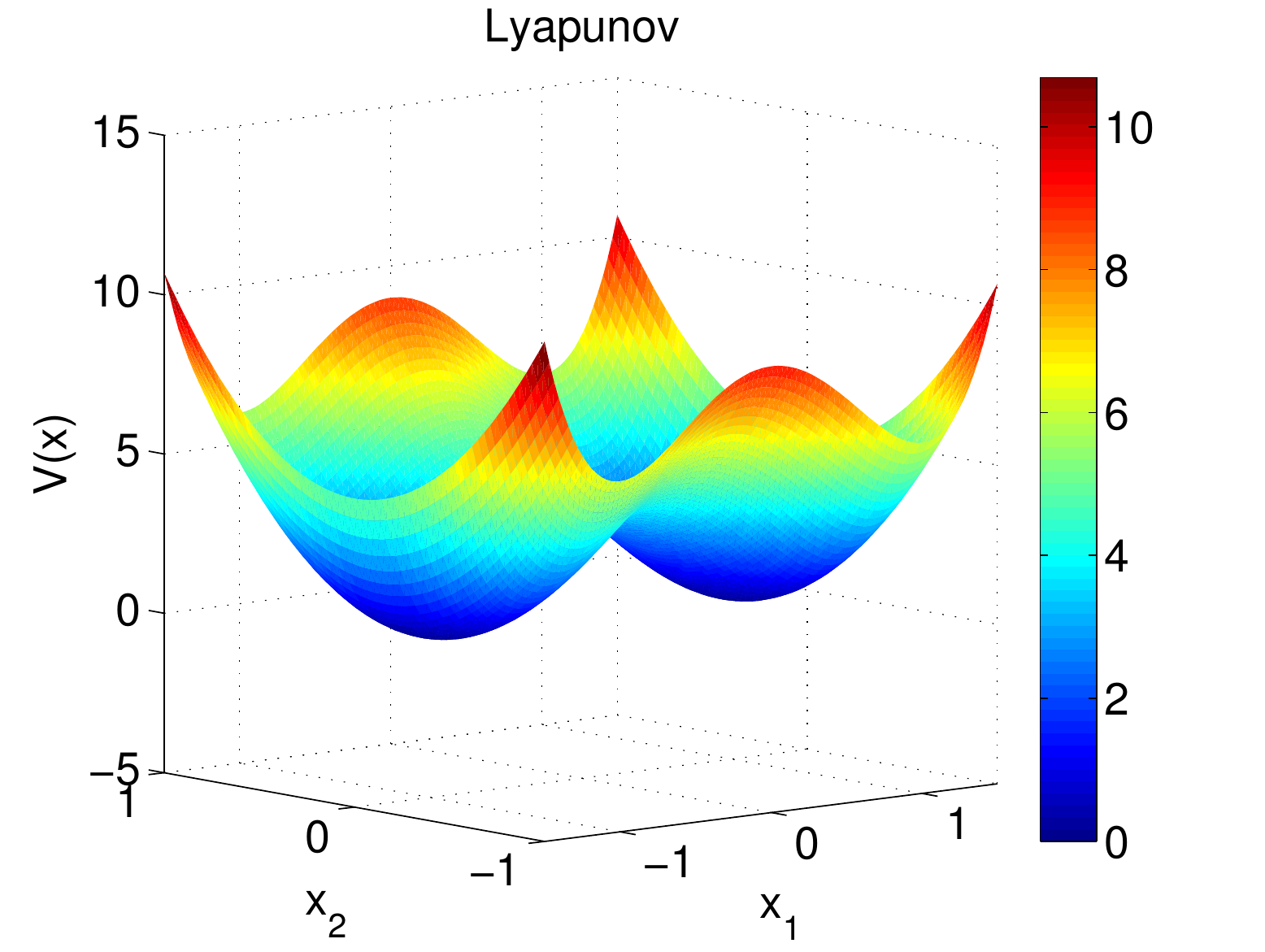}
    \caption{$t=0.01$.}
    \end{subfigure} \\ \begin{subfigure}[c]{0.33\textwidth}
    \includegraphics[scale=0.29]{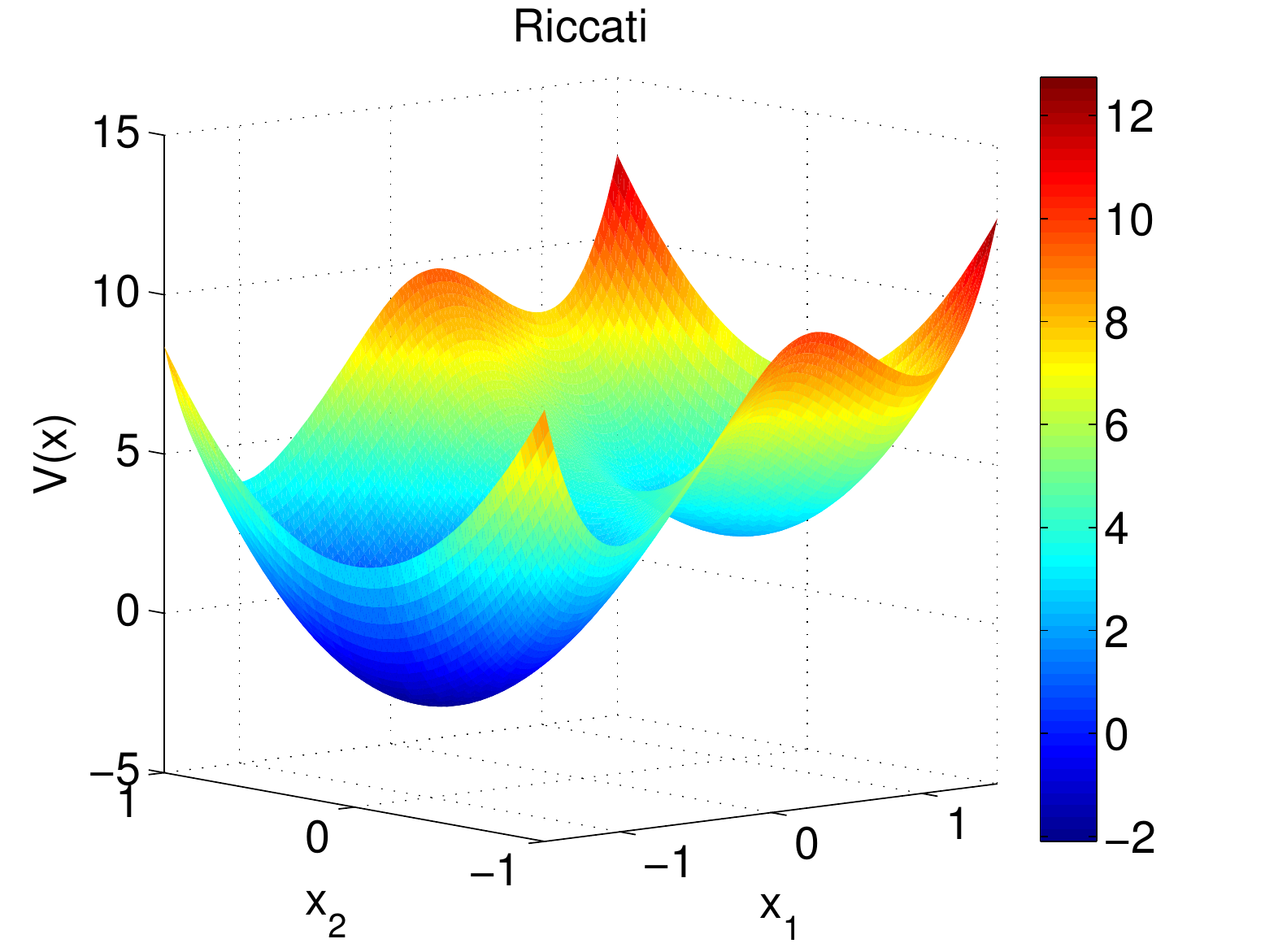}
    \caption{$t=0.15$.}
    \end{subfigure}\begin{subfigure}[c]{0.33\textwidth}
    \includegraphics[scale=0.29]{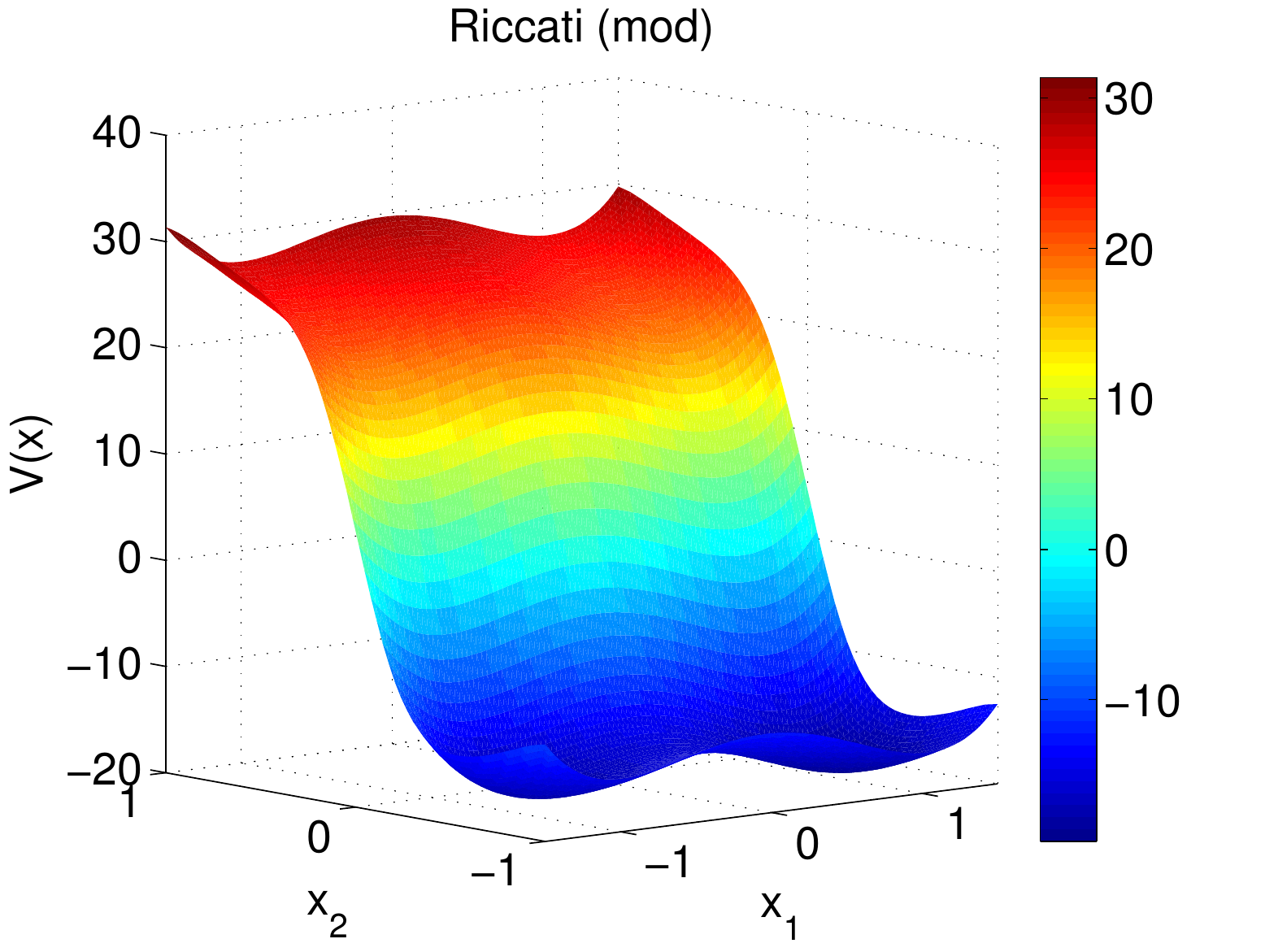}
    \caption{$t=0.15$.}
    \end{subfigure}\begin{subfigure}[c]{0.33\textwidth}
    \includegraphics[scale=0.29]{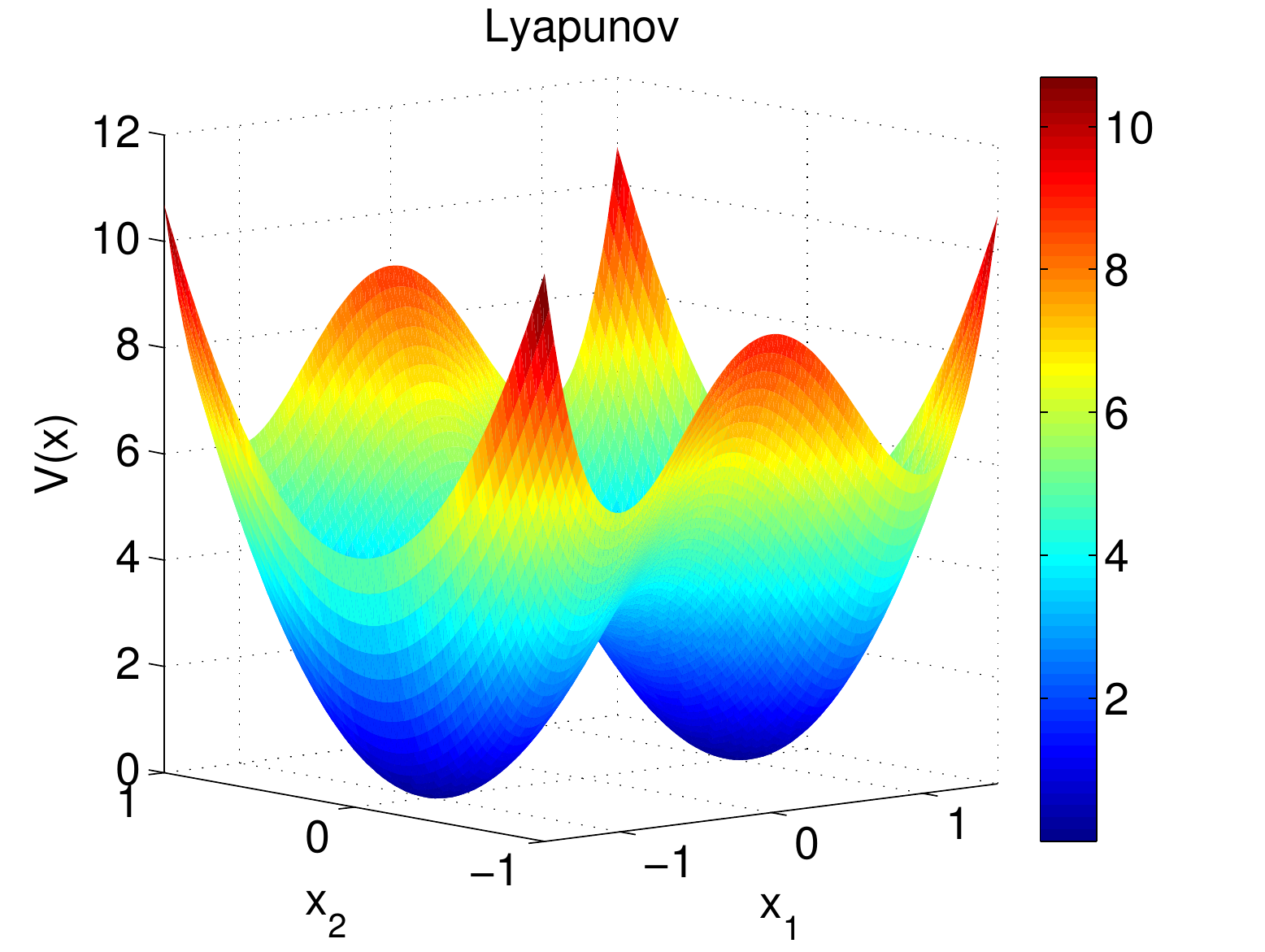}
    \caption{$t=0.15$.}
    \end{subfigure}
  \caption{Temporal evolution of the potential $V(x)$.}
  \label{fig:pot_rand}
    \end{figure}
Figure \ref{fig:pot_rand} visualizes the influence of the different control
strategies on the potential $G(x).$ Again, the effect of the modified Riccati
approach is the lowering of the potential on the bottom and top boundary
instead of the left and right boundary, respectively. It is further worthwhile
to note that the Lyapunov-based feedback law influences the potential only
moderately.

\subsection{The particle located in one well}

For the second test case, we assume the particle is initially
located in the center of the right potential well, i.e., the initial state
reflects
a numerical point mass at $x_1=1,x_2=0.$
\begin{figure}[htb]
\begin{center}
  \includegraphics{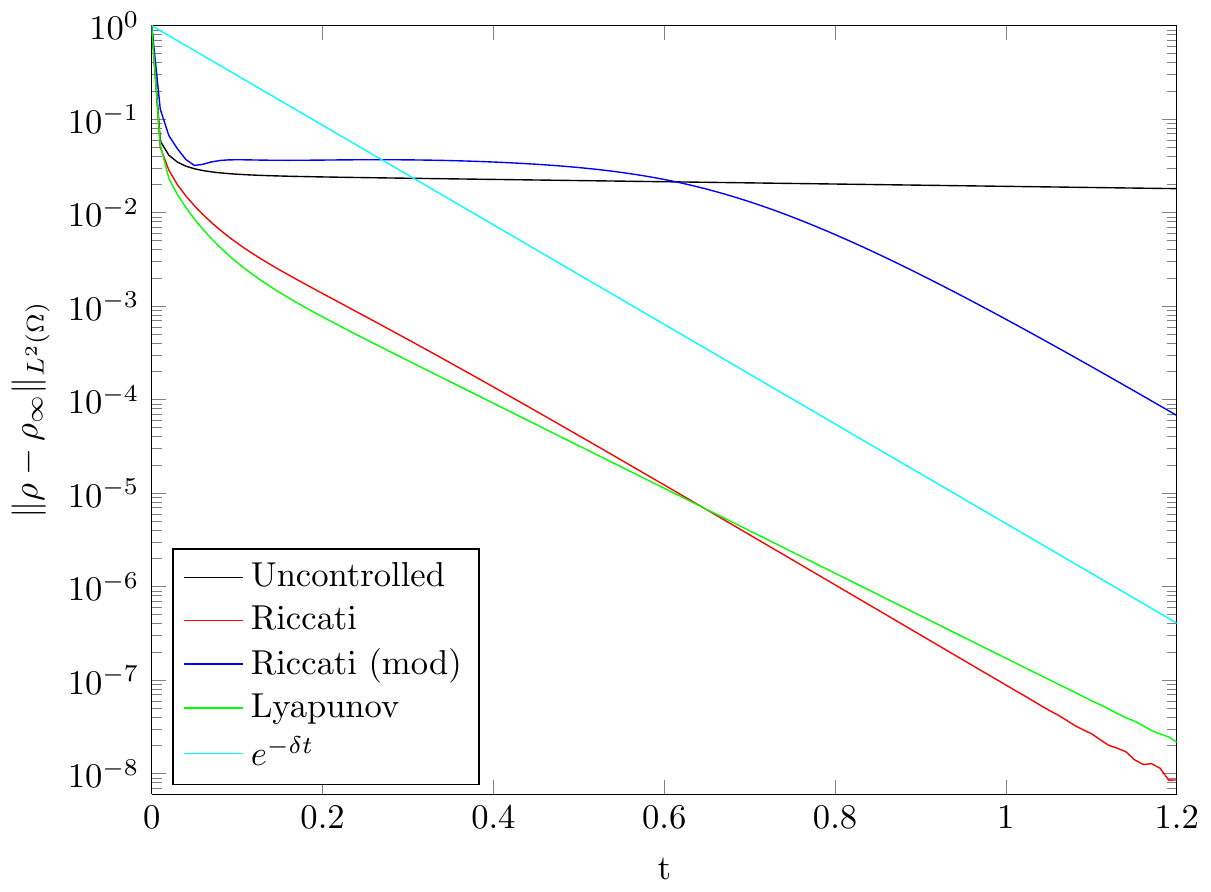}
  \caption{Comparison of $L^2(\Omega)$-norm evolution}
  \label{fig:L2_dirac}
  \end{center}
\end{figure}
As is shown in Figure \ref{fig:L2_dirac}, in this case the convergence rate of
the uncontrolled system is undesirably slow. We already mentioned that this is
mainly reflected by the fact that the particle has to overcome the
``energy barrier'' between the potential wells. Here, the feedback laws act by
lowering this
barrier, hence allowing the particle to ``jump'' into the left potential well.
\begin{figure}[htb]
  \begin{subfigure}[c]{0.24\textwidth}
    \includegraphics[scale=0.22]{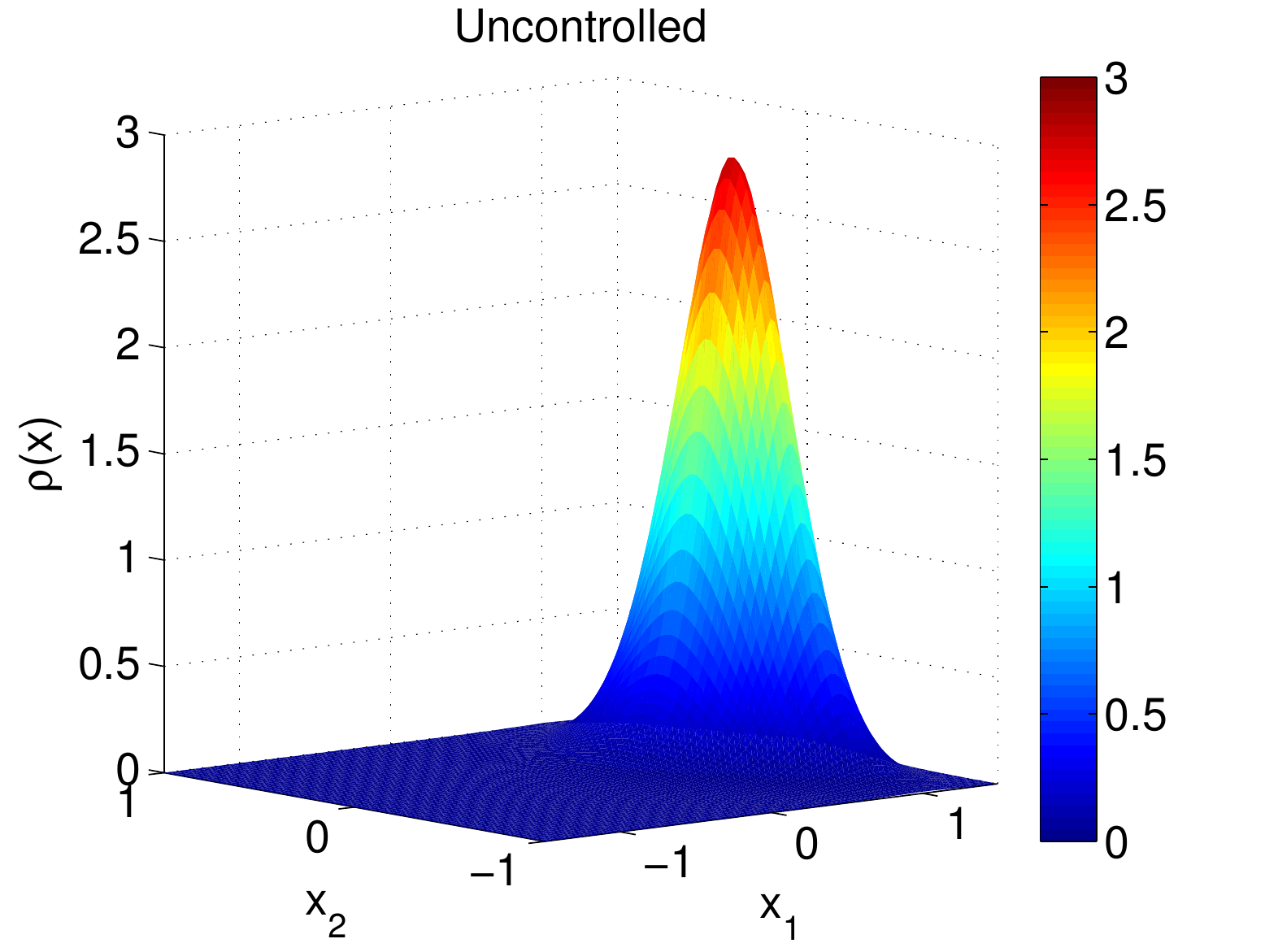}
    \caption{$t=0.1$.}
    \end{subfigure}\begin{subfigure}[c]{0.24\textwidth}
    \includegraphics[scale=0.22]{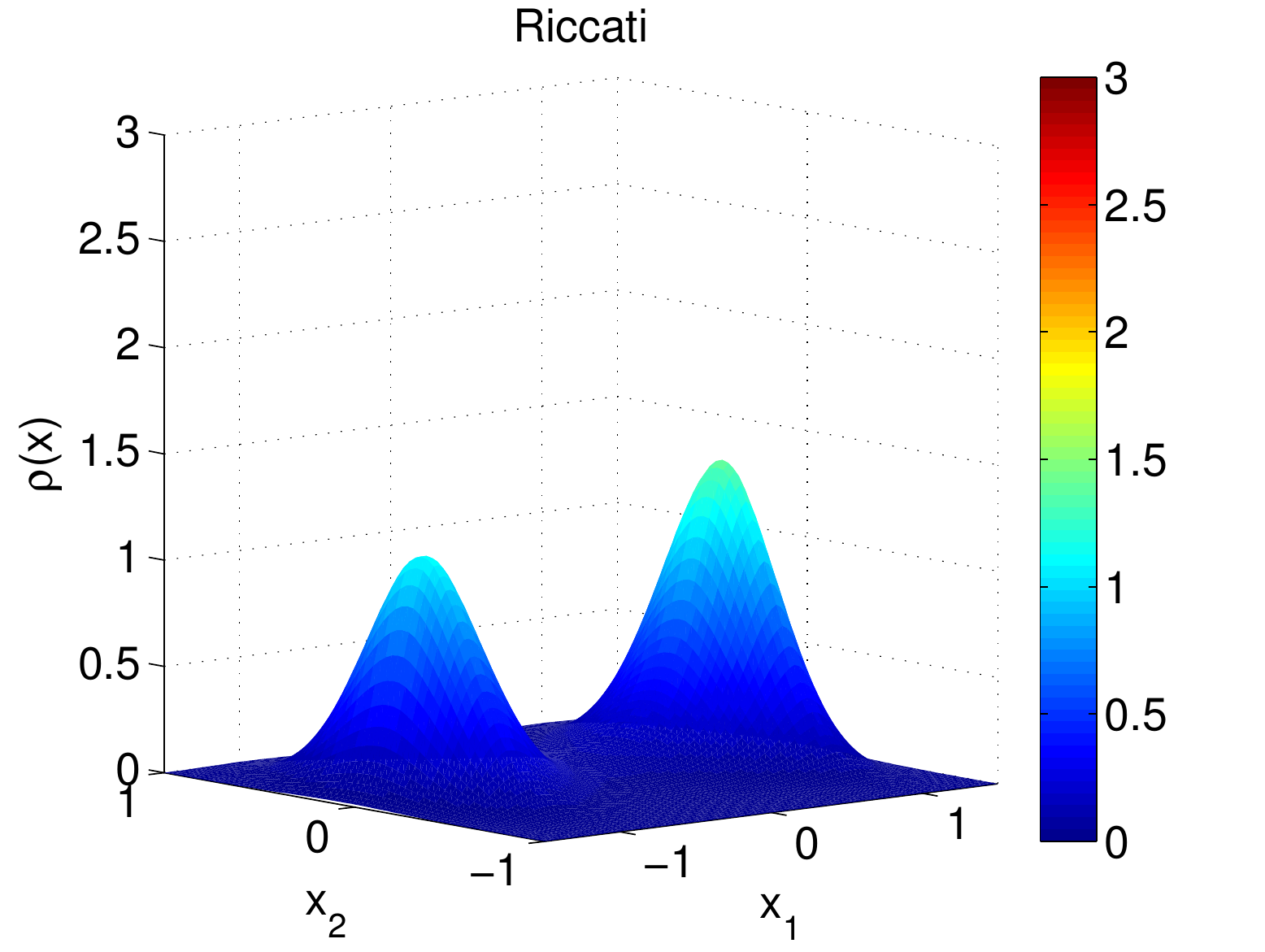}
    \caption{$t=0.1$.}
    \end{subfigure}\begin{subfigure}[c]{0.24\textwidth}
    \includegraphics[scale=0.22]{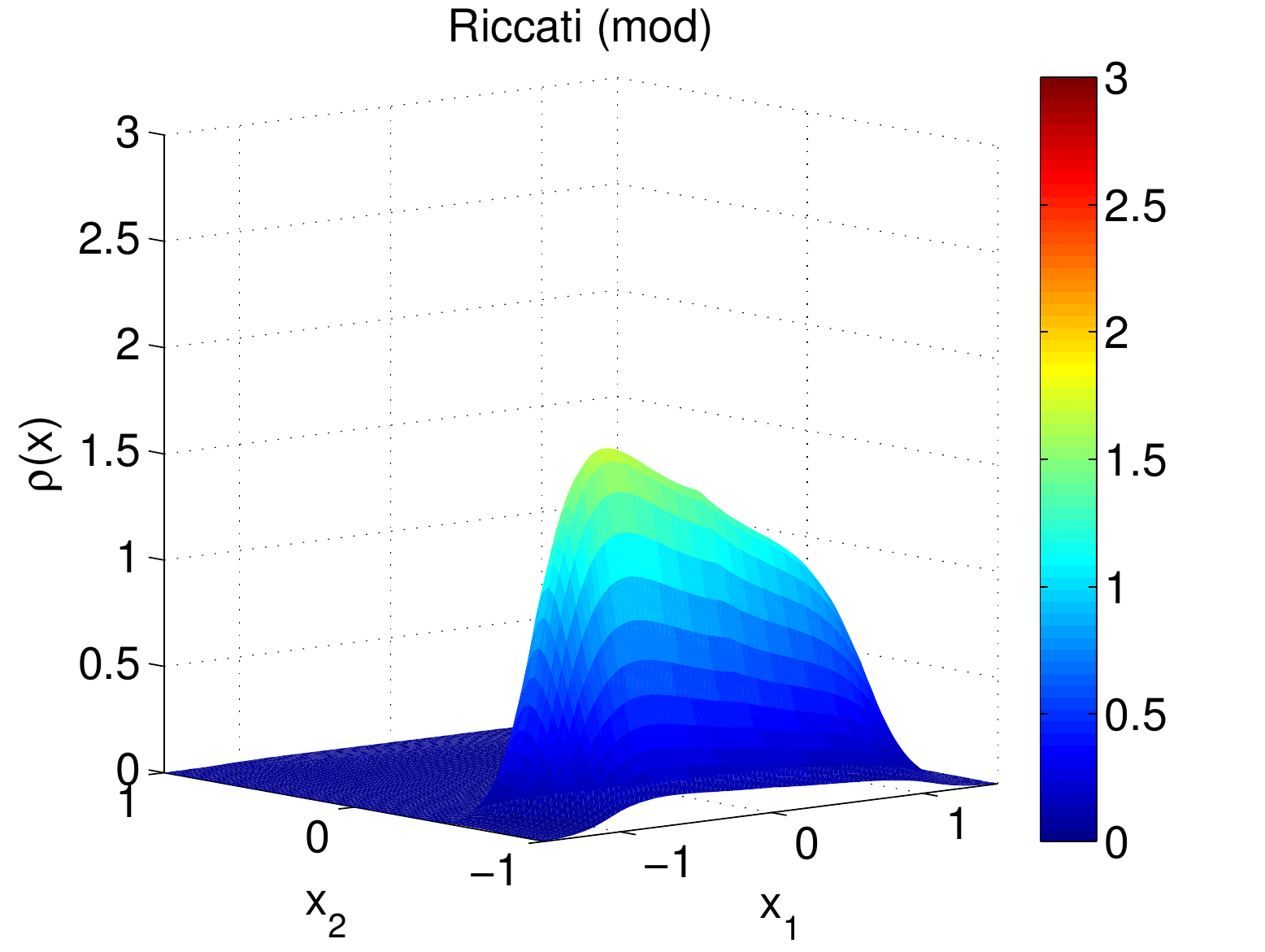}
    \caption{$t=0.1$.}
    \end{subfigure}\begin{subfigure}[c]{0.24\textwidth}
    \includegraphics[scale=0.22]{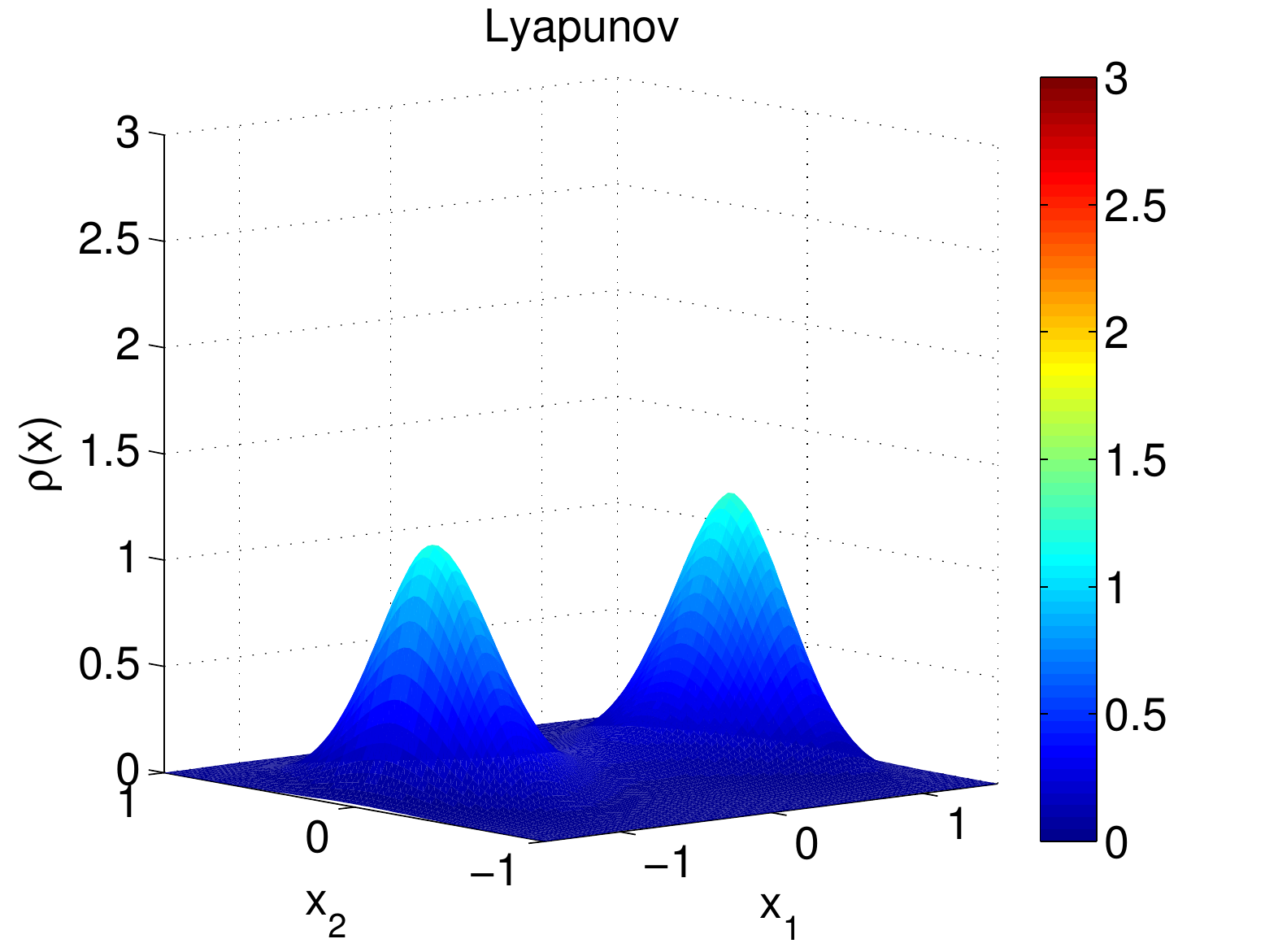}
    \caption{$t=0.1$.}
    \end{subfigure}\\ \begin{subfigure}[c]{0.24\textwidth}
    \includegraphics[scale=0.22]{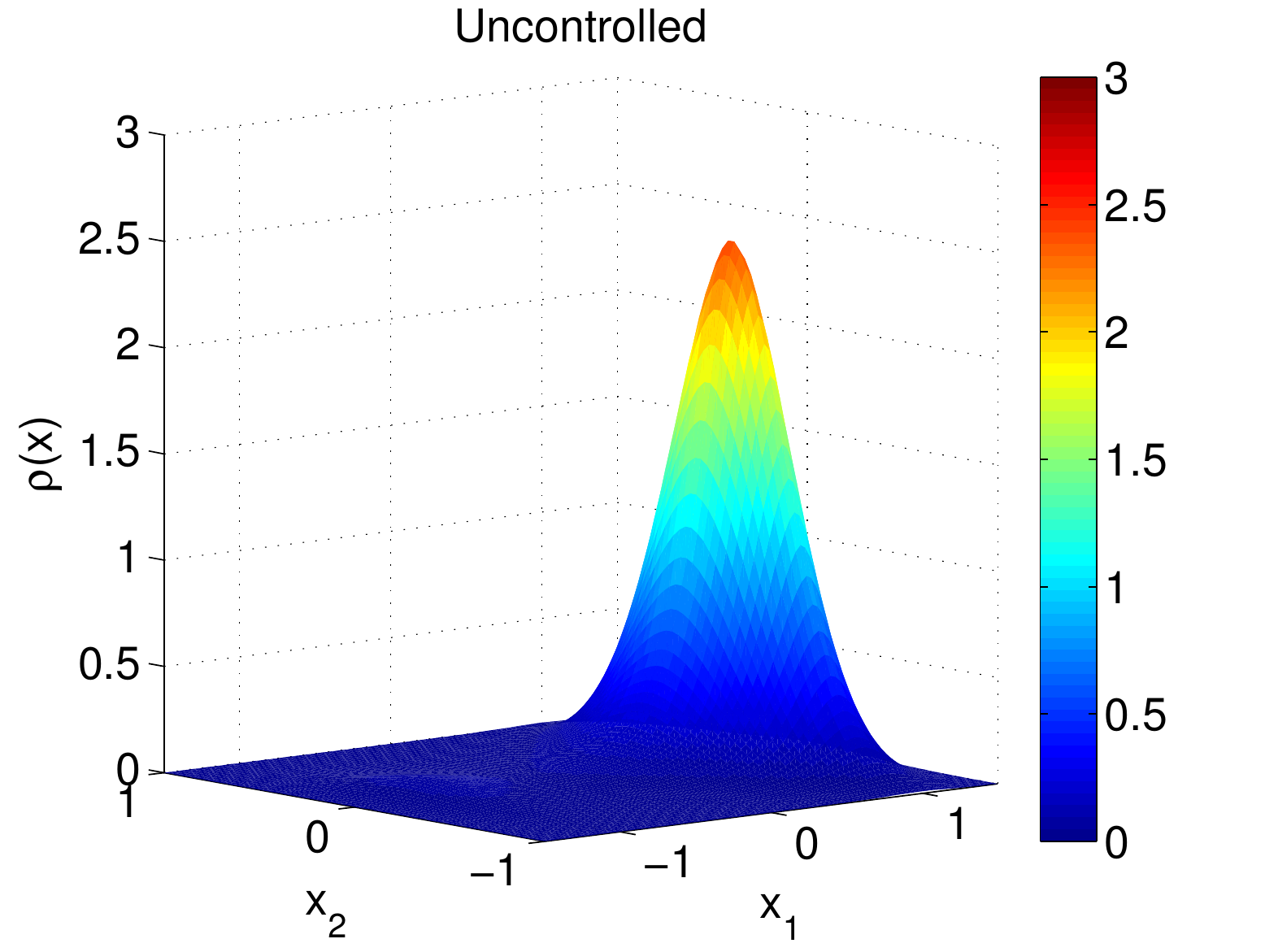}
    \caption{$t=0.5$.}
    \end{subfigure}\begin{subfigure}[c]{0.24\textwidth}
    \includegraphics[scale=0.22]{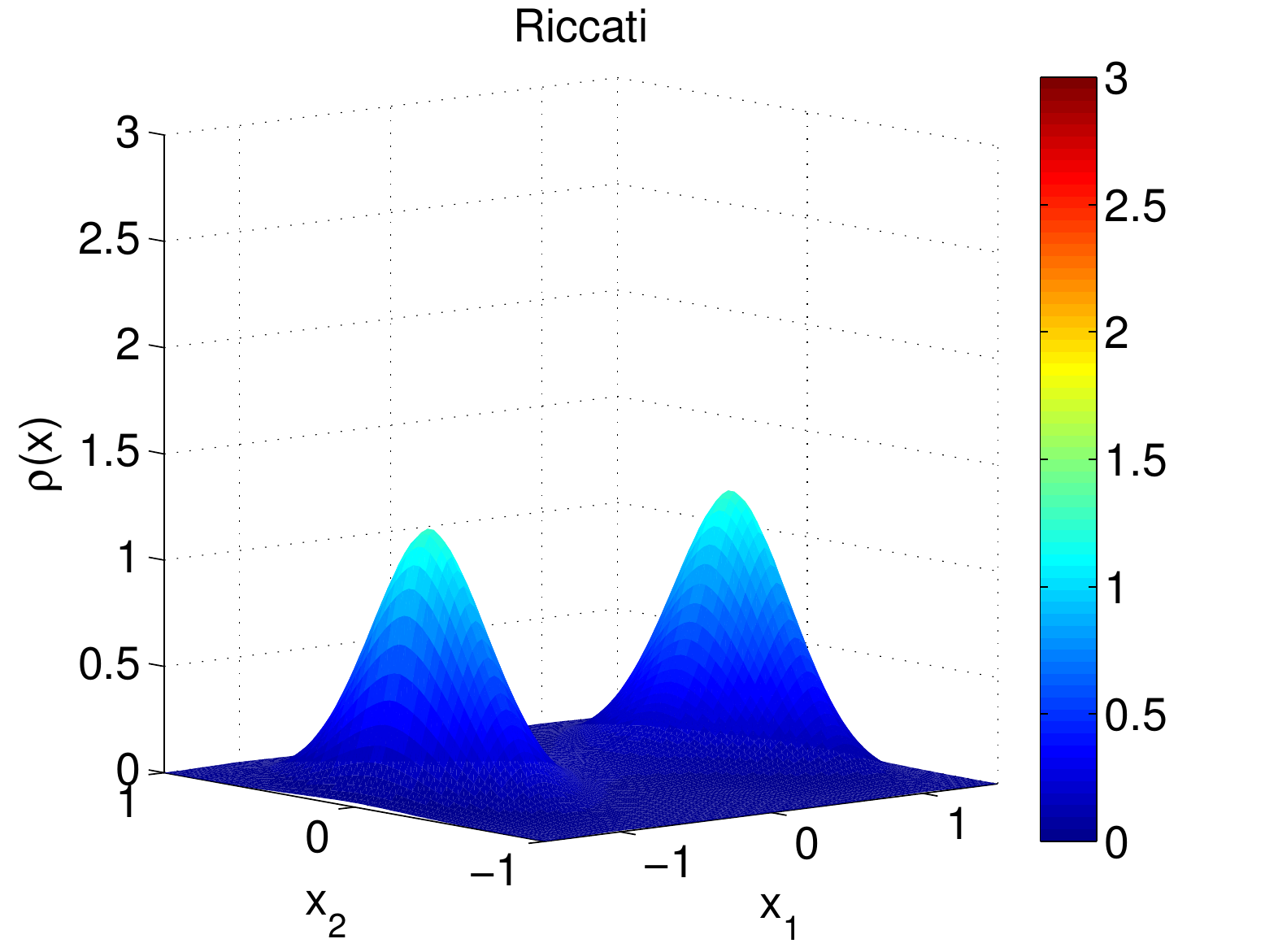}
    \caption{$t=0.5$.}
    \end{subfigure}\begin{subfigure}[c]{0.24\textwidth}
    \includegraphics[scale=0.22]{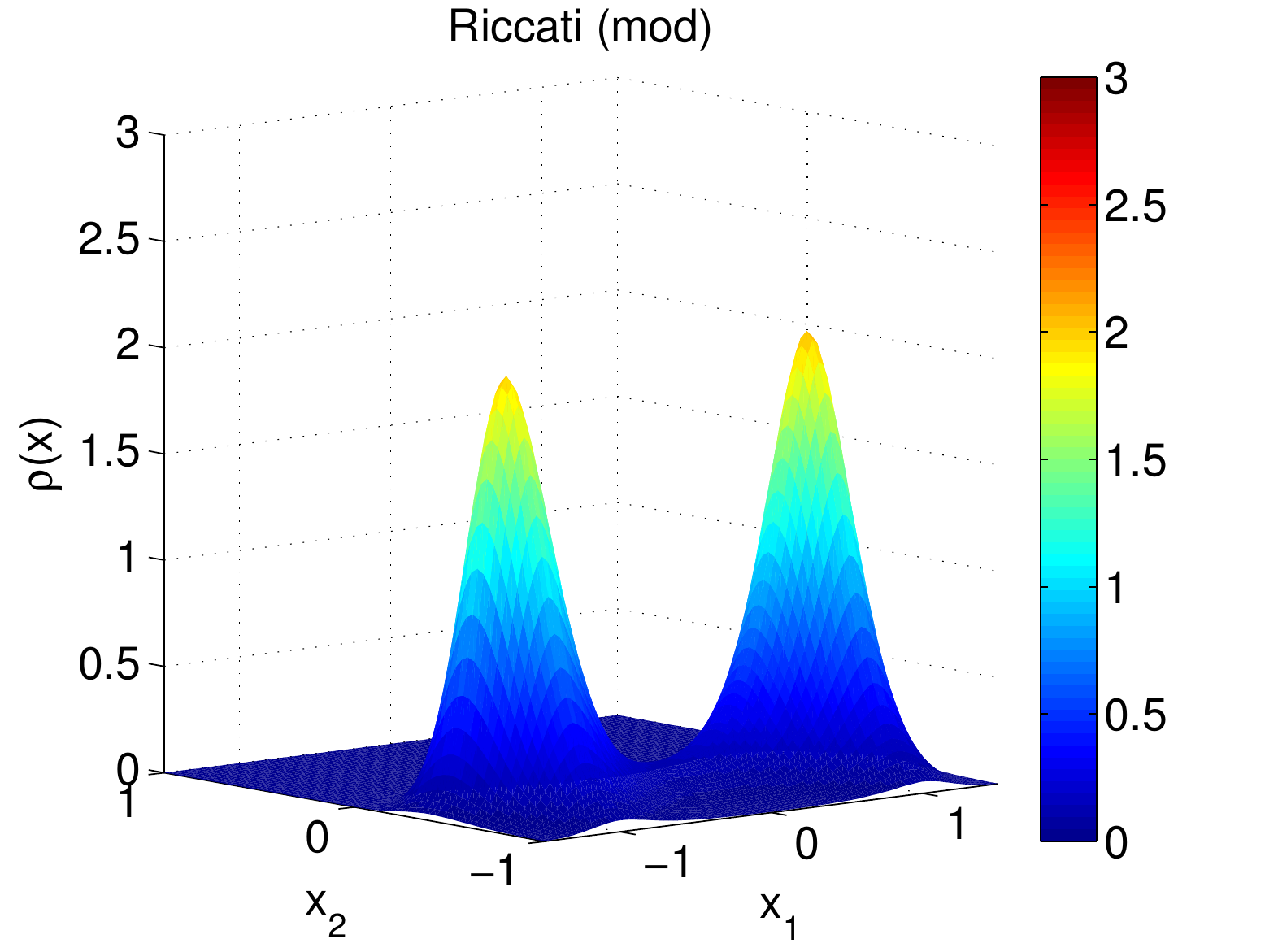}
    \caption{$t=0.5$.}
    \end{subfigure}\begin{subfigure}[c]{0.24\textwidth}
    \includegraphics[scale=0.22]{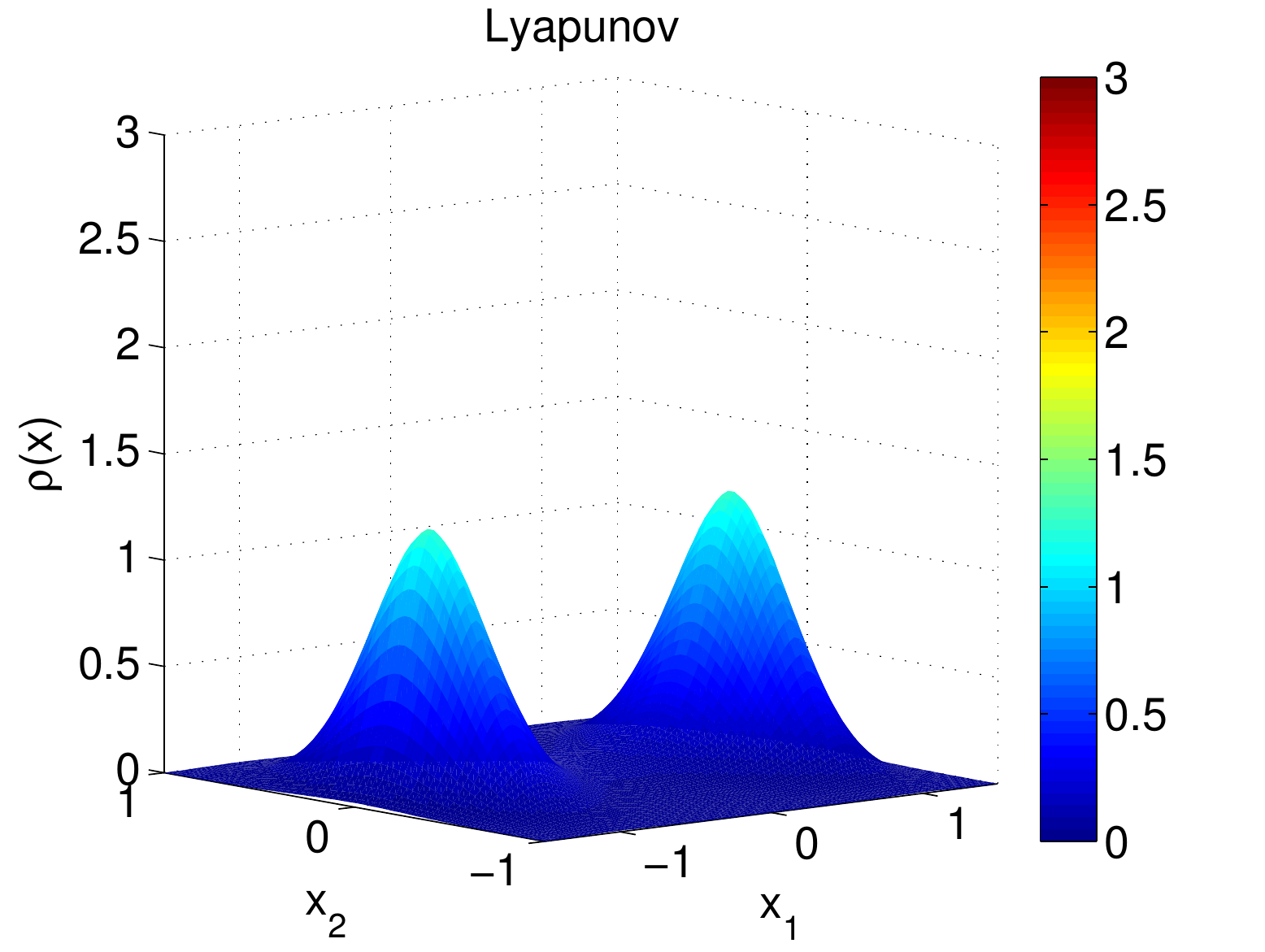}
    \caption{$t=0.5$.}
    \end{subfigure}
  \caption{Temporal evolution of the state $\rho$.}
  \label{fig:state_dirac}
    \end{figure}
As in the previous case, Figure \ref{fig:state_dirac} and Figure
\ref{fig:potential_dirac} show the temporal evolution of the state of the
systems as well as the influence on the potential.
\begin{figure}[htb]
  \begin{subfigure}[c]{0.33\textwidth}
    \includegraphics[scale=0.29]{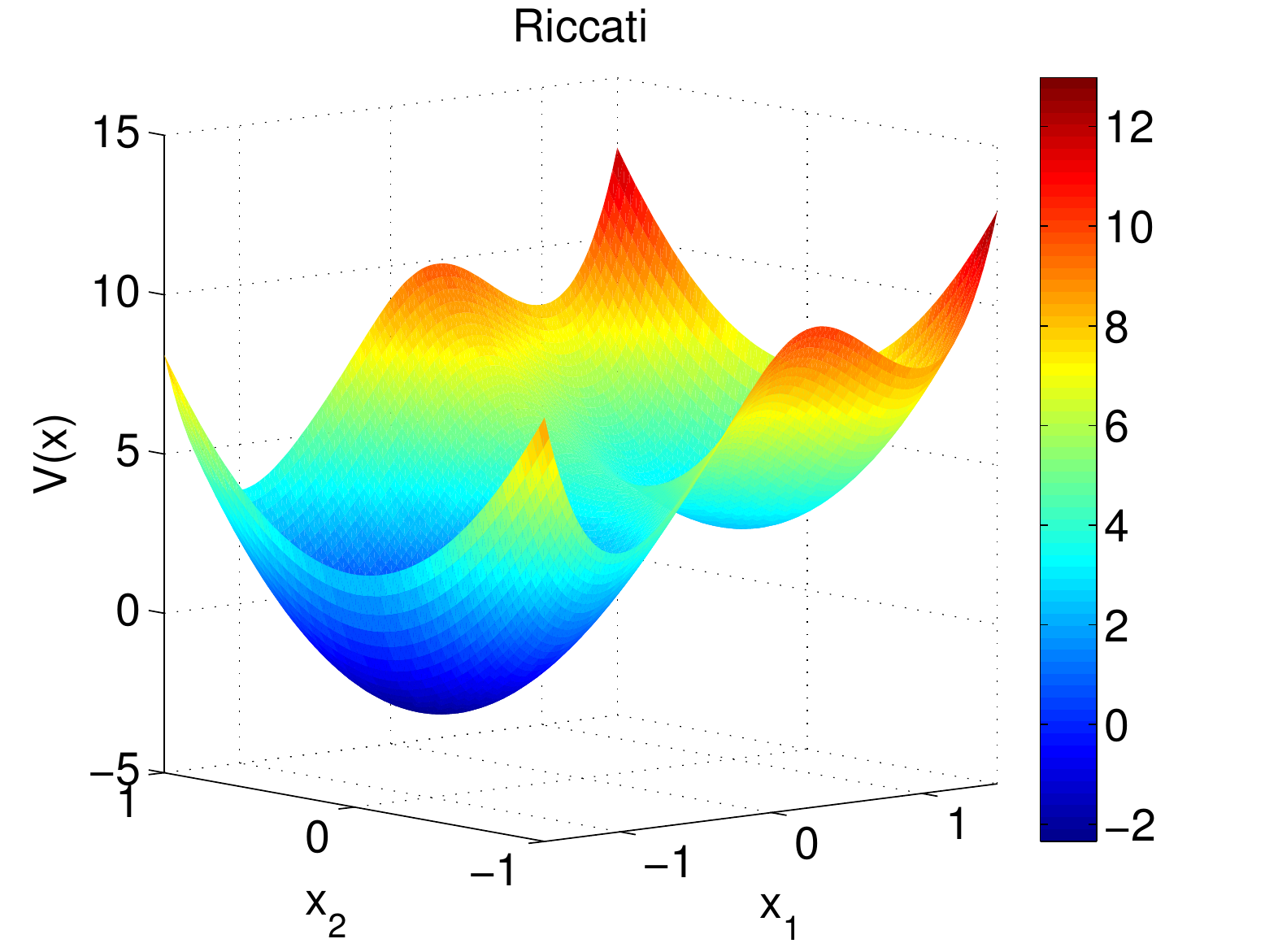}
    \caption{$t=0.1$.}
    \end{subfigure}\begin{subfigure}[c]{0.33\textwidth}
    \includegraphics[scale=0.29]{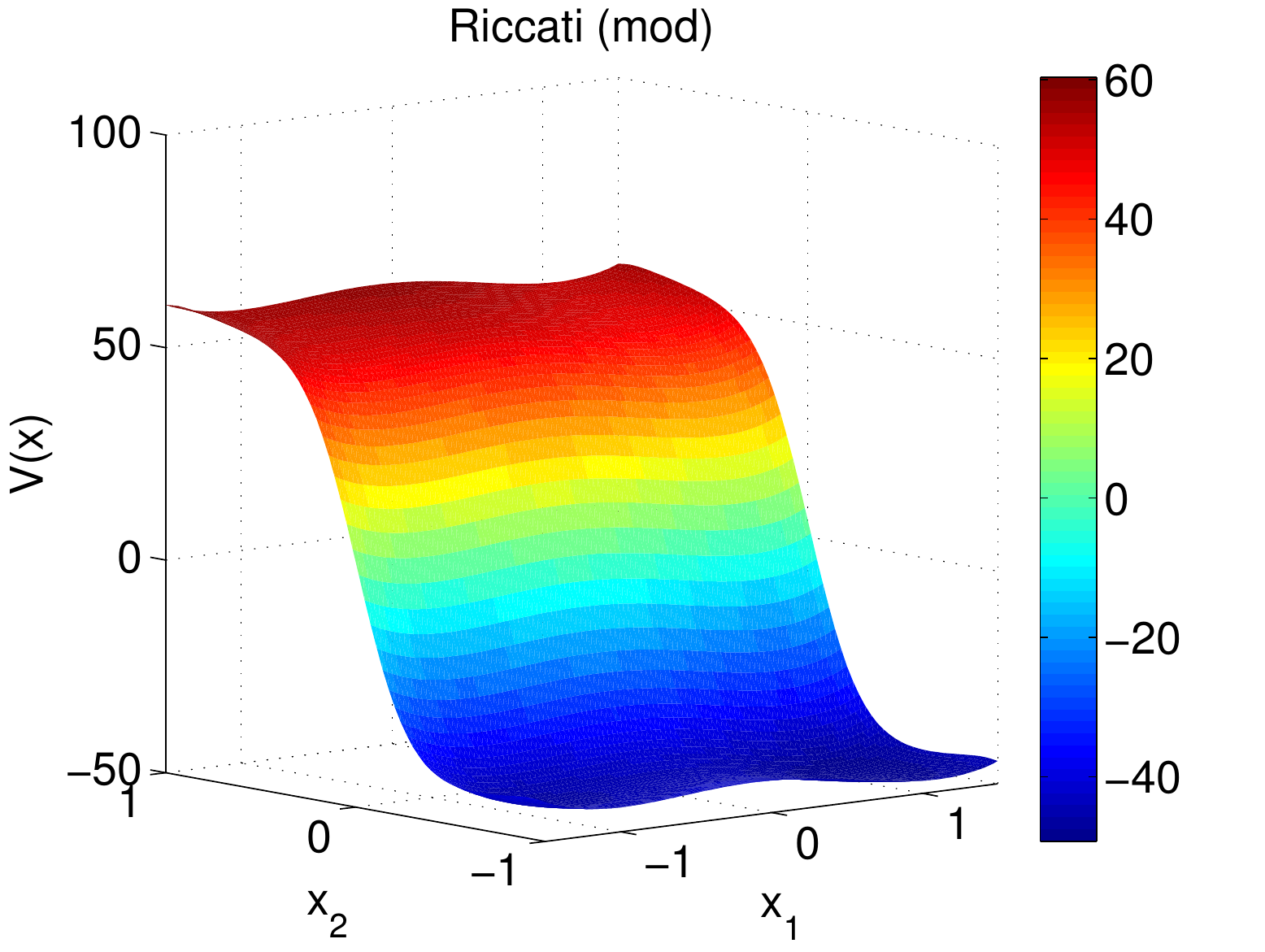}
    \caption{$t=0.1$.}
    \end{subfigure}\begin{subfigure}[c]{0.33\textwidth}
    \includegraphics[scale=0.29]{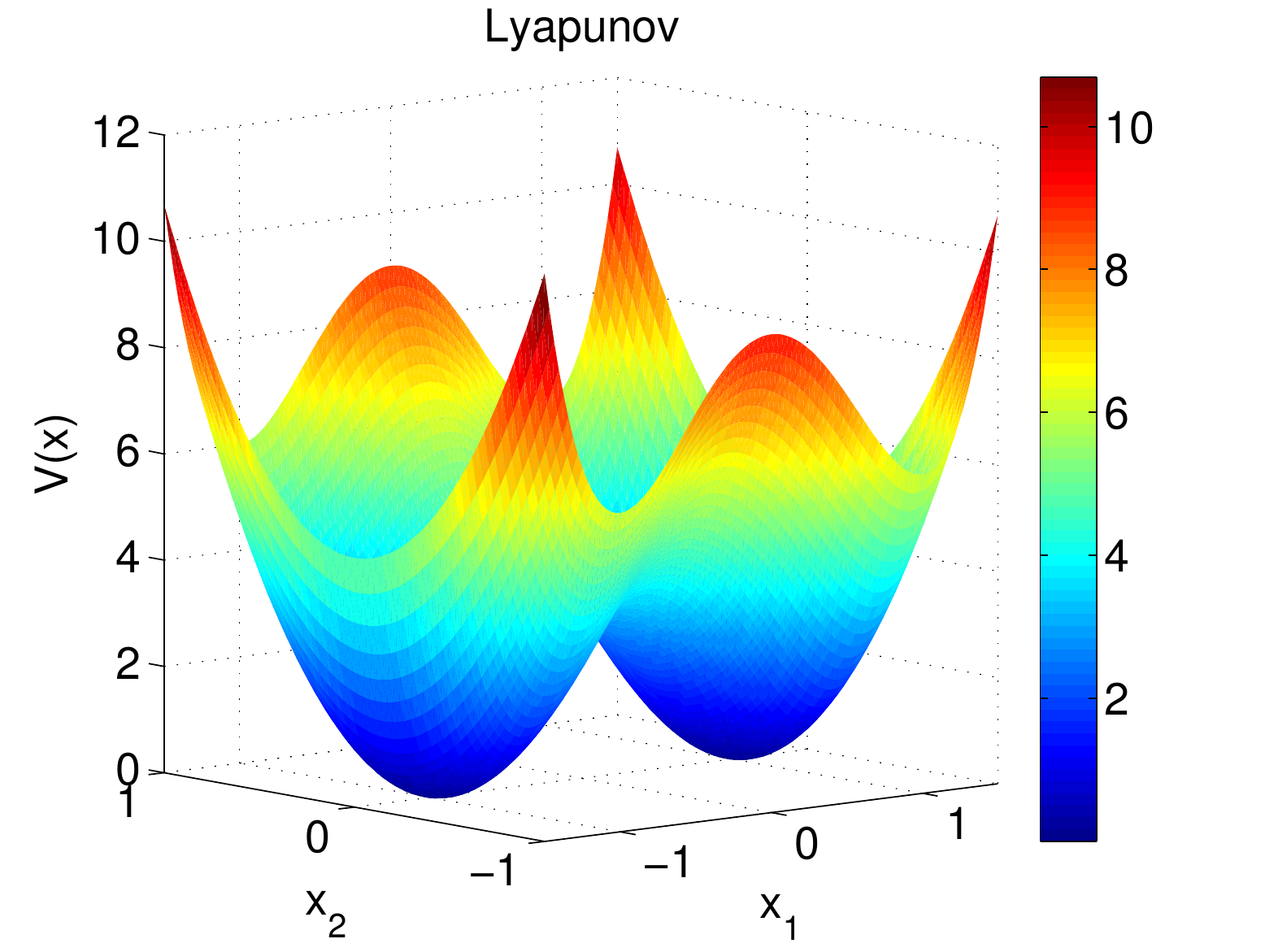}
    \caption{$t=0.1$.}
    \end{subfigure} \\ \begin{subfigure}[c]{0.33\textwidth}
    \includegraphics[scale=0.29]{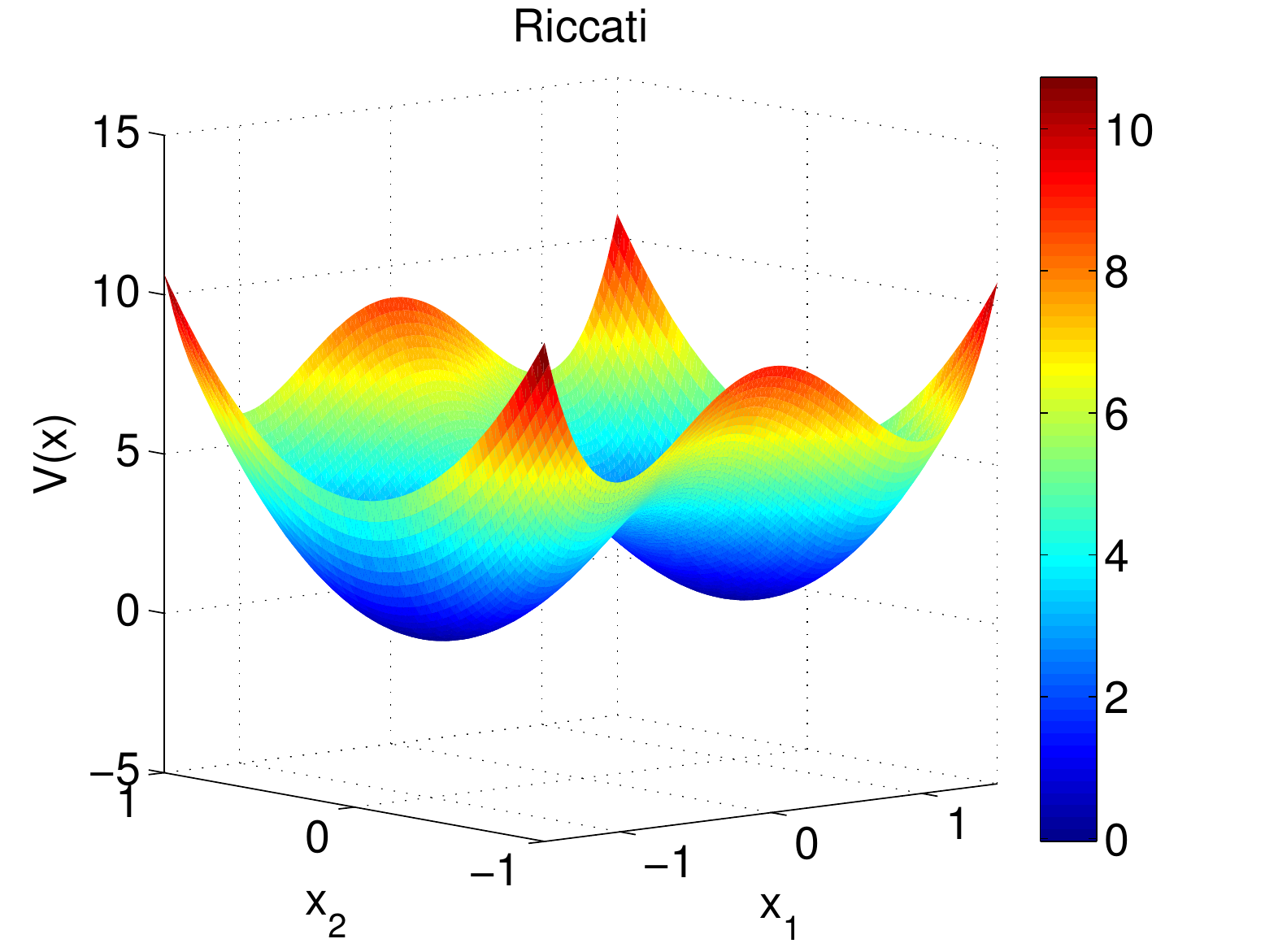}
    \caption{$t=0.5$.}
    \end{subfigure}\begin{subfigure}[c]{0.33\textwidth}
    \includegraphics[scale=0.29]{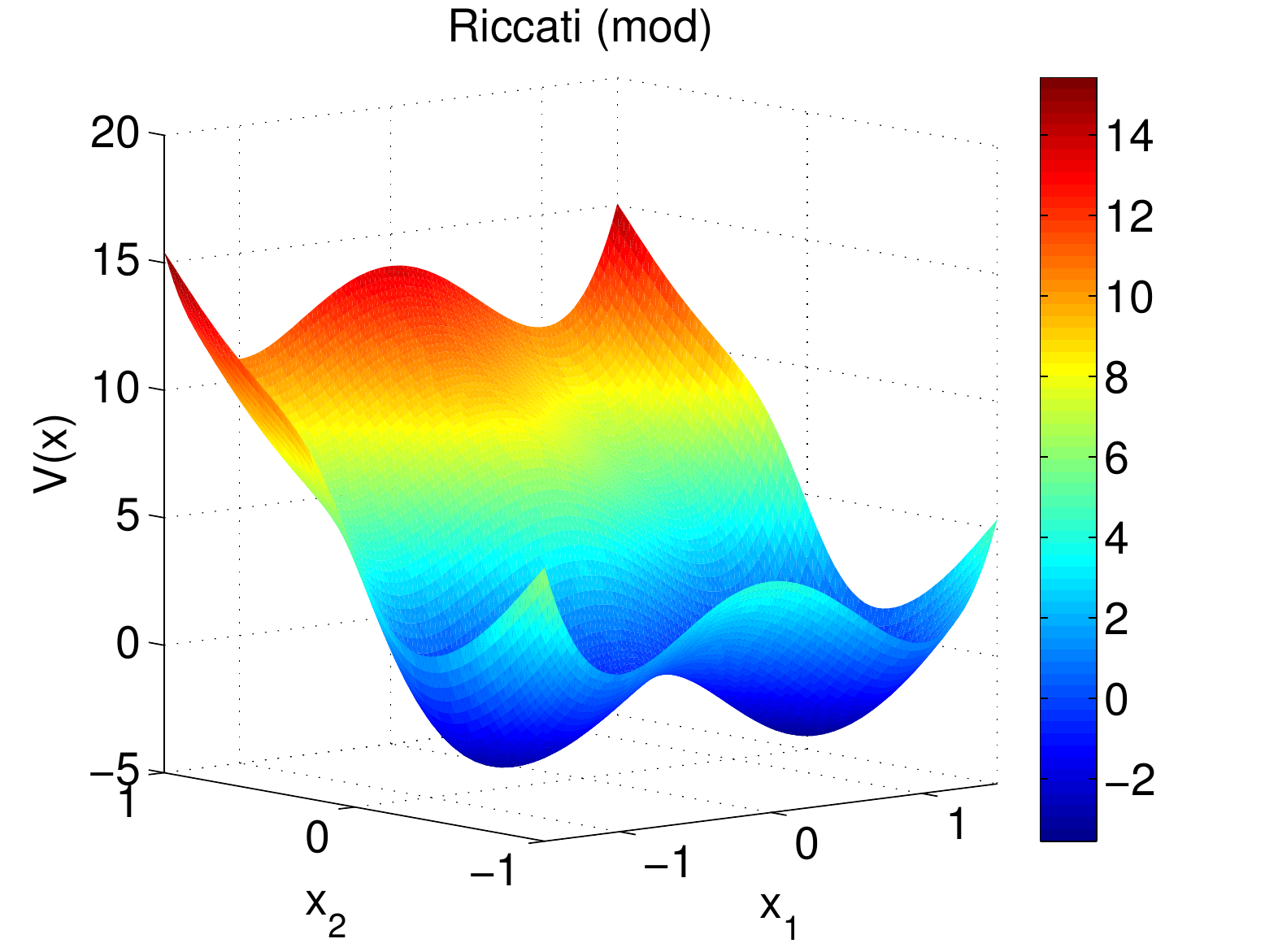}
    \caption{$t=0.5$.}
    \end{subfigure}\begin{subfigure}[c]{0.33\textwidth}
    \includegraphics[scale=0.29]{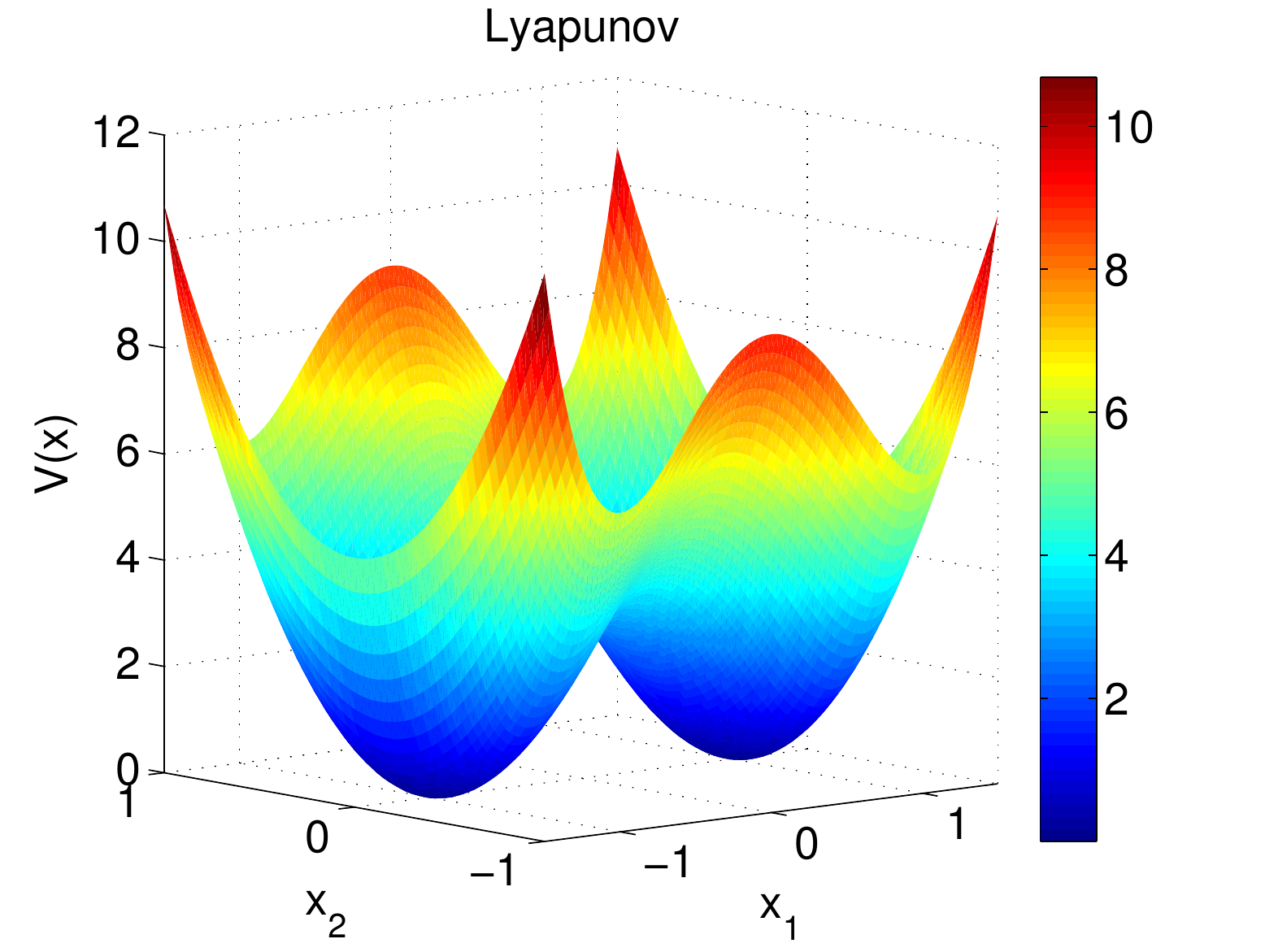}
    \caption{$t=0.5$.}
    \end{subfigure}
  \caption{Temporal evolution of the potential $V(x)$.}
  \label{fig:potential_dirac}
    \end{figure}
    Again, the modified Riccati approach acts on the dynamics by first
attracting the particle at the lower boundary from where it is slowly moved to
the center of the wells.

\section*{Acknowledgements}

This work was supported in part by the ERC advanced grant 668998 (OCLOC) under
the EU's H2020
research program.

\bibliographystyle{siam}

\end{document}